\documentclass[a4paper,oneside,10pt]{article}%
\usepackage{amsmath}
\usepackage{amsfonts}
\usepackage{amssymb}
\usepackage{graphicx}
\usepackage{bbm}
\usepackage[square,numbers,sort&compress]{natbib}%
\usepackage{hyperref}

\providecommand{\U}[1]{\protect \rule{.1in}{.1in}}
\newtheorem{theorem}{Theorem}[section]

\newtheorem{corollary}[theorem]{Corollary}

\newtheorem{definition}[theorem]{Definition}

\newtheorem{example}[theorem]{Example}

\newtheorem{lemma}[theorem]{Lemma}

\newtheorem{proposition}[theorem]{Proposition}
\newtheorem{remark}[theorem]{Remark}

\newenvironment{proof}[1][Proof]{\noindent \textbf{#1.} }{\  \rule{0.5em}{0.5em}}
\numberwithin{equation}{section}
\begin{document}

\title{A universal robust limit theorem for nonlinear L\'evy processes under
sublinear expectation\thanks{We thank two referees for their constructive and
helpful comments, which help to improve the presentation.}}
\author{Mingshang Hu\thanks{Zhongtai Securities Institute for Financial Studies,
Shandong University, Jinan, Shandong 250100, China. humingshang@sdu.edu.cn. Hu
is supported by the National Key R\&D Program of China (No. 2018YFA0703900),
the National Natural Science Foundation of China (No. 11671231), and the Qilu
Young Scholars Program of Shandong University. }
\and Lianzi Jiang\thanks{College of Mathematics and Systems Science, Shandong
University of Science and Technology, Qingdao, Shandong 266590, China, and
Zhongtai Securities Institute for Financial Studies, Shandong University,
Jinan, Shandong 250100, China. jianglianzi95@163.com.}
\and Gechun Liang\thanks{Department of Statistics, The University of Warwick,
Coventry CV4 7AL, U.K. g.liang@warwick.ac.uk.}
\and Shige Peng\thanks{School of Mathematics, Shandong University, Jinan, Shandong
250100, China. peng@sdu.edu.cn. Peng is supported by the Tian Yuan Projection
of the National Natural Science Foundation of China (No. 11526205 and
11626247) and the National Basic Research Program of China (973 Program) (No.
2007CB814900 (Financial Risk)).} }
\date{}
\maketitle

\textbf{Abstract}. This article establishes a universal robust limit theorem
under a sublinear expectation framework. Under moment and consistency
conditions, we show that, for $\alpha \in(1,2)$, the i.i.d. sequence
\[
\left \{  \left(  \frac{1}{\sqrt{n}}\sum_{i=1}^{n}X_{i},\frac{1}{n}\sum
_{i=1}^{n}Y_{i},\frac{1}{\sqrt[\alpha]{n}}\sum_{i=1}^{n}Z_{i}\right)
\right \}  _{n=1}^{\infty}%
\]
converges in distribution to $\tilde{L}_{1}$, where $\tilde{L}_{t}=(\tilde
{\xi}_{t},\tilde{\eta}_{t},\tilde{\zeta}_{t})$, $t\in \lbrack0,1]$, is a
multidimensional nonlinear L\'{e}vy process with an uncertainty set $\Theta$
as a set of L\'{e}vy triplets. This nonlinear L\'{e}vy process is
characterized by a fully nonlinear and possibly degenerate partial
integro-differential equation (PIDE)
\[
\left \{
\begin{array}
[c]{l}%
\displaystyle \partial_{t}u(t,x,y,z)-\sup \limits_{(F_{\mu},q,Q)\in \Theta
}\left \{  \int_{\mathbb{R}^{d}}\delta_{\lambda}u(t,x,y,z)F_{\mu}
(d\lambda)\right. \\
\displaystyle \text{\  \  \  \  \  \  \  \  \  \  \  \  \  \  \  \ }\left.  +\langle
D_{y}u(t,x,y,z),q\rangle+\frac{1}{2}tr[D_{x}^{2}u(t,x,y,z)Q]\right \}  =0,\\
\displaystyle u(0,x,y,z)=\phi(x,y,z),\  \  \forall(t,x,y,z)\in \lbrack
0,1]\times \mathbb{R}^{3d},
\end{array}
\right.
\]
with $\delta_{\lambda}u(t,x,y,z):=u(t,x,y,z+\lambda)-u(t,x,y,z)-\langle
D_{z}u(t,x,y,z),\lambda \rangle$. To construct the limit process $(\tilde
{L}_{t})_{t\in \lbrack0,1]}$, we develop a novel weak convergence approach
based on the notions of tightness and weak compactness on a sublinear
expectation space. We further prove a new type of L\'{e}vy-Khintchine
representation formula to characterize $(\tilde{L}_{t})_{t\in \lbrack0,1]}$. As
a byproduct, we also provide a probabilistic approach to prove the existence
of the above fully nonlinear degenerate PIDE. \newline


\textbf{MSC-classification}. 60F05, 60G51, 60G52, 60G65, 45K05

\section{Introduction}

Motivated by measuring risks under model uncertainty, Peng
\cite{P2004,P2007,P20081,P2010} introduced the notion of sublinear expectation
space, called $G$-expectation space. The $G$-expectation theory has been
widely used to evaluate random outcomes, not using a single probability
measure, but using the supremum over a family of possibly mutually singular
probability measures.

One of the fundamental results in the theory is Peng's robust central limit
theorem introduced in \cite{P20082,P2019,P2010}. Let $\{(X_{i},Y_{i}%
)\}_{i=1}^{\infty}$ be an i.i.d. sequence of random variables on a sublinear
expectation space $(\Omega,\mathcal{H},\mathbb{\hat{E}})$. Under certain
moment conditions, Peng proved that there exists a $G$-distributed random
variable $(\xi,\eta)$ such that
\[
\lim_{n\rightarrow \infty}\mathbb{\hat{E}}\bigg[\phi \bigg(\frac{1}{\sqrt{n}%
}\sum_{i=1}^{n}{X_{i}},\frac{1}{n}\sum_{i=1}^{n}Y_{i}%
\bigg)\bigg]=\mathbb{\tilde{E}}[\phi(\xi,\eta)],
\]
for any test function $\phi$. The $G$-distributed random variables $(\xi
,\eta)$ describes volatility and mean uncertainty, and can be characterized
via a fully nonlinear parabolic PDE, i.e., the function
$u(t,x,y):=\mathbb{\tilde{E}}[\phi(x+\sqrt{t}\xi,y+t\eta)]$ solves
\begin{equation}
\left \{
\begin{array}
[c]{l}%
\partial_{t}u(t,x,y)-G(D_{y}u,D_{x}^{2}u)=0,\\
u(0,x,y)=\phi(x,y),
\end{array}
\right.  \label{0.1}%
\end{equation}
where the sublinear function
\[
G\left(  p,A\right)  :=\mathbb{\hat{E}}\left[  \frac{1}{2}\left \langle
AX_{1},X_{1}\right \rangle +\left \langle p,Y_{1}\right \rangle \right]
,\ (p,A)\in \mathbb{R}^{d}\times \mathbb{S(}d\mathbb{)}.
\]
This limit theorem was established by Peng around 2008 (see \cite{P20082})
using the regularity theory of fully nonlinear PDEs from
\cite{CC1995,Krylov1987, Wang1992}. The corresponding convergence rate was
established by Fang et al \cite{FPSS2019} and Song \cite{Song2020} using
Stein's method and later by Krylov \cite{Krylov2020} using stochastic control
method under different model assumptions. More recently, Huang and Liang
\cite{HL2020} studied the convergence rate of a more general central limit
theorem via a monotone approximation scheme.

To further describe jump uncertainty, Hu and Peng \cite{HP2009,HP2021}
introduced a class of nonlinear L\'{e}vy processes with finite activity jumps,
called $G$-L\'{e}vy processes in the setting of sublinear expectation, and
built a type of L\'{e}vy-Khintchine representation for $G$-L\'{e}vy processes
by relating to a class of fully nonlinear partial integro-differential
equations (PIDEs).
%This provides a basic idea for the study of general nonlinear L\'{e}vy processes.
%On the basis of it,
For given characteristics, more general nonlinear L\'{e}vy processes with
infinite activity jumps have been studied by Neufeld and Nutz \cite{NN2017}
(see also \cite{DKN2020,Kuhn2019,NR2021}). An important class of nonlinear
L\'{e}vy processes is the $\alpha$-stable process $(\zeta_{t})_{t\geq0}$ for
$\alpha \in(1,2)$, whose characteristic is described by an uncertainty set
$\Theta_{0}=\{(F_{k_{\pm}},0,0):k_{\pm}\in K_{\pm}\}$, where $K_{\pm}%
\subset(\lambda_{1},\lambda_{2})$ for some $\lambda_{1},\lambda_{2}>0$,
$(0,0)$ means that $(\zeta_{t})_{t\geq0}$ is a pure jump L\'{e}vy process
without diffusion and drift, and $F_{k_{\pm}}(dz)$ is the $\alpha$-stable
L\'{e}vy measure
\[
F_{k_{\pm}}(dz)=\frac{k_{-}}{|z|^{\alpha+1}}\mathbf{1}_{(-\infty
,0)}(z)dz+\frac{k_{+}}{|z|^{\alpha+1}}\mathbf{1}_{(0,\infty)}(z)dz.
\]
The nonlinear $\alpha$-stable process $(\zeta_{t})_{t\geq0}$ on a sublinear
expectation space $(\Omega,\mathcal{H},\mathbb{\tilde{E}})$ can be
characterized via a fully nonlinear PIDE, i.e., the function
$u(t,x):=\mathbb{\tilde{E}}[\phi(x+\zeta_{t})]$ solves
\begin{equation}
\left \{
\begin{array}
[c]{l}%
\displaystyle \partial_{t}u(t,x)-\sup \limits_{k_{\pm}\in K_{\pm}}\left \{
\int_{\mathbb{R}}\delta_{\lambda}u(t,x)F_{k_{\pm}}(d\lambda)\right \}  =0,\\
\displaystyle u(0,x)=\phi(x),
\end{array}
\right.  \label{0.2}%
\end{equation}
where $\delta_{\lambda}u(t,x):=u(t,x+\lambda)-u(t,x)-D_{x}u(t,x)\lambda$.

The corresponding limit theorem for $\alpha$-stable processes under sublinear
expectation was established by Bayraktar and Munk \cite{BM2016}. Let
$\{Z_{i}\}_{i=1}^{\infty}$ be an i.i.d. sequence of real-valued random
variables on a sublinear expectation space $(\Omega,\mathcal{H},\mathbb{\hat
{E})}$ satisfying certain moment and consistency conditions, they proved that
there exists a nonlinear $\alpha$-stable process $(\zeta_{t})_{t\geq0}$ such
that
\[
\lim_{n\rightarrow \infty}\mathbb{\hat{E}}\bigg[\phi \bigg(\frac{1}%
{\sqrt[\alpha]{n}}\sum_{i=1}^{n} Z_{i}\bigg)\bigg]=\mathbb{\tilde{E}}%
[\phi(\zeta_{1})],
\]
for any test function $\phi$. Their proof relies on the interior regularity
results of fully nonlinear PIDEs from \cite{CC1995,LD20161,LD20162}. More
recently, Hu et al \cite{HJL2021} established the corresponding convergence
rate via a monotone approximation scheme.

The aim of this article is to \emph{study the robust limit theorem for a
multidimensional nonlinear L\'evy process under a sublinear expectation
framework}. To be more specific, let $\{(X_{i},Y_{i},Z_{i})\}_{i=1}^{\infty}$
be an i.i.d. sequence of $\mathbb{R}^{3d}$-valued random variables on a
sublinear expectation space $(\Omega,\mathcal{H},\mathbb{\hat{E})}$ and
$\alpha \in(1,2)$. Then, the first question is under what conditions does the
following i.i.d. sequence
\[
\left \{  \left(  \frac{1}{\sqrt{n}}\sum_{i=1}^{n}X_{i},\frac{1}{n}\sum
_{i=1}^{n}Y_{i},\frac{1}{\sqrt[\alpha]{n}} \sum_{i=1}^{n}Z_{i}\right)
\right \}  _{n=1}^{\infty}%
\]
converge? If so, then the second question is how to characterize the limit? We
provide affirmative answers for both questions in Theorem \ref{main theorem},
which is dubbed as \emph{a universal robust limit theorem} under sublinear
expectation. The result covers all the existing robust limit theorems in the
literature, namely, Peng's robust central limit theorem for $G$-distribution
(see \cite{P20082}) and Bayraktar-Munk's robust limit theorem for $\alpha
$-stable distribution (see \cite{BM2016}). One remarkable feature of the
result is that $(X_{1},Y_{1},Z_{1})$ may depend on each other, so one cannot
simply combine the robust limit theorems in \cite{BM2016} and \cite{P20082}.
Moreover, the conditions that we propose are mild. In fact, they are weaker
than the characterization condition proposed in linear setting (see Remark
\ref{Remark}) and the consistency condition in nonlinear setting (see Remark
\ref{Remark0}).

On the other hand, the existing methods for the robust limit theorems do not
work, because \cite{BM2016,P20082} reply on the regularity estimates of the
fully nonlinear PDE (\ref{0.1}) and PIDE (\ref{0.2}). However, the required
regularity for the general equation (\ref{PIDE}) in Theorem \ref{main theorem}
is unknown to date. Moreover, it seems that most of the existing methods are
analytical and heavily rely on the regularity theory. It is natural to ask
whether one can establish a probabilistic proof as in the classical linear
expectation case. As expected, weak convergence plays a pivotal role. Peng
\cite{P2010_CLT} firstly introduced the notions of tightness and weak
compactness on a sublinear expectation space and provided an alternative proof
for his robust central limit theorem. In Theorem
\ref{The construction of Levy process}, we further develop this weak
convergence approach to establish a Donsker-type result showing that the limit
indeed exists and is a nonlinear L\'evy process $\tilde{L}_{t}:=(\tilde{\xi
}_{t},\tilde{\eta} _{t},\tilde{\zeta}_{t})$ at $t=1$ with $(\tilde{\xi}%
_{1},\tilde{\eta} _{1})$ following $G$-distribution.

A more challenging task is to characterize the third component $\tilde{\zeta
}_{t}$ and also $\tilde{L}_{t}$ as a whole. This will in turn link the
nonlinear L\'evy process $(\tilde{L}_{t})_{t\in[0,1]}$ with the fully
nonlinear PIDE (\ref{PIDE}). However, the proofs for the classical linear
expectation cases (e.g. CLT and $\alpha$-stable limit theorem) are to a
considerable extent based on characteristic function techniques, which do not
exist in the sublinear framework. A new type of L\'evy-Khintichine formula is
therefore needed. Note that the proof of this representation formula is also
an important open question left in the literature (see Remark 23 in
\cite{HP2021} and Page 71 in \cite{NN2017}). We overcome this difficulty by
deriving a new estimate for the $\alpha$-stable L\'{e}vy measure (see Theorem
\ref{recursive}), which in turn enables us to prove a new type of
L\'evy-Khintichine representation formula for the nonlinear L\'evy process
(see Theorem \ref{represent theorem}). Thanks to the connection with the fully
nonlinear PIDE (\ref{PIDE}), we also obtain the existence of its viscosity
solution as a byproduct.

The article is organized as follows. In Section 2, we review some necessary
results about sublinear expectation. Section 3 details our main result: the
universal robust limit theorem in Theorem \ref{main theorem}. The proofs of
the main theorem is given in Section 4. Finally, an example highlighting the
applications of our main result is given in Section 5.

%%%%%%%%%%%%%%%%%%%%%%%%%%%%%%%%%%%%%%%%%%%%%%%%%%%%%%%%%%%%%%%%%%%%%%%%%%%%%%%%%%%%%%

\section{Preliminaries}

This section briefly introduces notions and preliminaries in the sublinear
expectation framework. For more details, we refer the reader to
\cite{P2007,P20081,P2010_CLT,P2010} and the references therein.

\subsection{Sublinear expectation}

Let $\Omega$ be a given set and let $\mathcal{H}$ be a linear space of real
valued functions defined on $\Omega$ such that $1\in \mathcal{H}$ and $\vert X
\vert \in \mathcal{H}$ if $X\in \mathcal{H}$. Then, a sublinear expectation is
defined as follows.

\begin{definition}
A functional $\mathbb{\hat{E}}$: $\mathcal{H}\rightarrow \mathbb{R}$ is called
a sublinear expectation if, for all $X,Y\in \mathcal{H}$, it satisfies the
following properties.

\begin{description}
\item[(i)] (Monotonicity) $\mathbb{\hat{E}}[X] \geq \mathbb{\hat{E}}[Y]$, if
$X\geq Y$;

\item[(ii)] (Constant preservation) $\mathbb{\hat{E}}[c] =c$, for
$c\in \mathbb{R}$;

\item[(iii)] (Sub-additivity) $\mathbb{\hat{E}}[X+Y] \leq \mathbb{\hat{E}} [X]
+\mathbb{\hat{E}}[Y] ;$

\item[(iv)] (Positive homogeneity) $\mathbb{\hat{E}}[\lambda X] =\lambda
\mathbb{\hat{E}}[X]$, for $\lambda>0$.
\end{description}
\end{definition}

The triplet $(\Omega,\mathcal{H},\mathbb{\hat{E})}$ is called a sublinear
expectation space.
%If (i) and (ii) are satisfied, $\mathbb{\hat{E}}$ is called
%a nonlinear expectation and the triplet $(\Omega,\mathcal{H},\mathbb{\hat{E}%
%)}$ is called a nonlinear expectation space.
From the definition of the sublinear expectation $\mathbb{\hat{E}}$, the
following results can be easily obtained.

\begin{proposition}
\label{Prop^E Plimi}For $X,Y$ $\in \mathcal{H}$, we have

\begin{description}
\item[(i)] if $\mathbb{\hat{E}}[X] =-\mathbb{\hat{E}}[-X]$, then
$\mathbb{\hat{E}}[X+Y] =\mathbb{\hat{E}}[X] +\mathbb{\hat{E}}[Y] ;$

\item[(ii)] $|\mathbb{\hat{E}}[X]-\mathbb{\hat{E}}[Y]|\leq \mathbb{\hat{E}%
}[\vert X-Y\vert]$, i.e., $|\mathbb{\hat{E}}[X] -\mathbb{\hat{E}}[Y]
|\leq \mathbb{\hat{E}}[X-Y] \vee \mathbb{\hat{E}}[Y-X];$

\item[(iii)] $\mathbb{\hat{E}}[\vert XY \vert] \leq(\mathbb{\hat{E}}[\vert
X\vert^{p}])^{1/p} (\mathbb{\hat{E}}[\vert Y\vert^{q}] )^{1/q},$ for $1\leq
p,q<\infty$ with $\frac{1}{p}+\frac{1}{q}=1.$
\end{description}
\end{proposition}

\begin{definition}
We say $X_{1}$ on a sublinear expectation space $(\Omega_{1},\mathcal{H}%
_{1},\hat{E}_{1})$ and $X_{2}$ on another sublinear expectation space
$(\Omega_{2},\mathcal{H}_{2},\mathbb{\hat{E}}_{2})$ identically distributed,
if $\mathbb{\hat{E}}_{1}[\varphi(X_{1})]=\mathbb{\hat{E}} _{2}[\varphi
(X_{2})]$, for all $\varphi \in{C_{b,Lip}}(\mathbb{R}^{n})$, the space of
bounded Lipschitz continuous functions on $\mathbb{R}^{n}$.
\end{definition}

The concept of independence plays a pivotal role in sublinear expectation.
Notably, $Y$ is independent from $X$ does not necessarily imply that $X$ is
independent from $Y$.

\begin{definition}
Let $(\Omega,\mathcal{H},\mathbb{\hat{E})}$ be a sublinear expectation space.
An $n$-dimensional random variable $Y$ is said to be independent from another
$m$-dimensional random variable $X$ under $\mathbb{\hat{E}}[\cdot]$, denoted
by $Y \perp \! \! \! \perp X$, if for every test function $\varphi \in
C_{b,Lip}(\mathbb{R}^{m}\times \mathbb{R}^{n})$ we have
\[
\mathbb{\hat{E}}\left[  \varphi(X,Y)\right]  =\mathbb{\hat{E}}\left[
\mathbb{\hat{E}}\left[  \varphi(x,Y)\right]  _{x=X}\right]  .
\]
$\bar{X}$ is said to be an independent copy of $X$ if $\bar{X}\overset{d}{=}X$
and $\bar{X}\perp \! \! \! \perp X$.
\end{definition}

The independence assumption implies the additivity of $\mathbb{\hat{E}}$ as
shown in the following proposition.

\begin{proposition}
\label{Y independent X}For $X,Y \in \mathcal{H}$, if $Y\perp \! \! \! \perp X$,
then
\[
\mathbb{\hat{E}}[X+Y] =\mathbb{\hat{E}}[X] +\mathbb{\hat{E}}[Y].
\]

\end{proposition}

\subsection{Tightness, weak compactness, and convergence in distribution}

Weak convergence plays an important role in establishing the universal robust
limit theorem. The following definitions of tightness and weak compactness are
adapted from Peng \cite[Definitions 7 and 8]{P2010_CLT}.

\begin{definition}
\label{def_2.7} A sublinear expectation $\mathbb{\hat{E}}$ on $(\mathbb{R}%
^{n},C_{b,Lip}(\mathbb{R}^{n}))$ is said to be tight if for each
$\varepsilon>0$, there exist an $N>0$ and $\varphi \in$ $C_{b,Lip}%
(\mathbb{R}^{n})$ with $\mathbbm{1}_{\{|x|\geq N\}}\leq \varphi$ such that
$\mathbb{\hat{E}}[\varphi]<\varepsilon$.
\end{definition}

\begin{definition}
\label{def_2.8} A family of sublinear expectations $\{ \mathbb{\hat{E}%
}_{\alpha}\}_{\alpha \in \mathcal{A}}$ on $(\mathbb{R}^{n},C_{b,Lip}%
(\mathbb{R}^{n}))$ is said to be tight if there exists a tight sublinear
expectation $\mathbb{\hat{E}}$ on $(\mathbb{R}^{n},C_{b,Lip}(\mathbb{R}^{n}))$
such that
\[
\mathbb{\hat{E}}_{\alpha}[\varphi]-\mathbb{\hat{E}}_{\alpha}[\varphi^{\prime
}]\leq \mathbb{\hat{E}}[\varphi-\varphi^{\prime}],\text{ for each }
\varphi,\varphi^{\prime}\in C_{b,Lip}(\mathbb{R}^{n}).
\]

\end{definition}

\begin{definition}
Let $\{ \mathbb{\hat{E}}_{n}\}_{n=1}^{\infty}$ be a sequence of sublinear
expectations defined on $(\mathbb{R}^{n},C_{b,Lip}(\mathbb{R}^{n}))$. They are
said to be weakly convergent if, for each $\varphi \in C_{b,Lip}(\mathbb{R}%
^{n})$, $\{ \mathbb{\hat{E}}_{n}[\varphi]\}_{n=1}^{\infty}$ is a Cauchy
sequence. A family of sublinear expectations $\{ \mathbb{\hat{E}}_{\alpha
}\}_{\alpha \in \mathcal{A}}$ defined on $(\mathbb{R}^{n},C_{b,Lip}%
(\mathbb{R}^{n}))$ is said to be weakly compact if for each sequence $\{
\mathbb{\hat{E}}_{\alpha_{i}}\}_{i=1}^{\infty}$ there exists a weakly
convergent subsequence.
\end{definition}

The following result is a generalization of the celebrated Prokhorov's theorem
to the sublinear expectation case, first proved in Peng \cite[Theorem
9]{P2010_CLT}. For the reader's convenience, we give the proof of Theorem
\ref{tight theorem} in Appendix under the tightness condition introduced in
Definition \ref{def_2.7}.

\begin{theorem}
[\cite{P2010_CLT}]\label{tight theorem} Let $\{ \mathbb{\hat{E}}_{\alpha
}\}_{\alpha \in \mathcal{A}}$ be a family of tight sublinear expectations on
$(\mathbb{R}^{n},C_{b,Lip}(\mathbb{R}^{n}))$. Then $\{ \mathbb{\hat{E}%
}_{\alpha}\}_{\alpha \in \mathcal{A}}$ is weakly compact, namely, for each
sequence $\{ \mathbb{\hat{E}}_{\alpha_{n}}\}_{n=1}^{\infty}$, there exists a
subsequence $\{ \mathbb{\hat{E}}_{\alpha_{n_{i}}}\}_{i=1}^{\infty}$ such that,
for each $\varphi \in C_{b,Lip}(\mathbb{R}^{n})$, $\{ \mathbb{\hat{E}}%
_{\alpha_{n_{i}}}[\varphi]\}_{i=1}^{\infty}$ is a Cauchy sequence.
\end{theorem}

Given the weak convergence of sublinear expectations, the convergence of
random variables can be defined accordingly as follows.

\begin{definition}
\label{converge in distribution}A sequence of $n$-dimensional random variables
$\{X_{i}\}_{i=1}^{\infty}$ defined on a sublinear expectation space
$(\Omega,\mathcal{H},\mathbb{\hat{E}})$ is said to converge in distribution
(or converge in law) under $\mathbb{\hat{E}}$ if for each $\varphi \in
C_{b,Lip}(\mathbb{R}^{n})$, the sequence $\{ \mathbb{\hat{E} }[\varphi
(X_{i})]\}_{i=1}^{\infty}$ converges. For each random variable $X_{i}$, the
mapping $\mathbb{F}_{X_{i}}[\cdot]:C_{b,Lip}(\mathbb{R} ^{n})\rightarrow
\mathbb{R}$ defined by
\[
\mathbb{F}_{X_{i}}[\varphi]:=\mathbb{\hat{E}}[\varphi(X_{i})],\text{ for
}\varphi \in C_{b,Lip}(\mathbb{R}^{n})
\]
is a sublinear expectation defined on $(\mathbb{R}^{n},C_{b,Lip}%
(\mathbb{R}^{n}))$.
\end{definition}

An immediate corollary of Theorem \ref{tight theorem} and Definition
\ref{converge in distribution} is the following result.

\begin{corollary}
\label{corollary tight theorem} Let $\{X_{i}\}_{i=1}^{\infty}$ be a sequence
of $n$-dimensional random variables defined on a sublinear expectation space
$(\Omega,\mathcal{H},\mathbb{\hat{E}})$. If $\{ \mathbb{F}_{X_{i}}%
\}_{i=1}^{\infty}$ is tight, then there exists a subsequence $\{X_{i_{j}%
}\}_{j=1}^{\infty}\subset \{X_{i}\}_{i=1}^{\infty}$ which converges in distribution.
\end{corollary}

One can then easily obtain the following result concerning the convergence in
distribution of random variables.

\begin{proposition}
\label{independent copy converge in law}Let $\{X_{i}\}_{i=1}^{\infty}$ be a
sequence of $n$-dimensional random variables defined on sublinear expectation
spaces $(\Omega,\mathcal{H},\mathbb{\hat{E}})$ and $\bar{X}_{i}$ be an
independent copy of $X_{i}$ for $i\in \mathbb{N}$. If $\{X_{i}\}_{i=1}^{\infty
}$ converges in law to $X$ in $(\Omega,\mathcal{H},\mathbb{\tilde{E}})$,
i.e.,
\[
\lim_{i\rightarrow \infty}\mathbb{\hat{E}}[\varphi(X_{i})]=\mathbb{\tilde{E}
}[\varphi(X)],\text{ for }\varphi \in C_{b,Lip}(\mathbb{R}^{n}),
\]
which we denote by $X_{i}\overset{\mathcal{D}}{\rightarrow}X$. Then,
\[
\bar{X}_{i}\overset{\mathcal{D}}{\rightarrow}\bar{X}\text{\ and \ }X_{i}%
+\bar{X}_{i}\overset{\mathcal{D}}{\rightarrow}X+\bar{X},
\]
where $\bar{X}$ is an independent copy of $X$.
\end{proposition}

\begin{remark}
\label{ramark i.i.d.} The above results can be generalized to multiple
summations. For each $i$ and $m\in \mathbb{N}$, let $\{X_{i}^{n}\}_{n=1}^{m}$
be an independent copy sequence of $X_{i}$ in the sense that $X_{i}%
^{1}\overset{d}{=}X_{i}$, $X_{i}^{n+1}\overset{d}{=}X_{i}^{n}$ and
$X_{i}^{n+1}\perp \! \! \! \perp(X_{i}^{1},X_{i}^{2},\ldots,X_{i}^{n})$ for
$n=1,\ldots,m-1$, and let $\{X^{n}\}_{n=1}^{m}$ be an independent copy
sequence of $X$ in the sense that $X^{1}\overset{d}{=}X$, $X^{n+1}\overset
{d}{=}X^{n}$ and $X^{n+1}\perp \! \! \! \perp(X^{1},X^{2},\ldots,X^{n})$ for
$n=1,\ldots,m-1$. If $X_{i}\overset{\mathcal{D}}{\rightarrow}X$, then
\[
\sum_{n=1}^{m}X_{i}^{n}\overset{\mathcal{D}}{\rightarrow}\sum_{n=1}^{m}X^{n}.
\]

\end{remark}

\subsection{The robust central limit theorem for $G$-distribution}

One of the most important class of distributions in the sublinear expectation
framework is $G$-distribution, which characterizes volatility and mean
uncertainty via $(\xi,\eta)$ below.

\begin{definition}
The pair of $d$-dimensional random variables $(\xi,\eta)$ on a sublinear
expectation space $(\Omega,\mathcal{H},\mathbb{\hat{E})}$ is called
$G$-distributed if for each $a,b\geq0$ we have
\begin{equation}
\left(  a\xi+b\bar{\xi},a^{2}\eta+b^{2}\bar{\eta}\right)  \overset{d}%
{=}\left(  \sqrt{a^{2}+b^{2}}\xi,(a^{2}+b^{2})\eta \right)  ,\label{G}%
\end{equation}
where $(\bar{\xi},\bar{\eta})$ is an independent copy of $(\xi,\eta)$, and
$G:\mathbb{R}^{d}\times \mathbb{S}(d)\rightarrow \mathbb{R}$ denotes the binary
function
\begin{equation}
G\left(  p,A\right)  :=\mathbb{\hat{E}}\left[  \frac{1}{2}\left \langle
A\xi,\xi \right \rangle +\left \langle p,\eta \right \rangle \right]
,\ (p,A)\in \mathbb{R}^{d}\times \mathbb{S}(d), \label{G(p,A)}%
\end{equation}
where $\mathbb{S}(d)$ denotes the collection of all $\mathbb{R}^{d\times d}$
symmetric matrices.
\end{definition}

\begin{remark}
In fact, if the pair $(\xi,\eta)$ satisfies (\ref{G}), then
\[
a\xi+b\bar{\xi}\overset{d}{=}\sqrt{a^{2}+b^{2}}\xi,\text{ \ }a\eta+b\bar{\eta
}\overset{d}{=}(a+b)\eta, \text{ \ for }a,b\geq0.
\]
This implies that $\xi$ is $G$-normally distributed and $\eta$ is maximally
distributed (cf. Peng \cite{P2010}).
\end{remark}

\begin{remark}
\label{remark_G} For the latter use, we recall from Proposition 4.1 in Peng
\cite{P20082} that there exists a bounded and closed subset $\Gamma
\subset \mathbb{R}^{d}\times \mathbb{S}_{+}(d)$ such that for $(p,A)\in
\mathbb{R}^{d}\times \mathbb{S}(d)$,
\[
G(p,A)=\sup_{(q,Q)\in \Gamma}\left[  \langle p,q\rangle+\frac{1}{2}%
tr[AQ]\right]  ,
\]
where $\mathbb{S}_{+}(d)$ denotes the collection of nonnegative definite
elements in $\mathbb{S}(d)$.
\end{remark}

%The definition of $G$-distributed can be characterized by the parabolic PDE:
%\begin{proposition}
%[\cite{P2010}]Let $(\xi,\eta)$ satisfy (\ref{G}). For each $\phi \in
%C_{b,Lip}(\mathbb{R}^{d}\times \mathbb{R}^{d})$ we define%
%\[
%v(t,x,y):=\mathbb{\hat{E}}[\phi(x+\sqrt{t}\xi,y+t\eta)]\text{, }%
%(t,x,y)\in \lbrack0,\infty)\times \mathbb{R}^{d}\times \mathbb{R}^{d}.
%\]
%Then $v$ is the unique viscosity solution of the following parabolic
%PDE£º
%\[
%\left \{
%\begin{array}
%[c]{l}%
%\partial_{t}v-G(D_{y}v,D_{x}^{2}v)=0,\text{ }(t,x,y)\in(0,\infty
%)\times \mathbb{R}^{d}\times \mathbb{R}^{d},\\
%v(0,x,y)=\phi(x,y),\text{ \ }(x,y)\in \mathbb{R}^{d}\times \mathbb{R}^{d},
%\end{array}
%\right.
%\]
%where $G:\mathbb{R}^{d}\times \mathbb{S}(d)\rightarrow \mathbb{R}$ is defined by
%(\ref{G(p,A)}).
%\end{proposition}

The robust central limit theorem for $G$-distribution was established by Peng
around 2008 using the regularity theory of fully nonlinear PDEs in
\cite{P20082}. A probabilistic proof using the weak convergence argument was
subsequently established in \cite{P2010_CLT}. We recall this major theorem in
its general form as in \cite[Theorem 2.4.7]{P2010} (see also
\cite{C2016,GL2021,Z2016} for more related research).

\begin{theorem}
[\cite{P2010}]\label{LLT+CLT}Let $\{(X_{i},Y_{i})\}_{i=1}^{\infty}$ be a
sequence of $\mathbb{R}^{2d}$-valued random variables on a sublinear
expectation space $(\Omega,\mathcal{H},\mathbb{\hat{E}})$. We assume that
$(X_{i+1},Y_{i+1})\overset{d}{=}(X_{i},Y_{i})$ and $(X_{i+1},Y_{i+1})$ is
independent from $\{(X_{1},Y_{1}),\ldots,(X_{i},Y_{i})\}$ for each
$i=1,2,\cdots$. We further assume that $\mathbb{\hat{E}}[X_{1}]=\mathbb{\hat
{E}}[-X_{1}]=0$ and
\[
\lim \limits_{\gamma \rightarrow+\infty}\mathbb{\hat{E}}[(|X_{1}|^{2}-
\gamma)^{+}]=0,\text{ \ }\lim \limits_{\gamma \rightarrow+\infty}\mathbb{\hat
{E}}[(|Y_{1}|-\gamma)^{+}]=0.
\]
Then for each function $\varphi \in C(\mathbb{R}^{2d})$ satisfying linear
growth condition, we have
\[
\lim_{n\rightarrow \infty}\mathbb{\hat{E}}\bigg[\varphi \bigg(\sum_{i=1}^{n}
\frac{X_{i}}{\sqrt{n}},\sum_{i=1}^{n}\frac{Y_{i}}{n}\bigg)\bigg]=\mathbb{\hat
{E}}\big[\varphi(\xi,\eta)\big],
\]
where the pair $(\xi,\eta)$ is $G$-distributed and the corresponding sublinear
function $G:$ $\mathbb{R}^{d}\times \mathbb{S}(d)$ $\rightarrow \mathbb{R}$ is
defined by
\[
G\left(  p,A\right)  :=\mathbb{\hat{E}}\left[  \frac{1}{2}\left \langle
AX_{1},X_{1}\right \rangle +\left \langle p,Y_{1}\right \rangle \right]
,\ (p,A)\in \mathbb{R}^{d}\times \mathbb{S}(d).
\]

\end{theorem}

%\textcolor[rgb]{1,0,0}{In contrast to \cite{HL2020} and \cite{Krylov2020}, where the third moment
%condition on $X_i$ is imposed, it seems impossible to obtain the convergence
%rate under the above moment condition. Moreover, it seems the monotone
%scheme method proposed in \cite{HL2020} and stochastic control method
%proposed in \cite{Krylov2020} do not apply.}

%\begin{definition}
%Let $G:\mathbb{R}^{d}\times \mathbb{S(}d\mathbb{)}\rightarrow \mathbb{R}$ be a
%given sublinear function monotonic in $A\in \mathbb{S(}d\mathbb{)}$. A
%$d$-dimensional process $(B_{t})_{t\geq0}$ on a sublinear expectation space is
%called a generalized $G$-Brownian motion if the following properties hold:
%\begin{description}
%\item[(i)] $B_{0}=0$;
%\item[(ii)] For each $0\leq s\leq t$, $B_{t+s}-B_{t}\overset{d}{=}\sqrt{s}%
%\xi+s\eta$, where $\left(  \xi,\eta \right)  $ is $G$-distributed;
%\item[(iii)] For each $0\leq s\leq t$, $B_{t+s}-B_{t}$ is independent from
%$(B_{t_{1}},\ldots,B_{t_{n}})$ for each $n\in \mathbb{N}$ and $0\leq t_{1}%
%\leq \cdots \leq t_{n}\leq t$.
%\end{description}
%\end{definition}

\subsection{The $\alpha$-stable limit theorem for $\alpha$-stable
distribution}

Let us first recall the classical $\alpha$-stable limit theorem in the linear
case. let $\{Z_{i}\}_{i=1}^{\infty}$ be a sequence of i.i.d. random variables
on a classical probability space. While the central limit theorem states that
the distribution of $n^{-1/2}\sum_{i=1}^{n}Z_{i}$ will converge to a normal
distribution when $Z_{1}$ has finite variance, the $\alpha$-stable limit
theorem states that the distribution of $n^{-1/\alpha}\sum_{i=1}^{n}Z_{i}$
with will converge to a classical $\alpha$-stable distribution for $\alpha
\in(0,2)$ when $Z_{1}$ has power-law tails decreasing as ${|x|^{-\alpha-1}}$
(see (\ref{normal attraction of stable law})).

\begin{definition}
The common distribution $F_{Z}$ of an i.i.d. sequence $\{Z_{i}\}_{i=1}%
^{\infty}$ is said in the domain of normal attraction of an $\alpha$-stable
distribution $F$ with $\alpha \in(0,2)$, if there exist nonnegative constants
$a_{n}$ and $b_{n}=bn^{\frac{1}{\alpha}}$ for $n=1,2,\cdots$, such that the
distribution of
\[
\frac{1}{b_{n}}\bigg(\sum_{i=1}^{n}Z_{i}-a_{n}\bigg)
\]
weakly converges to $F$ as $n\rightarrow \infty$.

%Moreover, $F_{Z}$ is said
%in the domain of normal attraction of an $\alpha $-stable distribution $G$
%with $\alpha \in (0,2)$, if
%\begin{equation*}
%b_{n}=bn^{\frac{1}{\alpha }},\text{\ \ for some }b>0.
%\end{equation*}

\end{definition}

The following theorem characterizes the domain of normal attraction of an
$\alpha$-stable distribution via a characterization condition (see
(\ref{normal attraction of stable law}) below). It can be found in Ibragimov
and Linnik \cite[Theorem 2.6.7]{IL1971}.

\begin{theorem}
[\cite{IL1971}]\label{classical_condition} The distribution $F_{Z}$ belongs to
the domain of normal attraction of an $\alpha$-stable distribution $F$ for
$\alpha \in(0,2)$ if and only if
\begin{equation}
F_{Z}(z)=\left \{
\begin{array}
[c]{ll}%
\displaystyle[c_{1}b^{\alpha}+\beta_{1}(z)]\frac{1}{|z|^{\alpha}}, & z<0,\\
\displaystyle1-[c_{2}b^{\alpha}+\beta_{2}(z)]\frac{1}{z^{\alpha}}, & z>0,
\end{array}
\right.  \label{normal attraction of stable law}%
\end{equation}
where $c_{1}$ and $c_{2}$ are constants with $c_{1},c_{2}\geq0$, $c_{1}%
+c_{2}>0$ related to the $\alpha$-stable distribution, and some functions
$\beta_{1}$ and $\beta_{2}$ satisfying
\[
\lim_{z\rightarrow-\infty}\beta_{1}(z)=\lim_{z\rightarrow \infty}\beta
_{2}(z)=0.
\]

\end{theorem}

%For recent development of the above classical $\alpha$-stable limit theorem for
%$\alpha$-stable distribution, we refer to \cite{CW1992,GK1954,H1981,HL2015,UZ1999}
%and the references therein for more research in this field.

%\begin{theorem}
%Let $\{Z_{i}\}_{i=1}^{\infty }$ be an i.i.d. sequence with distribution function $F_{Z}$
%satisfying (\ref{normal attraction of stable law}) with $\alpha\in(0,2)$.
%Then, there exist
%nonnegative constants $a_{n}$ and $b_{n}=bn^{\frac{1}{\alpha }}$,
%$n=1,2,\cdots$, such that the distribution of the centered and
%normalized sum
%\begin{equation*}
%\frac{1}{b_{n}}\bigg(\sum_{i=1}^{n}Z_{i}-a_{n}\bigg)
%\end{equation*}
%weakly converges to the $\alpha$-stable distribution as $n\rightarrow \infty
%$.
%\end{theorem}

In the following, we will present the nonlinear version of the $\alpha$-stable
limit theorem. Let us start by recalling the definition of $\alpha$-stable
distribution under sublinear expectations.

\begin{definition}
Let $\alpha \in(1,2)$. A random variable $\zeta$ on a sublinear expectation
space $(\Omega,\mathcal{H},\mathbb{\hat{E})}$ is said to be (strictly)
$\alpha$-stable if for all $a,b\geq0$,
\[
a\zeta+b\bar{\zeta}\overset{d}{=}(a^{\alpha}+b^{\alpha})^{1/\alpha}\zeta,
\]
where $\bar{\zeta}$ is an independent copy of $\zeta$.
\end{definition}

By analogy with the classical case, a nonlinear $\alpha$-stable random
variable $\zeta$ can be characterized by a set of L\'{e}vy triplets (see, for
example, Neufeld and Nutz \cite{NN2017}). For $\alpha \in(1,2)$, we consider
$\zeta$
%for a one-dimensional nonlinear
%$\alpha$-stable process $(\zeta_{t})_{t\geq0}$
on the sublinear expectation space $(\Omega,\mathcal{H},\mathbb{\tilde{E}})$
whose characteristics are described by a set of L\'{e}vy triplets $\Theta
_{0}=\{(F_{k_{\pm}},0,0):k_{\pm}\in K_{\pm}\}$, where $K_{\pm}\subset
(\lambda_{1},\lambda_{2})$ for some $\lambda_{1},\lambda_{2}>0$ and
$F_{k_{\pm}}(dz)$ is the $\alpha$-stable L\'{e}vy measure
\[
F_{k_{\pm}}(dz)=\frac{k_{-}}{|z|^{\alpha+1}}\mathbf{1}_{(-\infty
,0)}(z)dz+\frac{k_{+}}{|z|^{\alpha+1}}\mathbf{1}_{(0,\infty)}(z)dz.
\]

The corresponding nonlinear $\alpha$-stable limit theorem was first
established by Bayraktar and Munk \cite[Theorem 3.1]{BM2016}.
%For $\alpha \in(1,2)$, they considered a nonlinear
%$\alpha$-stable distribution $\zeta$ on a sublinear expectation space with an
%uncertainty set $\Theta_{0}=\{(F_{k_{\pm}},0,0):k_{\pm}\in K_{\pm}\}$, where
%$K_{\pm}\subset(\lambda_{1},\lambda_{2})$ for some $\lambda_{1},\lambda
%_{2}\geq0$ and $F_{k_{\pm}}(dz)$ is the $\alpha$-stable L\'{e}vy measure
%\[
%F_{k_{\pm}}(dz)=\frac{k_{-}}{|z|^{\alpha+1}}\mathbf{1}_{(-\infty
%,0)}(z)dz+\frac{k_{+}}{|z|^{\alpha+1}}\mathbf{1}_{(0,+\infty)}(z)dz\text{.}%
%\]
Since cumulative distribution functions do not exist in the sublinear
framework, they further replace the characterization condition
(\ref{normal attraction of stable law}) by a consistency condition (see (ii) below).

\begin{theorem}
[\cite{BM2016}]\label{Theorem_Erhan} Let $\{Z_{i}\}_{i=1}^{\infty}$ be an
i.i.d. sequence of real-valued random variables on a sublinear expectation
space $(\Omega,\mathcal{H},\mathbb{\hat{E})}$ in the sense that $Z_{i+1}%
\overset{d}{=}Z_{i}$ and $Z_{i+1}\perp \! \! \! \perp(Z_{1},Z_{2},\ldots
,Z_{i})$ for each $i\in \mathbb{N}$, and $b_{n}=bn^{\frac{1}{\alpha}}$, for
some $b>0$. Suppose that

\begin{description}
\item[(i)] $\mathbb{\hat{E}}[Z_{1}]=\mathbb{\hat{E}}[-Z_{1}]=0$ and
$\mathbb{\hat{E}}[|Z_{1}|]<\infty$;

\item[(ii)] for any $0<h<1$ and $\varphi \in C_{b,Lip}(\mathbb{R})$,
\[
n\bigg \vert \mathbb{\hat{E}}\big[\delta_{b_{n}^{-1}Z_{1}}v(t,x)\big]-\frac
{1}{n}\sup \limits_{k_{\pm}\in K_{\pm}}\bigg \{ \int_{\mathbb{R}}\delta
_{z}v(t,x)F_{k_{\pm}}(dz)\bigg \} \bigg \vert \rightarrow0,\  \ n\rightarrow
\infty,
\]
uniformly on $[0,1]\times \mathbb{R}$, where $v$ is the unique viscosity
solution of
\begin{equation}
\left \{
\begin{array}
[c]{ll}%
\displaystyle \partial_{t}v(t,x)+\sup \limits_{k_{\pm}\in K_{\pm}}\left \{
\int_{\mathbb{R}}\delta_{z}v(t,x)F_{k_{\pm}}(dz)\right \}  =0, & (-h,1+h)\times
\mathbb{R},\\
\displaystyle v(1+h,x)=\varphi(x),\text{ \ }x\in \mathbb{R}, &
\end{array}
\right.  \label{PIDE_Erhan}%
\end{equation}
with $\delta_{z}v(t,x):=v(t,x+z)-v(t,x)-D_{x}v(t,x)z$.
\end{description}

Then
\[
\frac{1}{b_{n}}\sum_{i=1}^{n}Z_{i}\overset{\mathcal{D}}{\rightarrow}%
\zeta,\  \  \  \  \text{as}\ n\rightarrow \infty.
\]
%for any $\phi \in C_{b,Lip}(\mathbb{R})$
%\begin{equation*}
%\lim_{n\rightarrow \infty }\mathbb{\hat{E}}\bigg[\phi \bigg(\frac{1}{b_{n}}%
%\sum_{i=1}^{n}Z_{i}\bigg)\bigg]=\mathbb{\tilde{E}}[\phi (\zeta )].
%\end{equation*}

\end{theorem}

\begin{remark}
When the above sublinear framework is constrained to the classical linear
case, by means of the solution regularity of PIDE (\ref{PIDE_Erhan}), it can
be verified that condition (ii) holds as long as $\beta_{i}$, $i=1,2$, in
(\ref{normal attraction of stable law}) are continuously differentiable on
their respective closed half-lines, see \cite{BM2016} or Remark \ref{Remark}.
\end{remark}

\section{Main results}

We first recall the definition of nonlinear L\'{e}vy process under sublinear
expectations as introduced in \cite{HP2021} and \cite{NN2017}.

\begin{definition}
A $d$-dimensional c\`{a}dl\`{a}g process $(X_{t})_{t\geq0}$ defined on a
sublinear expectation space $(\Omega,\mathcal{H},\mathbb{\hat{E})}$ is called
a nonlinear L\'evy process if the following properties hold.

\begin{description}
\item[(i)] $X_{0}=0$;

\item[(ii)] $(X_{t})_{t\geq0}$ has stationary increments, that is,
$X_{t}-X_{s}$ and $X_{t-s}$ are identically distributed for all $0\leq s\leq
t;$

\item[(iii)] $(X_{t})_{t\geq0}$ has independent increments, that is,
$X_{t}-X_{s}$ is independent from $(X_{t_{1}},\ldots,X_{t_{n}})$ for each
$n\in \mathbb{N}$ and $0\leq t_{1}\leq \cdots \leq t_{n}\leq s\leq t$.
\end{description}
\end{definition}

A nonlinear L\'{e}vy process is characterized via a set of L\'{e}vy triplets
$(F_{\mu},q,Q)$, where the first component $F_{\mu}$ is a L\'{e}vy measure
describing the jump uncertainty, and the second and third components
$(q,Q)\in \mathbb{R}^{d}\times \mathbb{S}_{+}(d)$ describe the mean and
volatility uncertainty. We call such a set an \emph{uncertainty set}
throughout the paper. Hu and Peng \cite{HP2009,HP2021} first proved that a
nonlinear L\'{e}vy process (with finite activity jumps) must admit an
uncertainty set in the spirit of L\'{e}vy-Khintchine representation. However,
a L\'{e}vy-Khintchine representation formula for a nonlinear L\'{e}vy process
(with infinite activity jumps) is still lacking as commented on Remark 23 in
\cite{HP2021} and Page 71 in \cite{NN2017}. It turns out such a representation
is crucial for the universal robust limit theorem.

We now introduce the universal robust limit theorem for a nonlinear L\'{e}vy
process. Let $\alpha \in(1,2)$, $(\underline{\Lambda},\overline{\Lambda})$ for
some $\underline{\Lambda},\overline{\Lambda}>0$, and $F_{\mu}$ be the $\alpha
$-stable L\'{e}vy measure on $(\mathbb{R}^{d},\mathcal{B}(\mathbb{R}^{d}))$,
\begin{equation}
F_{\mu}(B)=\int_{S}\mu(dz)\int_{0}^{\infty}\mathbbm{1}_{B}(rz)\frac
{dr}{r^{1+\alpha}},\text{ \ for }B\in \mathcal{B}(\mathbb{R}^{d}), \label{F_mu}%
\end{equation}
where $\mu$ is a finite measure on the unit sphere $S=\{z\in \mathbb{R}%
^{d}:|z|=1\}$. Set
\begin{equation}
\mathcal{L}_{0}=\left \{  F_{\mu}\  \text{measure on }\mathbb{R}^{d}:\mu
(S)\in(\underline{\Lambda},\overline{\Lambda})\right \}  , \label{L_0}%
\end{equation}
and $\mathcal{L\subset L}_{0}$ as a nonempty compact convex set.
%\begin{equation}  \label{estimate_1}
%\mathcal{K}:=\sup \limits_{F_{\mu}\in \mathcal{L}}\int_{\mathbb{R}^{d}}
%|z|\wedge|z|^{2}F_{\mu}(dz)<\infty,
%\end{equation}
%and
%\begin{equation}  \label{estimate_2}
%\lim_{\varepsilon \rightarrow0}\mathcal{K}_{\varepsilon}=0\text{ \ for \ }
%\mathcal{K}_{\varepsilon}:=\sup \limits_{F_{\mu}\in \mathcal{L}}\int
%_{|z|\leq \varepsilon}|z|^{2}F_{\mu}(dz).
%\end{equation}
Let $\{(X_{i},Y_{i},Z_{i})\}_{i=1}^{\infty}$ be an i.i.d. sequence of
$\mathbb{R}^{3d}$-valued random variables on a sublinear expectation space
$(\Omega,\mathcal{H},\mathbb{\hat{E}})$ in the sense that $(X_{i+1}%
,Y_{i+1},Z_{i+1})$ $\overset{d}{=}(X_{i},Y_{i},Z_{i})$ and $(X_{i+1}%
,Y_{i+1},Z_{i+1})$ is independent from $(X_{1},Y_{1},Z_{1}),\ldots
,(X_{i},Y_{i},Z_{i})$ for each $i\in \mathbb{N}$. Set
\[%
\begin{array}
[c]{rrr}%
S_{n}^{1}:=\sum \limits_{i=1}^{n}X_{i},\text{ } & S_{n}^{2}:=\sum
\limits_{i=1}^{n}Y_{i},\text{ } & S_{n}^{3}:=\sum \limits_{i=1}^{n}Z_{i}.
\end{array}
\]

We impose the following assumptions throughout the paper. The first two are
moment conditions on $(X_{1},Y_{1}, Z_{1})$ and the last one is a consistency
condition on $Z_{1}$.

\begin{description}
\item[(A1)] $\mathbb{\hat{E}}[X_{1}]=\mathbb{\hat{E}}[-X_{1}]=0$,
$\lim \limits_{\gamma \rightarrow \infty}\mathbb{\hat{E}}[(|X_{1}|^{2}-
\gamma)^{+}]=0$, and $\lim \limits_{\gamma \rightarrow \infty}\mathbb{\hat{E}
}[(|Y_{1}|-\gamma)^{+}]=0.$

\item[(A2)] $\mathbb{\hat{E}}[Z_{1}]=\mathbb{\hat{E}}[-Z_{1}]=0$ and
$M_{z}:=\sup \limits_{n}\mathbb{\hat{E}}[n^{-\frac{1}{\alpha} }|S_{n}%
^{3}|]<\infty$.

\item[(A3)] For each $\varphi \in C_{b}^{3}(\mathbb{R}^{d})$, the space of
functions on $\mathbb{R}^{d}$ with uniformly bounded derivatives up to the
order $3$, satisfies
\[
\frac{1}{s}\bigg \vert \mathbb{\hat{E}}\big[\varphi(z+s^{\frac{1}{\alpha}%
}Z_{1})-\varphi(z)\big]-s\sup \limits_{F_{\mu}\in \mathcal{L}}\int
_{\mathbb{R}^{d}}\delta_{\lambda}\varphi(z)F_{\mu}(d\lambda)\bigg \vert \leq
l(s)\rightarrow0
\]
uniformly on $z\in \mathbb{R}^{d}$ as $s\rightarrow0$, where $l$ is a function
on $[0,1]$ and $\delta_{\lambda}\varphi(z):=\varphi(z+\lambda)-\varphi
(z)-\langle D\varphi(z),\lambda \rangle$.
\end{description}

\begin{remark}
\label{Remark0} The assumption (A3) is essentially a consistency condition for
the distribution of the multidimensional $Z$, which has been exploited
successfully in the numerical analysis literature on the monotone
approximation schemes for nonlinear PDEs \cite{BJ2002,BJ2005,BJ2007,BS1991}.
In the one-dimensional case, the assumption (A3) is closely related to the
consistency condition proposed in Bayraktar and Munk \cite{BM2016} (see (ii)
in Theorem \ref{Theorem_Erhan}). However, they require that the solution
$v(t,x)$ of PIDE (\ref{PIDE_Erhan}) in prior satisfies the consistency
condition, whereas we only require the consistency condition (A3) holds
without involving the solution $v(t,x)$.
\end{remark}

\begin{remark}
\label{Remark} Although the assumption (A3) looks obscure, let us show that
when our attention is confined to the classical case, it turns out to be mild
and is more general than the characterization condition
(\ref{normal attraction of stable law}). We consider the two-dimensional case
in the following. The one-dimensional case can be found in \cite{BM2016}.

Let $\zeta=(\zeta^{1},\zeta^{2})$ be a classical $\alpha$-stable random
variable with $\zeta^{1}\perp \! \! \! \perp \zeta^{2}$ and L\'evy triplet
$(F_{\mu},0,0)$. From Samorodnitsky and Taqqu \cite[Example 2.3.5]{ST1994},
the finite measure $\mu$ in $F_{\mu}$ is discrete and concentrated on the
points $(1,0),(-1,0),(0,1)$, and $(0,-1)$. Denote
\[%
\begin{array}
[c]{llll}%
k_{1}^{1}=\mu((-1,0)), & k_{2}^{1}=\mu((1,0)), & k_{1}^{2}=\mu((0,-1)), &
k_{2}^{2}=\mu((0,1)).
\end{array}
\]
For $i\geq1$, let $Z_{i}=(Z_{i}^{1},Z_{i}^{2})$ be a zero mean classical
random variable with $Z_{i}^{1}\perp \! \! \! \perp Z_{i}^{2}$. Theorem
\ref{classical_condition} indicates that the i.i.d. sequence $\{Z_{i}%
\}_{i=1}^{\infty}$ is in the domain of normal attraction of an $\alpha$-stable
distribution $\zeta$ if and only if for $m=1,2$, $F_{Z_{1}^{m}}$ has the
cumulative distribution function
\[
F_{Z_{1}^{m}}(z)=\left \{
\begin{array}
[c]{ll}%
\displaystyle \left[  k_{1}^{m}/\alpha+\beta_{1}^{m}(z)\right]  \frac{1}{
|z|^{\alpha}}, & z<0,\\
\displaystyle1-\left[  k_{2}^{m}/\alpha+\beta_{2}^{m}(z)\right]  \frac{1}{
z^{\alpha}}, & z>0,
\end{array}
\right.
\]
where $\beta_{1}^{m}:$ $(-\infty,0]$ $\rightarrow \mathbb{R}$ and $\beta
_{2}^{m}:[0,\infty)\rightarrow \mathbb{R}$ are functions satisfying
\[
\lim_{z\rightarrow-\infty}\beta_{1}^{m}(z)=\lim_{z\rightarrow \infty}\beta
_{2}^{m}(z)=0.
\]

We further assume that $\beta_{1}^{m}$ and $\beta_{2}^{m}$, $m=1,2$, are
continuously differentiable functions defined on $(-\infty,0]$ and
$[0,\infty)$, respectively. It can be verified that $E[|Z_{1}^{m}|]<\infty$,
$m=1,2$. Moreover, for $\varphi \in C_{b}^{3}(\mathbb{R}^{2})$, we note that
\begin{align*}
E\big[\varphi(z+s^{\frac{1}{\alpha}}Z_{1})-\varphi(z)\big]  &  =E\big[\varphi
(z_{1}+s^{\frac{1}{\alpha}}Z_{1}^{1},z_{2}+s^{\frac{1}{\alpha}}Z_{1}
^{2})-\varphi(z_{1},z_{2}+s^{\frac{1}{\alpha}}Z_{1}^{2})\big]\\
&  \text{ \  \ }+E\big[\varphi(z_{1},z_{2}+s^{\frac{1}{\alpha}}Z_{1}
^{2})-\varphi(z_{1},z_{2})\big]\\
&  =E\big[\delta_{s^{1/\alpha}Z_{1}^{1}}^{1}\varphi(z_{1},z_{2}+s^{\frac
{1}{\alpha}}Z_{1}^{2})\big]+E\big[\delta_{s^{1/\alpha}Z_{1}^{2}}^{2}%
\varphi(z_{1},z_{2})\big],
\end{align*}
and
\[
\int_{\mathbb{R}^{2}}\delta_{\lambda}\varphi(z)F_{\mu}(d\lambda)=\int
_{\mathbb{R}}\delta_{\lambda_{1}}^{1}\varphi(z_{1},z_{2})F_{\mu}^{1}
(d\lambda_{1})+\int_{\mathbb{R}}\delta_{\lambda_{2}}^{2}\varphi(z_{1}
,z_{2})F_{\mu}^{2}(d\lambda_{2}),
\]
where
\begin{align*}
\delta_{\gamma}^{1}\varphi(x_{1},x_{2})  &  :=\varphi(x_{1}+\gamma
,x_{2})-\varphi(x_{1},x_{2})-D_{1}\varphi(x_{1},x_{2})\gamma,\\
\delta_{\gamma}^{2}\varphi(x_{1},x_{2})  &  :=\varphi(x_{1},x_{2}
+\gamma)-\varphi(x_{1},x_{2})-D_{2}\varphi(x_{1},x_{2})\gamma,
\end{align*}
and
\[
F_{\mu}^{m}(d\lambda_{m}):=\frac{k_{1}^{m}}{|\lambda_{m}|^{\alpha+1}
}\mathbbm{1}_{(-\infty,0)}(\lambda_{m})d\lambda_{m}+\frac{k_{2}^{m}}
{|\lambda_{m}|^{\alpha+1}}\mathbbm{1}_{(0,\infty)}(\lambda_{m})d\lambda
_{m},\text{ \ }m=1,2.
\]
Then, it follows that
\begin{align*}
&  \frac{1}{s}\bigg \vert E\big[\varphi(z+s^{\frac{1}{\alpha}}Z_{1}
)-\varphi(z)\big]-s\int_{\mathbb{R}^{2}}\delta_{\lambda}\varphi(z)F_{\mu
}(d\lambda)\bigg \vert \\
&  \leq \frac{1}{s}\bigg \vert E\big[\delta_{s^{1/\alpha}Z_{1}^{1}}^{1}
\varphi(z_{1},z_{2}+s^{\frac{1}{\alpha}}Z_{1}^{2})\big]-s\int_{\mathbb{R}
}\delta_{\lambda_{1}}^{1}\varphi(z_{1},z_{2})F_{\mu}^{1}(d\lambda
_{1})\bigg \vert \\
&  \text{ \  \ }+\frac{1}{s}\bigg \vert E\big[\delta_{s^{1/\alpha}Z_{1}^{2}
}^{2}\varphi(z_{1},z_{2})\big]-s\int_{\mathbb{R}}\delta_{\lambda_{2}}
^{2}\varphi(z_{1},z_{2})F_{\mu}^{2}(d\lambda_{2})\bigg \vert \\
&  :=I+II.
\end{align*}
In view of Lemma \ref{Lipschitz}, we get
\begin{align*}
I  &  =\frac{1}{s}\bigg \vert E\bigg[\bigg(E\big[\delta_{s^{1/\alpha}Z_{1}%
^{1}}^{1}\varphi(z_{1},z_{2}+x_{2})\big]-s\int_{\mathbb{R}}\delta_{\lambda
_{1}}^{1}\varphi(z_{1},z_{2}+x_{2})F_{\mu}^{1}(d\lambda_{1})\\
&  \text{ \  \ }+s\int_{\mathbb{R}}\left(  \delta_{\lambda_{1}}^{1}%
\varphi(z_{1},z_{2}+x_{2})-\delta_{\lambda_{1}}^{1}\varphi(z_{1}%
,x_{2})\right)  F_{\mu}^{1}(d\lambda_{1})\bigg)_{x_{2}=s^{1/\alpha}Z_{1}^{2}%
}\bigg]\bigg \vert \\
&  \leq Cs^{\frac{1}{\alpha}}E[|Z_{1}^{2}|]\\
&  \  \  \ +\frac{1}{s}E\bigg[\bigg \vert E\big[\delta_{s^{1/\alpha}Z_{1}^{1}%
}^{1}\varphi(z_{1},z_{2}+x_{2})\big]-s\int_{\mathbb{R}}\delta_{\lambda_{1}%
}^{1}\varphi(z_{1},z_{2}+x_{2})F_{\mu}^{1}(d\lambda_{1})\bigg \vert_{x_{2}%
=s^{1/\alpha}Z_{1}^{2}}\bigg].
\end{align*}
Following along similar arguments as in (3.4)-(3.8) in \cite{BM2016}, we
obtain that%
\[
\frac{1}{s}\bigg \vert E\big[\delta_{s^{1/\alpha}Z_{1}^{1}}^{1}\varphi
(z_{1},z_{2}+x_{2})\big]-s\int_{\mathbb{R}}\delta_{\lambda_{1}}^{1}%
\varphi(z_{1},z_{2}+x_{2})F_{\mu}^{1}(d\lambda_{1})\bigg \vert \rightarrow0
\]
uniformly on $(z_{1},z_{2}+x_{2})\in \mathbb{R}^{2}$ as $s\rightarrow0$, and
similarly, the part $II\ $converges to 0 uniformly in $(z_{1},z_{2}%
)\in \mathbb{R}^{2}\ $as $s\rightarrow0$. Thus, the assumption (A3) holds.
\end{remark}

We are ready to state our main result of this paper, which is dubbed as
\emph{a universal robust limit theorem} under sublinear expectation. It covers
all the existing robust limit theorems in the literature, namely, Peng's
robust central limit theorem for $G$-distribution (Theorem \ref{LLT+CLT}) and
Bayraktar-Munk's robust limit theorem for $\alpha$-stable distribution (see
\cite{BM2016}).

\begin{theorem}
\label{main theorem} Suppose that assumptions (A1)-(A3) hold. Then, there
exists a nonlinear L\'evy process $\tilde{L}_{t}=(\tilde{\xi}_{t},\tilde{\eta
}_{t},\tilde{\zeta}_{t})$, $t\in[0,1]$, associated with an uncertainty set
$\Theta \subset \mathcal{L}\times \mathbb{R}^{d}\times \mathbb{S}_{+}(d)$
satisfying
\[
\sup_{(F_{\mu},q,Q)\in \Theta}\left \{  \int_{\mathbb{R}^{d}}|z|\wedge
|z|^{2}F_{\mu}(dz)+|q|+|Q|\right \}  <\infty,
\]
such that for any $\phi \in C_{b,Lip}(\mathbb{R}^{3d})$,
\[
\lim_{n\rightarrow \infty}\mathbb{\hat{E}}\left[  \phi \left(  \frac{S_{n}^{1}
}{\sqrt{n}},\frac{S_{n}^{2}}{n},\frac{S_{n}^{3}}{\sqrt[\alpha]{n}}\right)
\right]  = \mathbb{\tilde{E}}[\phi(\tilde{\xi}_{1},\tilde{\eta}_{1},
\tilde{\zeta}_{1})]=u^{\phi}(1,0,0,0),
\]
where $u^{\phi}$ is the unique viscosity solution of the following fully
nonlinear PIDE
\begin{equation}
\left \{
\begin{array}
[c]{l}%
\displaystyle \partial_{t}u(t,x,y,z)-\sup \limits_{(F_{\mu},q,Q)\in \Theta
}\left \{  \int_{\mathbb{R}^{d}}\delta_{\lambda}u(t,x,y,z)F_{\mu}(d\lambda)
\right. \\
\displaystyle \text{\  \  \  \  \  \  \  \  \  \  \  \  \  \  \ }\left.  +\langle
D_{y}u(t,x,y,z),q\rangle+\frac{1}{2}tr[D_{x}^{2}u(t,x,y,z)Q]\right \}  =0,\\
\displaystyle u(0,x,y,z)=\phi(x,y,z),\text{ \ }\forall(t,x,y,z)\in
[0,1]\times \mathbb{R}^{3d},
\end{array}
\right.  \label{PIDE}%
\end{equation}
with $\delta_{\lambda}u(t,x,y,z)=u(t,x,y,z+\lambda)-u(t,x,y,z)-\langle
D_{z}u(t,x,y,z),\lambda \rangle$.
%\begin{equation}
%\left \{
%\begin{array}{l}
%\displaystyle \partial_{t}u(t,x,y,z)-\tilde{G}(u(t,x,y,z+
%\cdot),D_{y}u(t,x,y,z),D_{x}^{2} u(t,x,y,z))=0, \\
%\displaystyle u(0,x,y,z)=\phi(x,y,z),\text{ \ }\forall(t,x,y,z)\in
%\lbrack0,1]\times \mathbb{R}^{d}\times \mathbb{R}^{d}\times \mathbb{R}^{d},
%\end{array}
%\right.   \label{PIDE}
%\end{equation}
%with $\tilde{G}: C_b^3(\mathbb{R}^d)\times\mathbb{R}^{d}\times \mathbb{S}
%(d)\rightarrow \mathbb{R}$ defined by
%\begin{align*}
%\tilde{G}(f(\cdot),p,A):=\sup \limits_{(F_{\mu},q,Q)\in \Theta }\left \{ \int_{
%\mathbb{R}^{d}}\delta_{\lambda}f(\lambda)F_{\mu}(d\lambda)+\langle
%p,q\rangle+\frac{1}{2}tr[AQ]\right \},
%\end{align*}
%where
%\begin{align*}
%\delta_{\lambda}f(\lambda):=f(\lambda)-f(0)-\langle D_zf(0),\lambda \rangle,
%\end{align*}
%and $\Theta \subset \mathcal{L}\times \mathbb{R}^{d}\times \mathbb{S}_{+}
%\mathbb{(}d\mathbb{)}$ satisfies
%\begin{equation*}
%\sup_{(F_{\mu},q,Q)\in \Theta}\left \{ \int_{\mathbb{R}^{d}}|z|\wedge
%|z|^{2}F_{\mu}(dz)+|q|+|Q|\right \} <\infty.
%\end{equation*}

\end{theorem}

\begin{proof}
We outline the main steps below, with the detailed proof provided in Section 4.

\textbf{(i)} By using the notions of tightness and weak compactness, we first
construct a nonlinear L\'evy process $\tilde{L}_{t}:=(\tilde{ \xi}_{t},
\tilde{\eta}_{t},\tilde{\zeta}_{t})$, $t\in[0,1]$, on some sublinear
expectation space $(\tilde{\Omega},Lip(\tilde{\Omega}),\mathbb{\tilde{E}})$ in
Section \ref{Section_The construction of Levy process}, which is generated by
the weak convergence limit of the sequences $\big \{ \big(
(t/n)^{1/2}S_{n}^{1},(t/n)S_{n}^{2},(t/n)^{1/\alpha}S_{n}^{3}\big)\big \}
_{n=1}^{\infty}$ for $t\in \lbrack0,1]$.

\textbf{(ii)} To link the nonlinear L\'evy process $(\tilde{L}_{t}
)_{t\in \lbrack0,1]}$ with the fully nonlinear PIDE (\ref{PIDE}), a key step is
to give the characterization of
\[
\lim \limits_{\delta \rightarrow0}\mathbb{\tilde{E}}[\varphi(\tilde{\zeta}
_{\delta})+\langle p,\tilde{\eta}_{\delta}\rangle+\frac{1}{2}\langle A
\tilde{\xi}_{\delta},\tilde{\xi}_{\delta}\rangle]\delta^{-1}%
\]
for $\varphi \in C_{b}^{3}(\mathbb{R}^{d})$ with $\varphi(0)=0$ and
$(p,A)\in \mathbb{R}^{d}\times \mathbb{S}(d)$. It follows from a new estimate
for the $\alpha$-stable L\'{e}vy measure and a new L\'evy-Khintchine
representation formula for the nonlinear L\'{e}vy process in Sections
\ref{Section_Estimates for stable Levy measure} and
\ref{Section_Representation of Levy process}, respectively.

\textbf{(iii)} Once the representation of the nonlinear L\'evy process is
established, with the help of nonlinear stochastic analysis techniques and
viscosity solution methods, Theorem \ref{main theorem} is a consequence of the
dynamic programming principle in Section \ref{Section_connection PIDE} by
defining $u(t,x,y,z)=\mathbb{\tilde{E}}[\phi(x+\tilde{\xi} _{t},y+\tilde{
\eta}_{t},z+\tilde{\zeta}_{t})]$, for $(t,x,y,z)\in[0,1]\times \mathbb{R}^{3d}$.
\end{proof}

The following corollary can be readily obtained from Theorem
\ref{main theorem}.

\begin{corollary}
Suppose that assumptions in Theorem \ref{main theorem} hold. Then, for any
$\phi \in C_{b,Lip}(\mathbb{R}^{d})$,
\[
\lim_{n\rightarrow \infty}\mathbb{\hat{E}}\left[  \phi \left(  \frac{S_{n}^{1}
}{\sqrt{n}}+\frac{S_{n}^{2}}{n}+\frac{S_{n}^{3}}{\sqrt[\alpha]{n}}\right)
\right]  =u^{\phi}(1,0),
\]
where $u^{\phi}$ is the unique viscosity solution of the following fully
nonlinear PIDE
\[
\left \{
\begin{array}
[c]{l}%
\displaystyle \partial_{t}u(t,x)-\sup \limits_{(F_{\mu},q,Q)\in \Theta}\left \{
\int_{\mathbb{R}^{d}}\delta_{\lambda}u(t,x)F_{\mu}(d\lambda) +\langle
D_{x}u(t,x),q\rangle+\frac{1}{2}tr[D_{x}^{2}u(t,x)Q]\right \}  =0,\\
\displaystyle u(0,x)=\phi(x),\text{ \ }\forall(t,x)\in[0,1]\times \mathbb{R}
^{d},
\end{array}
\right.
\]
where $\delta_{\lambda}u(t,x)=u(t,x+\lambda)-u(t,x)-\langle D_{x}%
u(t,x),\lambda \rangle$ and the uncertainty set $\Theta$ defined in Theorem
\ref{main theorem}.
%with $\delta_{\lambda}u(t,x)=u(t,x+\lambda)-u(t,x)-\langle Du(t,x),\lambda
%\rangle$ and $\Theta \subset \mathcal{L}\times \mathbb{R}^{d}\times \mathbb{S}
%_{+}\mathbb{(}d\mathbb{)}$ satisfying
%\[
%\sup_{(F_{\mu},q,Q)\in \Theta}\left \{  \int_{\mathbb{R}^{d}}|z|\wedge
%|z|^{2}F_{\mu}(dz)+|q|+|Q|\right \}  <\infty.
%\]

\end{corollary}

\begin{proof}
For any $\phi \in C_{b,Lip}(\mathbb{R}^{d})$, define $\tilde{\phi}%
(x,y,z):=\phi(x+y+z)\in C_{b,Lip}(\mathbb{R}^{3d})$, for $(x,y,z)\in
\mathbb{R}^{3d}$. From Theorem \ref{main theorem} we know that
$v(t,x,y,z):=\mathbb{\tilde{E}}[\tilde{\phi}(x+\tilde{\xi}_{t},y+\tilde{\eta
}_{t},z+\tilde{\zeta}_{t})]$ is the unique viscosity solution of the PIDE
(\ref{PIDE}). Set $u(t,x):=\mathbb{\tilde{E}}[\phi(x+\tilde{\xi}_{t}%
+\tilde{\eta}_{t}+ \tilde{\zeta}_{t})]$, for $(t,x)\in \lbrack0,1]\times
\mathbb{R}^{d}$. Noting that $u(t,x+y+z)=v(t,x,y,z)$, $\partial_{t}%
u=\partial_{t}v$, $D_{x}u=D_{y}v=D_{z}v$, and $D_{x}^{2}u=D_{x}^{2}v$, we
conclude the result.
\end{proof}

%\begin{corollary}
%Suppose that assumptions in Theorem \ref{main theorem} hold and
%$(X_1,Y_1)\perp \!\!\! \perp Z_1$. Then, for any $\phi \in C_{b,Lip}(\mathbb{R}^{3d})$,
%\begin{equation*}
%\lim_{n\rightarrow \infty}\mathbb{\hat{E}}\left[ \phi \left( \frac{S_{n}^{1}%
%}{\sqrt{n}},\frac{S_{n}^{2}}{n},\frac{S_{n}^{3}}{\sqrt[\alpha]{n}}\right) %
%\right] =u^{\phi}(1,0,0,0),
%\end{equation*}
%where $u^{\phi}\ $is the unique viscosity solution of the following fully
%nonlinear PIDE:
%\textcolor[rgb]{1,0,0}{\begin{equation*}
%\left \{
%\begin{array}{l}
%\displaystyle \partial_{t}u(t,x)-\sup \limits_{F_{\mu}\in \mathcal{L} }\left
%\{ \int_{\mathbb{R}^{d}}\delta_{\lambda}u(t,x,y,z)F_{\mu} (d\lambda)\right\}
%\\
%\ \ \ \ \ \ \ \ \ \ \ -G(D_y u(t,x,y,z), D_x^2u(t,x,y,z))=0, \\
%\displaystyle u(0,x,y,z)=\phi(x,y,z),\text{ \ }(t,x,y,z)\in \lbrack
%0,1]\times \mathbb{R}^{d}\times \mathbb{R}^{d}\times \mathbb{R}^{d},%
%\end{array}
%\right.
%\end{equation*}}
%with $\delta_{\lambda}u(t,x,y,z):=u(t,x,y,z+\lambda)-u(t,x,y,z)-\langle
%D_{z}u(t,x,y,z),\lambda \rangle$ and $G:$ $\mathbb{R}^{d}\times \mathbb{S}(d)
%$ $\rightarrow \mathbb{R}$ given in Remark \ref{remark_G}.
%\end{corollary}
%\begin{proof}
%\textcolor[rgb]{1.00,0.00,0.00}{???Give a sketch proof-change position???}
%\end{proof}

The following corollary extends the $\alpha$-stable limit theorem for $\alpha
$-stable distribution under sublinear expectation in Bayraktar and Munk
\cite[Theorem 3.1]{BM2016} from one-dimensional to multidimensional case under
a weaker consistency condition.

\begin{corollary}
\label{corollary atable}Suppose that assumptions (A2)-(A3) hold. Then, there
exists a nonlinear L\'{e}vy process $(\tilde{\zeta}_{t})_{t\in \lbrack0,1]}$
associated with an uncertainty set $\Theta=\{(F_{\mu},0,0):F_{\mu}%
\in \mathcal{L}\}$, such that for any $\phi \in C_{b,Lip}(\mathbb{R}^{d})$,
\[
\lim_{n\rightarrow \infty}\mathbb{\hat{E}}\left[  \phi \left(  \frac{S_{n}^{3}%
}{\sqrt[\alpha]{n}}\right)  \right]  =\mathbb{\tilde{E}}[\phi(\tilde{\zeta
}_{1})]=u^{\phi}(1,0),
\]
where $u^{\phi}$ is the unique viscosity solution of the following fully
nonlinear PIDE
\begin{equation}
\left \{
\begin{array}
[c]{l}%
\displaystyle \partial_{t}u(t,x)-\sup \limits_{F_{\mu}\in \mathcal{L}}\left \{
\int_{\mathbb{R}^{d}}\delta_{\lambda}u(t,x)F_{\mu}(d\lambda)\right \}  =0,\\
\displaystyle u(0,x)=\phi(x),\text{ \ }\forall(t,x)\in \lbrack0,1]\times
\mathbb{R}^{d},
\end{array}
\right.  \label{PIDE-jump}%
\end{equation}
where $\delta_{\lambda}u(t,x)=u(t,x+\lambda)-u(t,x)-\langle D_{x}%
u(t,x),\lambda \rangle$. In fact, $\tilde{\zeta}$ is a nonlinear $\alpha
$-stable process satisfying a scaling property, that is, $\tilde{\zeta}_{\beta
t}$ and $\beta^{1/\alpha}\tilde{\zeta}_{t}$ are identically distributed, for
any $0<\beta<1$ and $0\leq t\leq1$.
\end{corollary}

\begin{proof}
In light of Theorem \ref{main theorem}, it remains to show that $\tilde{\zeta
}$ satisfies the scaling property. For any given $\phi \in C_{b,Lip}%
(\mathbb{R}^{d})$, Theorem \ref{unique viscosity theorem} implies that
$u(\beta t,0)=\mathbb{\tilde{E}}[\phi(\tilde{\zeta}_{\beta t})]$, where $u$ is
the unique viscosity solution of the PIDE (\ref{PIDE-jump}) with initial
condition $\phi$. Note that for every $\beta>0$,
\[
F_{\mu}(B)=\beta F_{\mu}(\beta^{1/\alpha}B),\text{ \ for }B\in \mathcal{B}%
(\mathbb{R}^{d}).
\]
For any given $0<\beta<1$ and $0\leq t\leq1$, define $v(t,x):=u(\beta
t,\beta^{1/\alpha}x)$. It follows from
\[
\beta \int_{\mathbb{R}^{d}}\delta_{\lambda}u(\beta t,\beta^{1/\alpha}x)F_{\mu
}(d\lambda)=\int_{\mathbb{R}^{d}}\delta_{\lambda}v(t,x)F_{\mu}(d\lambda)
\]
that $v$ is the unique viscosity solution of the PIDE (\ref{PIDE-jump}) with
initial condition $\tilde{\phi}(x):=\phi(\beta^{1/\alpha}x)$. From Theorem
\ref{unique viscosity theorem}, we derive that $v(t,0)=\mathbb{\tilde{E}%
}[\tilde{\phi}(\tilde{\zeta}_{t})]$. Therefore,
\[
\mathbb{\tilde{E}}[\phi(\tilde{\zeta}_{\beta t})]=u(\beta t,0)=v(t,0)=
\mathbb{\tilde{E}}[\tilde{\phi}(\tilde{\zeta}_{t})]=\mathbb{\tilde{E}}%
[\phi(\beta^{1/\alpha}\tilde{\zeta}_{t})],
\]
and the proof is complete.
\end{proof}

\section{Proof of Theorem \ref{main theorem}}

\subsection{The construction of nonlinear L\'{e}vy process}

\label{Section_The construction of Levy process}

Let $\tilde{\Omega}=C_{0}^{d}[0,1]\times C_{0}^{d}[0,1]\times D_{0}^{d}[0,1]$
be the space of all $\mathbb{R}^{d}\times \mathbb{R}^{d}\times \mathbb{R}^{d}%
$-valued paths $(\omega_{t})_{t\in \lbrack0,1]}$ with $\omega_{0}=0$, equipped
with the Skorohod topology, where $C_{0}^{d}[0,1]$ is the space of
$\mathbb{R}^{d}$-valued continuous paths and $D_{0}^{d}[0,1]$ is the space of
$\mathbb{R}^{d}$-valued c\`{a}dl\`{a}g paths. Consider the canonical process
$(\tilde{\xi}_{t},\tilde{\eta}_{t},\tilde{\zeta}_{t})(\omega)=(\omega_{t}
^{1},\omega_{t}^{2},\omega_{t}^{3})$, $t\in \lbrack0,1]$, for $\omega
=(\omega^{1},\omega^{2},\omega^{3})\in \tilde{\Omega}$. Set $\tilde{L}%
_{t}:=(\tilde{\xi}_{t},\tilde{\eta}_{t},\tilde{\zeta}_{t})$ and
\[
Lip(\tilde{\Omega})=\left \{  \varphi(\tilde{L}_{t_{1}},\ldots,\tilde{L}
_{t_{n}}-\tilde{L} _{t_{n-1}}):\forall0\leq t_{1}<t_{2}<\cdots<t_{n}
\leq1,\varphi \in C_{b,Lip}(\mathbb{R}^{n\times3d})\right \}  .
\]
%%\[
%%Lip(\tilde{\Omega})=\{ \varphi(\tilde{L}_{t_{1}},\ldots,\tilde{L}_{t_{m}%
%%}):m\geq1,t_{1},\ldots,t_{m}\in \lbrack0,1],\varphi \in C_{b,Lip}(\mathbb{R}%
%%^{m\times3d})\} \text{.}%
%%\]%

\begin{theorem}
\label{The construction of Levy process}Assume that (A1)-(A2) hold. Then,
there exists a sublinear expectation $\mathbb{\tilde{E}}$ on $(\tilde{\Omega
},Lip(\tilde{\Omega}))$ such that the sequence $\{(n^{-1/2}S_{n}^{1}%
,n^{-1}S_{n}^{2},n^{-1/\alpha}S_{n}^{3})\}_{n=1}^{\infty}$ converges in
distribution to $\tilde{L}_{1}$, where $(\tilde{L}_{t})_{t\in \lbrack0,1]}$ is
a nonlinear L\'evy process on $(\tilde{\Omega},Lip(\tilde{\Omega}),
\mathbb{\tilde{E}})$.
\end{theorem}

\begin{remark}
Theorem \ref{The construction of Levy process} can be regarded as a Donsker
theorem for the nonlinear L\'evy process $(\tilde{L}_{t})_{t\in \lbrack0,1]}$.
\end{remark}

The proof of the above theorem depends on the following lemma.

\begin{lemma}
\label{tight}Assume that (A1)-(A2) hold. For $\phi \in C_{b,Lip}(\mathbb{R}%
^{3d})$, let
\[
\mathbb{\hat{F}}[\phi]:=\sup_{n}\mathbb{\hat{E}}\left[  \phi \left(
\frac{S_{n}^{1}}{\sqrt{n}},\frac{S_{n}^{2}}{n},\frac{S_{n}^{3}}{\sqrt[\alpha
]{n}}\right)  \right]  .
\]
Then, the sublinear expectation $\mathbb{\hat{F}}$ on $(\mathbb{R}%
^{3d},C_{b,Lip}(\mathbb{R}^{3d}))$ is tight.
\end{lemma}

\begin{proof}
It is clear that $\mathbb{\hat{F}}$ is a sublinear expectation on
$(\mathbb{R}^{3d},C_{b,Lip}(\mathbb{R}^{3d}))$. Now we show that
$\mathbb{\hat{F}}$ is tight. For any $N>0$, we define
\[
\varphi_{N}(x)=\left \{
\begin{array}
[c]{ll}%
1, & |x|>N,\\
|x|-N+1, & N-1\leq|x|\leq N,\\
0, & |x|<N-1.
\end{array}
\right.
\]
One can easily check that $\varphi_{N}\in C_{b,Lip}(\mathbb{R})$ and
$\mathbbm{1}_{\{|x|>N\}}\leq \varphi_{N}(x)\leq \mathbbm{1}_{\{|x|>N-1\}}$.
Denote $\tilde{\varphi}_{N}(x,y)=\varphi_{N}(x)+\varphi_{N}(y)$. Under the
assumption (A1), from Theorem \ref{LLT+CLT}, we get
\[
\lim_{n\rightarrow \infty}\mathbb{\hat{E}}\left[  \tilde{\varphi}_{N}\left(
\frac{S_{n}^{1}}{\sqrt{n}},\frac{S_{n}^{2}}{n}\right)  \right]  =\mathbb{\hat
{E}}_{1}\big[\tilde{\varphi}_{N}(\xi_{1},\eta_{1})\big],
\]
where $(\xi_{1},\eta_{1})$ is $G$-distributed under another sublinear
expectation $\mathbb{\hat{E}}_{1}$ (possibly different from $\mathbb{\hat{E}}%
$). Noting that $\tilde{\varphi}_{N}\downarrow0$ as $N\rightarrow \infty$, by
Lemma 1.3.4 in Peng [31], we obtain that $\mathbb{\hat{E}}_{1}\big[\tilde
{\varphi}_{N}(\xi_{1},\eta_{1})\big]\downarrow0$ as $N\rightarrow \infty$. So
for each $\varepsilon>0$, there exists large $N_{0}$, such that $\mathbb{\hat
{E}}_{1}\big[\tilde{\varphi}_{N_{0}}(\xi_{1},\eta_{1})\big]<\varepsilon/4$.
Then, we find some large $n_{0}>1$ such that for $n\geq n_{0}$, $\mathbb{\hat
{E}}\big[\tilde{\varphi}_{N_{0}}(n^{-1/2}S_{n}^{1},n^{-1}S_{n}^{2}%
)\big]<\varepsilon/2$. Since $0<\tilde{\varphi}_{N}\leq \tilde{\varphi}_{N_{0}%
}$ for any $N>N_{0}$, it follows that $\mathbb{\hat{E}}\big[\tilde{\varphi
}_{N}(n^{-1/2}S_{n}^{1},n^{-1}S_{n}^{2})\big]<\varepsilon/2$, for any
$N>N_{0}$ and $n\geq n_{0}$. In addition, note that $\tilde{\varphi}%
_{N}(x,y)\leq \frac{1}{N-1}(|x|+|y|)$, which yields that for $n<n_{0}$,
\[
\mathbb{\hat{E}}\left[  \tilde{\varphi}_{N}\left(  \frac{S_{n}^{1}}{\sqrt{n}%
},\frac{S_{n}^{2}}{n}\right)  \right]  \leq \frac{\sqrt{n_{0}}}{N-1}%
(M_{x}+M_{y}).
\]
where $M_{x}:=\mathbb{\hat{E}}[|X_{1}|]$ and $M_{y}:=\mathbb{\hat{E}}%
[|Y_{1}|]$. Thus, by choosing
\[
N>\max \big \{N_{0},2\sqrt{n_{0}}\varepsilon^{-1}(M_{x}+M_{y})+1\big \},
\]
we have $\mathbb{\hat{F}}\big[\tilde{\varphi}_{N}\big]<\varepsilon/2$. On the
other hand, under the assumption (A2), for $N>2\varepsilon^{-1}M_{z}+1$, it
follows that
\[
\mathbb{\hat{F}}[\varphi_{N}(z)]=\sup_{n}\mathbb{\hat{E}}\left[  \varphi
_{N}\left(  \frac{S_{n}^{3}}{\sqrt[\alpha]{n}}\right)  \right]  \leq \frac
{1}{N-1}M_{z}<\varepsilon/2.
\]
Observe that for any $N>0$,
\[
\mathbbm{1}_{\{ \left \vert (x,y,z)\right \vert \geq N\}}\leq \mathbbm{1}_{\{
\left \vert x\right \vert \geq N/\sqrt{3}\}}+\mathbbm{1}_{\{ \left \vert
y\right \vert \geq N/\sqrt{3}\}}+\mathbbm{1}_{\{ \left \vert z\right \vert \geq
N/\sqrt{3}\}}\leq \phi_{N/\sqrt{3}}(x,y,z),
\]
where $\phi_{N}(x,y,z):=\varphi_{N}(x)+\varphi_{N}(y)+\varphi_{N}(z)$.
Therefore, for each $\varepsilon>0$, we choose $N^{\prime}>\sqrt{3}%
\max \big \{N_{0}$, $2\sqrt{n_{0}}\varepsilon^{-1}(M_{x}+M_{y})+1$,
$2\varepsilon^{-1}M_{z}+1\big \}$ and $\phi_{N^{\prime}/\sqrt{3}}(x,y,z)\in
C_{b,Lip}(\mathbb{R}^{3d})$ with $\mathbbm{1}_{\{ \left \vert
(x,y,z)\right \vert \geq N^{\prime}\}}\leq \phi_{N^{\prime}/\sqrt{3}}(x,y,z)$
such that $\mathbb{\hat{F}}\big[\phi_{N^{\prime}/\sqrt{3}}\big]<\varepsilon$.
This proves the desired result.
\end{proof}

\begin{proof}
[Proof of Theorem \ref{The construction of Levy process}]We denote $\bar
{S}_{n}=(n^{-1/2}S_{n}^{1},n^{-1}S_{n}^{2},n^{-1/\alpha}S_{n}^{3})$. Seeing
that $\mathbb{\hat{F}}$ is tight and
\[
\mathbb{\hat{E}}[\phi(\bar{S}_{n})]-\mathbb{\hat{E}}[\phi^{\prime}(\bar{S}%
_{n})]\leq \mathbb{\hat{F}}[\phi-\phi^{\prime}],\text{ for }\phi,\phi^{\prime
}\in C_{b,Lip}(\mathbb{R}^{3d}),
\]
by Corollary \ref{corollary tight theorem}, there exists a subsequence $\{
\bar{S}_{n_{i}}\}_{i=1}^{\infty}\subset \{ \bar{S}_{n}\}_{n=1}^{\infty}$ which
converges in law to some $(\xi_{1},\eta_{1},\zeta_{1})$ in $(\Omega
,\mathcal{H},\mathbb{\hat{E}}_{1})$. By Theorem \ref{LLT+CLT}, we further know
that the marginal distribution $(\xi_{1},\eta_{1})$ is $G$-distributed. For
the above convergent subsequence $\{ \bar{S}_{n_{i}}\}_{i=1}^{\infty}$, it is
clear that for an arbitrarily increasing integers of $\{ \tilde{n}_{i}%
\}_{i=1}^{\infty}$ such that $|\tilde{n}_{i}-n_{i}|\leq1$, both $\{ \bar
{S}_{n_{i}}\}_{i=1}^{\infty}$ and $\{ \bar{S}_{\tilde{n}_{i}}\}_{i=1}^{\infty
}$ converges in law to the same limit. Thus, without loss of generality, we
assume that $n_{i}$, $i=1,2,\ldots$, are all even numbers and decompose into
two parts:
\begin{align*}
&  \bar{S}_{n_{i}}=\left(  \frac{1}{\sqrt{2}}(n_{i}/2)^{-\frac{1}{2}}%
S_{n_{i}/2}^{1},\frac{1}{2}(n_{i}/2)^{-1}S_{n_{i}/2}^{2},\frac{1}%
{\sqrt[\alpha]{2}}(n_{i}/2)^{-\frac{1}{\alpha}}S_{n_{i}/2}^{3}\right) \\
&  +\left(  \frac{1}{\sqrt{2}}(n_{i}/2)^{-\frac{1}{2}}(S_{n_{i}}^{1}%
-S_{n_{i}/2}^{1}),\frac{1}{2}(n_{i}/2)^{-1}(S_{n_{i}}^{2}-S_{n_{i}/2}%
^{2}),\frac{1}{\sqrt[\alpha]{2}}(n_{i}/2)^{-\frac{1}{\alpha}}(S_{n_{i}}%
^{3}-S_{n_{i}/2}^{3})\right) \\
&  \text{\  \  \ }:=\  \bar{S}_{n_{i}/2}^{1/2}+(\bar{S}_{n_{i}}-\bar{S}_{n_{i}%
/2}^{1/2}),
\end{align*}
where $\bar{S}_{n}^{t}:=\big((t/n)^{1/2}S_{n}^{1},(t/n)S_{n}^{2}%
,(t/n)^{1/\alpha}S_{n}^{3}\big)$ for $t\in \lbrack0,1)$. For the first part,
applying the same argument again, we prove that there exists a subsequence
$\big \{ \bar{S}_{n_{i}^{1}/2}^{1/2}\big \}_{i=1}^{\infty}$ $\subset
\big \{ \bar{S}_{n_{i}/2}^{1/2}\big \}_{i=1}^{\infty}\ $such that
$\big \{ \bar{S}_{n_{i}^{1}/2}^{1/2}\big \}_{i=1}^{\infty}$ converging in law
to $(\xi_{1/2},\eta_{1/2},\zeta_{1/2})$. Also, from Theorem \ref{LLT+CLT}, we
have $(\xi_{1/2},\eta_{1/2})\overset{d}{=}(\sqrt{1/2}\xi_{1},(1/2)\eta_{1})$.
Since $\bar{S}_{n_{i}^{1}}-\bar{S}_{n_{i}^{1}/2}^{1/2}$ is an independent copy
of $\bar{S}_{n_{i}^{1}/2}^{1/2}$, by Proposition
\ref{independent copy converge in law}, we know that
\[
\bar{S}_{n_{i}^{1}}-\bar{S}_{n_{i}^{1}/2}^{1/2}\overset{\mathcal{D}%
}{\rightarrow}(\bar{\xi}_{1/2},\bar{\eta}_{1/2},\bar{\zeta}_{1/2})\text{ \ and
\ }\bar{S}_{n_{i}^{1}}\overset{\mathcal{D}}{\rightarrow}(\xi_{1/2}+\bar{\xi
}_{1/2},\eta_{1/2}+\bar{\eta}_{1/2},\zeta_{1/2}+\bar{\zeta}_{1/2}),
\]
where $(\bar{\xi}_{1/2},\bar{\eta}_{1/2},\bar{\zeta}_{1/2})$ is an independent
copy of $(\xi_{1/2},\eta_{1/2},\zeta_{1/2})$. In addition, $\bar{S}_{n_{i}%
^{1}}\overset{\mathcal{D}}{\rightarrow}(\xi_{1},\eta_{1},\zeta_{1})$. Thus
\[
(\xi_{1},\eta_{1},\zeta_{1})\overset{d}{=}(\xi_{1/2}+\bar{\xi}_{1/2}%
,\eta_{1/2}+\bar{\eta}_{1/2},\zeta_{1/2}+\bar{\zeta}_{1/2}).
\]

Repeating the previous procedure for $\bar{S}_{n_{i}^{1}/2}^{1/2}$, we can
define random variable $L_{1/4}:=(\xi_{1/4},\eta_{1/4},\zeta_{1/4})$.
Proceeding in this way, one can obtain $L_{1/2^{m}}:=(\xi_{1/2^{m}}%
,\eta_{1/2^{m}},\zeta_{1/2^{m}})$ in $(\Omega,\mathcal{H},\mathbb{\hat{E}}%
_{1})$, $m\in \mathbb{N}$, such that for each $L_{1/2^{m}}$ there exists a
convergent sequence $\{ \bar{S}_{n_{i}^{m}/2^{m}}^{1/2^{m}}\}_{i=1}^{\infty}$
converging in law to it. Finally, using the random variables $L_{1/2^{m}}%
(m\in \mathbb{N)}$, we may construct a sublinear expectation $\mathbb{\tilde
{E}}$ on $(\tilde{\Omega},Lip(\tilde{\Omega}))$ such that the canonical
process $(\tilde{L}_{t})_{t\in \lbrack0,1]}$ is a nonlinear L\'{e}vy process
(see Appendix for details), and $(\tilde{\xi}_{t},\tilde{\eta}_{t})\overset
{d}{=}(\sqrt{t}\tilde{\xi}_{1},t\tilde{\eta}_{1})$ is a generalized
$G$-Brownian motion (see Peng \cite[Chapter 3]{P2010}).

In the following Theorem \ref{unique viscosity theorem}, we will prove that
the distribution of $\tilde{L}_{1}$ is uniquely determined by $u(1,0,0,0)$,
where $u$ is the unique viscosity solution of the fully nonlinear PIDE
(\ref{1.0}). Thus we can infer that for any $\phi \in C_{b,Lip}(\mathbb{R}%
^{3d})$, $\mathbb{\hat{E}}[\phi(\bar{S}_{n})]\rightarrow u(1,0,0,0)$, as
$n\rightarrow \infty$. The proof is completed.
\end{proof}

\subsection{Estimate for $\alpha$-stable L\'{e}vy measure}

\label{Section_Estimates for stable Levy measure}

Recall that $F_{\mu}$ is the $\alpha$-stable L\'{e}vy measure given in
(\ref{F_mu}), $\mathcal{L}_{0}$ is the set of $\alpha$-stable L\'{e}vy measure
on $\mathbb{R}^{d}$ satisfying (\ref{L_0}), and $\mathcal{L\subset L}_{0}$ is
a nonempty compact convex set. It can be verified by the Sato-type result (see
\cite[Remark 14.4]{Sato1999}) that
\begin{equation}
\mathcal{K}:=\sup \limits_{F_{\mu}\in \mathcal{L}}\int_{\mathbb{R}^{d}}%
|z|\wedge|z|^{2}F_{\mu}(dz)<\infty, \label{F_mu condition 1}%
\end{equation}
and
\begin{equation}
\lim_{\varepsilon \rightarrow0}\mathcal{K}_{\varepsilon}=0\text{ \ for
\ }\mathcal{K}_{\varepsilon}:=\sup \limits_{F_{\mu}\in \mathcal{L}}\int
_{|z|\leq \varepsilon}|z|^{2}F_{\mu}(dz). \label{F_mu condition 2}%
\end{equation}
%(\ref{F_mu condition 1}) means the jumps considered are integrable, and
%(\ref{F_mu condition 2}) guarantees the uniqueness of the viscosity
%solution to the fully nonlinear PIDE (\ref{PIDE}).

\begin{lemma}
\label{Lipschitz}For each $\varphi \in C_{b}^{3}(\mathbb{R}^{d})$, we have for
$z,z^{\prime}\in \mathbb{R}^{d}$,
\[
\sup \limits_{F_{\mu}\in \mathcal{L}}\int_{\mathbb{R}^{d}}\left \vert
\delta_{\lambda}\varphi(z^{\prime})-\delta_{\lambda}\varphi \left(  z\right)
\right \vert F_{\mu}(d\lambda)\leq C|z^{\prime}-z|,
\]
where $C\ $is a constant depending on the bounds of $D^{2}\varphi$,
$D^{3}\varphi$, and $\mathcal{K}$.
\end{lemma}

\begin{proof}
Note that for $z\in \mathbb{R}^{d}$,
\[
\delta_{\lambda}\varphi \left(  z\right)  =\int_{0}^{1}\langle D\varphi
(z+\theta \lambda)-D\varphi(z),\lambda \rangle d\theta=\int_{0}^{1}\int_{0}%
^{1}\langle D^{2}\varphi(z+\tau \theta \lambda)\lambda,\lambda \rangle \theta
d\tau d\theta.
\]
Then, it follows that
\begin{align*}
&  \int_{|\lambda|\leq1}\left \vert \delta_{\lambda}\varphi(z^{\prime}%
)-\delta_{\lambda}\varphi \left(  z\right)  \right \vert F_{\mu}(d\lambda)\\
&  =\int_{|\lambda|\leq1}\left \vert \int_{0}^{1}\int_{0}^{1}\langle
(D^{2}\varphi(z^{\prime}+\tau \theta \lambda)-D^{2}\varphi(z+\tau \theta
\lambda))\lambda,\lambda \rangle \theta d\tau d\theta \right \vert F_{\mu
}(d\lambda)\\
&  \leq|D^{3}\varphi|_{0}\int_{|\lambda|\leq1}|\lambda|^{2}F_{\mu}%
(d\lambda)|z^{\prime}-z|,
\end{align*}
and
\begin{align*}
&  \int_{|\lambda|>1}\left \vert \delta_{\lambda}\varphi(z^{\prime}%
)-\delta_{\lambda}\varphi \left(  z\right)  \right \vert F_{\mu}(d\lambda)\\
&  =\int_{|\lambda|>1}\left \vert \int_{0}^{1}\langle D\varphi(z^{\prime
}+\theta \lambda)-D\varphi(z+\theta \lambda),\lambda \rangle d\theta-\int_{0}%
^{1}\langle D\varphi(z^{\prime})-D\varphi(z),\lambda \rangle d\theta \right \vert
F_{\mu}(d\lambda)\\
&  \leq2\left \vert D^{2}\varphi \right \vert _{0}\int_{|\lambda|>1}%
|\lambda|F_{\mu}(d\lambda)|z^{\prime}-z|,
\end{align*}
for $z,z^{\prime}\in \mathbb{R}^{d}$, where $|D^{2}\varphi|_{0}:=\sup
_{z\in \mathbb{R}^{d}}|D^{2}\varphi(z)|$ and $|D^{3}\varphi|_{0}:=\sup
_{z\in \mathbb{R}^{d}}|D^{3}\varphi(z)|$. Consequently,
\[
\sup \limits_{F_{\mu}\in \mathcal{L}}\int_{\mathbb{R}^{d}}\left \vert
\delta_{\lambda}\varphi(z^{\prime})-\delta_{\lambda}\varphi \left(  z\right)
\right \vert F_{\mu}(d\lambda)\leq C_{\varphi,\mathcal{K}}|z^{\prime}-z|,
\]
with $C_{\varphi,\mathcal{K}}=(|D^{3}\varphi|_{0}+2\left \vert D^{2}%
\varphi \right \vert _{0})\mathcal{K}$. The proof is completed.
\end{proof}

The following estimate is crucial to our main result.

\begin{theorem}
\label{recursive}Assume that (A2)-(A3) hold. Then, for $\varphi \in C_{b}%
^{3}(\mathbb{R}^{d})$ and $s\in \lbrack0,1]$,
\begin{equation}
\lim_{n\rightarrow \infty}\bigg \vert \mathbb{\hat{E}}\bigg[\varphi \left(
z+(s/n)^{\frac{1}{\alpha}}S_{n}^{3}\right)  -\varphi(z)\bigg]-s\sup
\limits_{F_{\mu}\in \mathcal{L}}\int_{\mathbb{R}^{d}}\delta_{\lambda}
\varphi(z)F_{\mu}(d\lambda)\bigg \vert=o(s),\nonumber
\end{equation}
uniformly on $z\in \mathbb{R}^{d}$, where $o(s)/s\rightarrow0$ as
$s\rightarrow0$.
\end{theorem}

\begin{proof}
Because $Z_{n}$ is independent from $Z_{1},\dots,Z_{n-1}$, we have
\begin{align*}
&  \mathbb{\hat{E}}\bigg[\varphi \left(  z+(s/n)^{\frac{1}{\alpha}}S_{n}%
^{3}\right)  \bigg]-\varphi(z)-s\epsilon(z)\\
&  =\mathbb{\hat{E}}\bigg[\left.  \mathbb{\hat{E}}\bigg[\varphi \left(
z+(s/n)^{\frac{1}{\alpha}}(\omega_{n-1}+Z_{n})\right)  \bigg]\right \vert
_{\substack{z_{1}=Z_{1}\\ \cdots \\z_{n-1}=Z_{n-1}}}\bigg]-s\epsilon
(z)-\varphi(z),
\end{align*}
where
\[
\omega_{n}:=\sum \limits_{k=1}^{n}z_{k}\  \  \  \text{and}\  \  \  \epsilon
(z):=\sup \limits_{F_{\mu}\in \mathcal{L}}\int_{\mathbb{R}^{d}}\delta_{\lambda
}\varphi(z)F_{\mu}(d\lambda).
\]
Thanks to the assumptions (A2)-(A3) and Lemma \ref{Lipschitz}, we deduce that
\begin{align*}
&  \mathbb{\hat{E}}\bigg[\varphi \left(  z+(s/n)^{\frac{1}{\alpha}}%
(\omega_{n-1}+Z_{n})\right)  \bigg]\\
&  =\mathbb{\hat{E}}\bigg[\varphi \left(  z+(s/n)^{\frac{1}{\alpha}}%
(\omega_{n-1}+Z_{n})\right)  -\varphi \left(  z+(s/n)^{\frac{1}{\alpha}}%
\omega_{n-1}\right) \\
&  \text{ \  \ }-\frac{s}{n}\sup \limits_{F_{\mu}\in \mathcal{L}}\int
_{\mathbb{R}^{d}}\delta_{\lambda}\varphi \left(  z+(s/n)^{\frac{1}{\alpha}%
}\omega_{n-1}\right)  F_{\mu}(d\lambda)\bigg]\\
&  \text{ \  \ }+\frac{s}{n}\sup \limits_{F_{\mu}\in \mathcal{L}}\int
_{\mathbb{R}^{d}}\delta_{\lambda}\varphi \left(  z+(s/n)^{\frac{1}{\alpha}%
}\omega_{n-1}\right)  F_{\mu}(d\lambda)\\
&  \text{ \  \ }-\frac{s}{n}\sup \limits_{F_{\mu}\in \mathcal{L}}\int
_{\mathbb{R}^{d}}\delta_{\lambda}\varphi \left(  z\right)  F_{\mu}%
(d\lambda)+\varphi \left(  z+(s/n)^{\frac{1}{\alpha}}\omega_{n-1}\right)
+\frac{s}{n}\epsilon(z)\\
&  \leq \varphi \left(  z+(s/n)^{\frac{1}{\alpha}}\omega_{n-1}\right)  +\frac
{s}{n}\epsilon(z)+\frac{s}{n}l\left(  \frac{s}{n}\right)  +C\left(  \frac
{s}{n}\right)  ^{1+\frac{1}{\alpha}}\left \vert \omega_{n-1}\right \vert ,
\end{align*}
which implies the following one-step estimate
\begin{align*}
&  \mathbb{\hat{E}}\left[  \varphi \left(  z+(s/n)^{\frac{1}{\alpha}}S_{n}%
^{3}\right)  \right] \\
&  \leq \mathbb{\hat{E}}\left[  \varphi \left(  z+(s/n)^{\frac{1}{\alpha}%
}S_{n-1}^{3}\right)  \right]  +\frac{s}{n}\epsilon(z)+\frac{s}{n}l\left(
\frac{s}{n}\right)  +CM_{z}s^{1+\frac{1}{\alpha}}\frac{1}{n}\left(  \frac
{n-1}{n}\right)  ^{\frac{1}{\alpha}}.
\end{align*}
Repeating the above process recursively, we obtain that
\[
\mathbb{\hat{E}}\left[  \varphi \left(  z+(s/n)^{\frac{1}{\alpha}}S_{n}%
^{3}\right)  \right]  \leq \varphi \left(  z\right)  +s\epsilon(z)+sl\left(
\frac{s}{n}\right)  +CM_{z}s^{1+\frac{1}{\alpha}}\frac{1}{n}\sum
\limits_{k=1}^{n-1}\left(  \frac{k}{n}\right)  ^{\frac{1}{\alpha}}.
\]
Analogously, we have
\[
\mathbb{\hat{E}}\left[  \varphi \left(  z+(s/n)^{\frac{1}{\alpha}}S_{n}%
^{3}\right)  \right]  \geq \varphi \left(  z\right)  +s\epsilon(z)-sl\left(
\frac{s}{n}\right)  -CM_{z}s^{1+\frac{1}{\alpha}}\frac{1}{n}\sum
\limits_{k=1}^{n-1}\left(  \frac{k}{n}\right)  ^{\frac{1}{\alpha}}.
\]
Thus,
\[
\lim_{n\rightarrow \infty}\left \vert \mathbb{\hat{E}}\left[  \varphi \left(
z+(s/n)^{\frac{1}{\alpha}}S_{n}^{3}\right)  \right]  -\varphi(z)-s\epsilon
(z)\right \vert \leq CM_{z}s^{1+\frac{1}{\alpha}}\frac{\alpha}{1+\alpha},
\]
where we have used the fact that
\[
\lim_{n\rightarrow \infty}\frac{1}{n}\sum \limits_{k=1}^{n-1}\left(  \frac{k}%
{n}\right)  ^{\frac{1}{\alpha}}=\frac{\alpha}{1+\alpha}.
\]
This implies the desired result.
\end{proof}

\subsection{L\'evy-Khintchine representation of nonlinear L\'evy process}

\label{Section_Representation of Levy process}

In this section, we shall present the characterization of $\lim \limits_{\delta
\rightarrow0}\mathbb{\tilde{E}}[\varphi(\tilde{\zeta} _{\delta})+\langle
p,\tilde{\eta}_{\delta}\rangle+\frac{1}{2}\langle A\tilde{ \xi}_{\delta
},\tilde{\xi}_{\delta}\rangle]\delta^{-1}$ for $\varphi \in C_{b}%
^{3}(\mathbb{R}^{d})$ with $\varphi(0)=0$ and $(p,A)\in \mathbb{R}^{d}%
\times \mathbb{S}(d)$, which can be regarded as a new type of L\'evy-Khintchine
representation for the nonlinear L\'evy process $(\tilde{L}_{t})_{t\in[0,1]}$.
It will play an important role in establishing the related PIDE in Section
\ref{Section_connection PIDE} (see (\ref{1.5})).

For each $N>0$ and $s\in \lbrack0,1]$, under the assumptions (A1)-(A2), Theorem
\ref{The construction of Levy process} shows that there exists a sequence
$\{s_{k}\}_{k=1}^{\infty}\subset \mathcal{D}_{\infty}[0,1]$ satisfying
$s_{k}\downarrow s$ as $k\rightarrow \infty$ and a convergent sequence
$\big \{(s_{k}/n_{i}^{\ast})^{\frac{1}{\alpha}}S_{n_{i}^{\ast}}^{3}%
\big \}_{i=1}^{\infty}$ for each $s_{k}$, such that
\[
\mathbb{\tilde{E}}[|\tilde{\zeta}_{s}|\wedge N]=\lim_{k\rightarrow \infty}%
\lim_{i\rightarrow \infty}\mathbb{\hat{E}}\big[|(s_{k}/n_{i}^{\ast})^{\frac
{1}{\alpha}}S_{n_{i}^{\ast}}^{3}|\wedge N\big]\leq s^{\frac{1}{\alpha}}%
\sup \limits_{n}\mathbb{\hat{E}}\big[|n^{-\frac{1}{\alpha}}S_{n}^{3}|\big].
\]
Define
\begin{equation}
\mathbb{\tilde{E}}[|\tilde{\zeta}_{s}|]:=\lim_{N\rightarrow \infty
}\mathbb{\  \tilde{E}}[|\tilde{\zeta}_{s}|\wedge N]\leq s^{\frac{1}{\alpha}%
}M_{z}. \label{Z moment}%
\end{equation}
Also, for any $\varphi \in C_{b}^{3}(\mathbb{R}^{d})$ and $(s,z)\in
\lbrack0,1]\times \mathbb{R}^{d}$,\ we have
\[
\mathbb{\tilde{E}}[\varphi(z+\tilde{\zeta}_{s})]=\lim_{k\rightarrow \infty}%
\lim_{i\rightarrow \infty}\mathbb{\hat{E}}\big[\varphi(z+(s_{k}/n_{i}^{\ast
})^{\frac{1}{\alpha}}S_{n_{i}^{\ast}}^{3})\big].
\]
Furthermore, under the assumption (A3), it follows from Theorem
\ref{recursive} that
\begin{align}
&  \bigg \vert \mathbb{\tilde{E}}[\varphi(z+\tilde{\zeta}_{s})]-\varphi
(z)-s\sup \limits_{F_{\mu}\in \mathcal{L}}\int_{\mathbb{R}^{d}}\delta_{\lambda
}\varphi(z)F_{\mu}(d\lambda)\bigg \vert \label{Z condition}\\
&  \leq \lim_{k\rightarrow \infty}\bigg \vert \lim_{i\rightarrow \infty
}\mathbb{\hat{E}}\big[\varphi(z+(s_{k}/n_{i}^{\ast})^{\frac{1}{\alpha}%
}S_{n_{i}^{\ast}}^{3})\big]-\varphi(z)-s_{k}\sup \limits_{F_{\mu}\in
\mathcal{L}}\int_{\mathbb{R}^{d}}\delta_{\lambda}\varphi(z)F_{\mu}%
(d\lambda)\bigg \vert=o(s),\nonumber
\end{align}
uniformly on $z\in \mathbb{R}^{d}$.

Consider
\[
\mathfrak{F}_{0}=\left \{  \varphi \in C_{b}^{3}(\mathbb{R}^{d}):\varphi
(0)=0\right \}
\]
and
\[
\mathfrak{F}=\left \{  (\varphi,p,A):\varphi \in \mathfrak{F}_{0}\text{ and }
(p,A)\in \mathbb{R}^{d}\times \mathbb{S}(d)\right \}  .
\]
Obviously, $\mathfrak{F}$ and $\mathfrak{F}_{0}$ are both linear spaces.

\begin{lemma}
\label{exist theorem}Assume that (A1)-(A3) hold. Then, for each $(\varphi
,p,A)\in \mathfrak{F}$, $\lim \limits_{\delta \rightarrow0}\mathbb{\tilde{E}}
\big[\varphi(\tilde{\zeta}_{\delta})+\langle p,\tilde{\eta} _{\delta}\rangle+
\frac{1}{2}\langle A\tilde{\xi}_{\delta},\tilde{\xi}_{\delta}\rangle
\big]\delta^{-1}$ exists.
\end{lemma}

\begin{proof}
For given $(\varphi,p,A)\in \mathfrak{F}$, define
\[
f(s):=\mathbb{\tilde{E}}\big[\varphi(\tilde{\zeta}_{s})+\langle p,\tilde{\eta}
_{s}\rangle+\frac{1}{2}\langle A\tilde{\xi}_{s},\tilde{\xi}_{s}\rangle \big],
\text{ \ }s\in[0,1].
\]
Clearly, $f(0)=0$. We first claim that $f$ is a Lipschitz function. In fact,
for each $s,\delta \in \lbrack0,1]$, it follows from Proposition
\ref{Prop^E Plimi} that
\[
|f(s+\delta)-f(s)|\leq \mathbb{\tilde{E}}[R]\vee \mathbb{\tilde{E}}[-R],
\]
where%
\begin{align*}
R  &  :=\varphi(\tilde{\zeta}_{s}+\tilde{\zeta}_{s+\delta}-\tilde{\zeta}
_{s})-\varphi(\tilde{\zeta}_{s})+\frac{1}{2}\langle A(\tilde{\xi}_{s+\delta
}-\tilde{\xi}_{s}),\tilde{\xi}_{s+\delta}-\tilde{\xi}_{s}\rangle \\
&  \text{ \  \ }+\langle A\tilde{\xi}_{s},\tilde{\xi}_{s+\delta}-\tilde{\xi}
_{s}\rangle+\langle p,\tilde{\eta}_{s+\delta}-\tilde{\eta}_{s}\rangle.
\end{align*}
By using the independent stationary increments property of $(\tilde{L}
_{t})_{t\in[0,1]}$ and $\mathbb{\tilde{E}}[\tilde{\xi}_{t}]=\mathbb{\tilde{E}%
}[-\tilde{\xi}_{t}]=0$, we obtain
\[
\mathbb{\tilde{E}}[R]=\mathbb{\tilde{E}}\big[\mathbb{\tilde{E}}[\varphi
(z+\tilde{\zeta} _{\delta})-\varphi(z)+\langle p,\tilde{\eta}_{\delta}%
\rangle+\frac{1}{2} \langle A\tilde{\xi}_{\delta},\tilde{\xi}_{\delta}%
\rangle]|_{z=\tilde{\zeta}_{s}}\big].
\]
In view of the estimate (\ref{Z condition}) and Lemma \ref{Lipschitz}, we
derive
\[
\mathbb{\tilde{E}}[R]\leq \mathbb{\tilde{E}}\big[|\langle p,\tilde{\eta
}_{\delta}\rangle+\frac{1}{2}\langle A\tilde{\xi}_{\delta},\tilde{\xi}
_{\delta}\rangle|\big]+\delta \mathbb{\tilde{E}}\Big[\sup \limits_{F_{\mu}%
\in \mathcal{L}}\int_{\mathbb{R}^{d}}|\delta_{\lambda}\varphi(\tilde{\zeta}
_{s})|F_{\mu}(d\lambda)\Big]+o(\delta)\leq C\delta,
\]
where $C>0$ is a constant. Similarly, we have $\mathbb{\tilde{E}}[-R]\leq
C\delta$. Hence, $f(\cdot)$ is differentiable almost everywhere on $[0,1]$. We
assume that for each fixed $t_{0}\in \lbrack0,1]$, $f^{\prime}(t_{0})$ exists.
Using the independent stationary increments property again, by Proposition
\ref{Y independent X}, we derive
\[
\frac{f(\delta)}{\delta}=\frac{f(t_{0}+\delta)-f(t_{0})}{\delta}
-\Lambda_{\delta},
\]
where
\begin{align*}
\Lambda_{\delta}  &  =\delta^{-1}\Big(f(t_{0}+\delta)-\mathbb{\tilde{E}}
\big[ \varphi(\tilde{\zeta}_{t_{0}+\delta}-\tilde{\zeta}_{t_{0}}%
)+\varphi(\tilde{ \zeta}_{t_{0}})+\langle p,\tilde{\eta}_{t_{0}+\delta}%
\rangle \\
&  \text{ \  \ }+\frac{1}{2}\langle A(\tilde{\xi}_{t_{0}+\delta}-\tilde{\xi}
_{t_{0}}),\tilde{\xi}_{t_{0}+\delta}-\tilde{\xi}_{t_{0}}\rangle+\frac{1}{2}
\langle A\tilde{\xi}_{t_{0}},\tilde{\xi}_{t_{0}}\rangle \big]\Big).
\end{align*}
Note that
\[
\frac{1}{2}\langle A\tilde{\xi}_{t_{0}+\delta},\tilde{\xi}_{t_{0}+\delta
}\rangle=\frac{1}{2}\langle A(\tilde{\xi}_{t_{0}+\delta}-\tilde{\xi} _{t_{0}%
}),\tilde{\xi}_{t_{0}+\delta}-\tilde{\xi}_{t_{0}}\rangle+\frac{1}{2} \langle
A\tilde{\xi}_{t_{0}},\tilde{\xi}_{t_{0}}\rangle+\langle A\tilde{\xi}_{t_{0}},
\tilde{\xi}_{t_{0}+\delta}-\tilde{\xi}_{t_{0}}\rangle.
\]
Similar to the above procedure, we deduce that
\[
|\Lambda_{\delta}|\leq(\mathbb{\tilde{E}}[U]\vee \mathbb{\tilde{E}}
[-U])\delta^{-1},
\]
where $U:=\varphi(\tilde{\zeta}_{t_{0}}+\tilde{\zeta}_{t_{0}+\delta}-\tilde{
\zeta}_{t_{0}})-\varphi(\tilde{\zeta}_{t_{0}+\delta}-\tilde{\zeta} _{t_{0}%
})-\varphi(\tilde{\zeta}_{t_{0}})$. For each fixed $z_{0}\in \mathbb{R}^{d}$,
denote
\[
\tilde{\varphi}(z;z_{0}):=\varphi(z+z_{0})-\varphi(z)-\varphi(z_{0}),\text{
for }z\in \mathbb{R}^{d}.
\]
It is easy to check that $\tilde{\varphi}(0;z_{0})=0$, $\delta_{\lambda}
\tilde{\varphi}(z;z_{0})=\delta_{\lambda}\varphi(z+z_{0})-\delta_{\lambda
}\varphi(z)$. Then
\[
\mathbb{\tilde{E}}[U]=\mathbb{\tilde{E}}\Big[\mathbb{\tilde{E}} \big[\tilde
{\varphi}( \tilde{\zeta}_{t_{0}+\delta}-\tilde{\zeta} _{t_{0}};z_{0}%
)\big]\big|_{z_{0}=\tilde{\zeta}_{t_{0}}}\Big] =\mathbb{\tilde{E}%
}\Big[\mathbb{\tilde{E}} \big[\tilde{\varphi}( \tilde{\zeta}_{\delta}%
;z_{0})\big]\big|_{z_{0}=\tilde{\zeta}_{t_{0}}}\Big]
\]
and
\[
\mathbb{\tilde{E}}[U]\leq \delta \mathbb{\tilde{E}}\Big[\sup \limits_{F_{\mu}%
\in \mathcal{L}}\int_{\mathbb{R}^{d}}|\delta_{\lambda}\tilde{\varphi}(0;\tilde{
\zeta}_{t_{0}})|F_{\mu}(d\lambda)\Big]+o(\delta)\leq C\delta \mathbb{\tilde{E}}
[|\tilde{\zeta}_{t_{0}}|]+o(\delta)\leq C\delta \sqrt[\alpha]{t_{0}}
+o(\delta),
\]
where we have used the estimates (\ref{Z moment})-(\ref{Z condition}) and
Lemma \ref{Lipschitz}. Similarly, $\mathbb{\tilde{E}}[-U]\leq C\delta
\sqrt[\alpha]{t_{0}}+o(\delta)$. In turn, $\lim_{\delta \rightarrow0}%
|\Lambda_{\delta}|\leq C\sqrt[\alpha]{t_{0}}$, where $C>0$ is a constant
independent of $t_{0}$. Therefore,
\[
\underset{\delta \downarrow0}{|\lim \sup}f(\delta)\delta^{-1}-\underset{
\delta \downarrow0}{\lim \inf}f(\delta)\delta^{-1}|\leq2C\sqrt[\alpha]{t_{0}} .
\]
By letting $t_{0}\rightarrow0$, the desired result follows.
\end{proof}

Thanks to Lemma \ref{exist theorem}, we can define a functional $\mathbb{F}:$
$\mathfrak{F}\mathbb{\rightarrow R}$ by
\[
\mathbb{F}[(\varphi,p,A)]:=\lim_{\delta \rightarrow0}\mathbb{\tilde{E}}
\big[\varphi(\tilde{\zeta}_{\delta})+\langle p,\tilde{\eta}_{\delta}\rangle+
\frac{1}{2}\langle A\tilde{\xi}_{\delta},\tilde{\xi}_{\delta}\rangle
\big]\delta^{-1}.
\]
It is easy to verify that $\mathbb{F}$ is a sublinear functional, monotone in
$(\varphi,A)\in \mathfrak{F}_{0}\times \mathbb{S}(d)$ in the following sense:
for $\varphi,\varphi^{\prime}\in \mathfrak{F}_{0}$, $p,p^{\prime}\in
\mathbb{R}^{d}$, and $A,A^{\prime}\in \mathbb{S}(d)$,
\[
\left \{
\begin{array}
[c]{l}%
\mathbb{F}[(\varphi+\varphi^{\prime},p+p^{\prime},A+A^{\prime})]\leq
\mathbb{F}[(\varphi,p,A)]+\mathbb{F}[(\varphi^{\prime},p^{\prime},A^{\prime
})],\\
\mathbb{F}[\lambda(\varphi,p,A)]=\lambda \mathbb{F}[(\varphi,p,A)],\  \forall
\lambda \geq0,\\
\mathbb{F}[(\varphi,p,A)]\leq \mathbb{F}[(\varphi^{\prime},p,A^{\prime})],
\text{ if }\varphi \leq \varphi^{\prime}\text{\ and }A\leq A^{\prime}.
\end{array}
\right.
\]

The following can be regarded as a L\'evy-Khintchine representation for the
nonlinear L\'evy process $(\tilde{L}_{t})_{t\in[0,1]}$.

\begin{theorem}
\label{represent theorem}Assume that (A1)-(A3) hold. Then, for each
$(\varphi,p,A)\in \mathfrak{F}$, there exists an uncertainty set $\Theta
\subset \mathcal{L}\times \mathbb{R}^{d}\times \mathbb{S}_{+}(d)$ satisfying
\[
\sup_{(F_{\mu},q,Q)\in \Theta}\left \{  \int_{\mathbb{R}^{d}}|z|\wedge
|z|^{2}F_{\mu}(dz)+|q|+|Q|\right \}  <\infty
\]
such that
\[
\mathbb{F}[(\varphi,p,A)]=\sup_{(F_{\mu},q,Q)\in \Theta}\left \{  \int_{
\mathbb{R}^{d}}\left(  \varphi(z)-\langle D\varphi(0),z\rangle \right)  F_{\mu
}(dz)+\langle p,q\rangle+\frac{1}{2}tr[AQ]\right \}  .
\]

\end{theorem}

\begin{proof}
From the representation theorem of sublinear functional (see, for instance
\cite[Theorem 1.2.1]{P2010}), there exists a family of linear functionals
$F_{\theta}:\mathfrak{F}\mathbb{\rightarrow R}$ indexed by $\theta \in \Theta$,
such that
\begin{equation}
\mathbb{F}[(\varphi,p,A)]=\sup_{\theta \in \Theta}F_{\theta}[(\varphi,p,A)].
\label{2.6}%
\end{equation}

Since $F_{\theta}$ is a linear functional, then $F_{\theta}[(\varphi
,p,A)]=F_{\theta}[(\varphi,0,0)]+F_{\theta}[(0,p,A)]$. It is easily seen that
for any $(p,A)\in \mathbb{R}^{d}\times \mathbb{S}(d)$, there exists a $(q,Q)$
belonging to a bounded and closed subset $\Gamma \subset \mathbb{R}^{d}%
\times \mathbb{S}_{+}(d)$ such that
\begin{equation}
F_{\theta}[(0,p,A)]=\langle p,q\rangle+\frac{1}{2}tr[AQ]. \label{2.7}%
\end{equation}
See, for example, Remark \ref{remark_G}. On the other hand, define
\[
\tilde{F}_{\theta}[\varphi]:=F_{\theta}[(\varphi,0,0)],\text{ \ for any }
\varphi \in \mathfrak{F}_{0}.
\]
From the monotone property of $\mathbb{F}$, one can check that $\tilde
{F}_{\theta}$ is a positive linear functional. Moreover, it follows from
(\ref{Z condition}) that
\begin{align*}
\mathbb{F}[(\varphi,0,0)]  &  =\lim_{\delta \rightarrow0}\bigg(\mathbb{\tilde
{E}} [\varphi(\tilde{\zeta}_{\delta})]-\delta \sup \limits_{F_{\mu}%
\in \mathcal{L}}\int_{\mathbb{R}^{d}}\delta_{\lambda}\varphi(0)F_{\mu}%
(d\lambda) \bigg)\delta^{-1}+\sup \limits_{F_{\mu}\in \mathcal{L}}%
\int_{\mathbb{R} ^{d}}\delta_{\lambda}\varphi(0)F_{\mu}(d\lambda)\\
&  =\sup \limits_{F_{\mu}\in \mathcal{L}}\int_{\mathbb{R}^{d}}\delta_{\lambda
}\varphi(0)F_{\mu}(d\lambda),
\end{align*}
which yields that
\begin{equation}
\tilde{F}_{\theta}[\varphi]\leq \mathbb{F}[(\varphi,0,0)]= \sup \limits_{F_{\mu
}\in \mathcal{L}}\int_{\mathbb{R}^{d}}\delta_{\lambda} \varphi(0)F_{\mu
}(d\lambda),\text{ for }\varphi \in \mathfrak{F}_{0}. \label{2.0}%
\end{equation}
In the following, we shall give the representation of $\tilde{F}_{\theta
}[\varphi]$ for any $\varphi \in \mathfrak{F}_{0}$. The proof is divided into
the following three steps.

Step 1. We introduce a linear space
\[
\mathfrak{R}=\left \{  \varphi \in C_{Lip}(\mathbb{R}^{d}):\exists C>0,
|\varphi(z)|\leq C(|z|\wedge|z|^{2})\right \}  .
\]
It is clear that $\mathfrak{R}$ is a vector lattice, that is, if $\varphi
\in \mathfrak{R}$ then $|\varphi|\in \mathfrak{R}$ and $\varphi \wedge
1\in \mathfrak{R}$. Define a sublinear functional $\mathbb{K}[\cdot]$ on
$\mathfrak{R}$ by
\[
\mathbb{K}[\varphi]=\sup \limits_{F_{\mu}\in \mathcal{L}} \int_{\mathbb{R}^{d}%
}\varphi(z)F_{\mu}(dz),\text{ for }\varphi \in \mathfrak{R}.
\]

We claim that the functional $\mathbb{K}[\cdot]$ is regular, that is, if for
each $\{ \varphi_{n}\}_{n=1}^{\infty}$ in $\mathfrak{R}$ such that
$\varphi_{n}\downarrow0$ as $n\rightarrow \infty$, then $\mathbb{K}[\varphi
_{n}]\downarrow0$ as $n\rightarrow \infty$. Indeed, for each fixed
$0<\gamma_{1}<1<\gamma_{2}<\infty$, we have for\ $z\in \mathbb{R}^{d}$,
\[
\varphi_{n}(z)\leq \varphi_{1}(z)\mathbbm{1}_{\{|z|\leq \gamma_{1}\}}%
+\varphi_{n}(z)\mathbbm{1}_{\{ \gamma_{1}\leq|z|\leq \gamma_{2}\}}+\varphi
_{1}(z)\mathbbm{1}_{\{|z|\geq \gamma_{2}\}},
\]
then
\[
\mathbb{K}[\varphi_{n}]\leq \ C\sup \limits_{F_{\mu}\in \mathcal{L}}%
\int_{\{|z|\leq \gamma_{1}\}}|z|^{2}F_{\mu}(dz)+\sup \limits_{F_{\mu}%
\in \mathcal{L}}\int_{\{ \gamma_{1}\leq|z|\leq \gamma_{2}\}}\varphi_{n}%
(z)F_{\mu}(dz)+C\frac{\overline{\Lambda}}{\alpha-1}\gamma_{2}^{1-\alpha},
\]
where we have used the fact that (cf. \cite[Remark 14.4]{Sato1999})%
\[
\sup \limits_{F_{\mu}\in \mathcal{L}}\int_{\{|z|\geq \gamma_{2}\}}|z|F_{\mu
}(dz)=\sup \limits_{F_{\mu}\in \mathcal{L}}\int_{S}\mu(d\beta)\int_{0}^{\infty
}\mathbbm{1}_{\{|r\beta|\geq \gamma_{2}\}}|r\beta|\frac{dr}{r^{\alpha+1}}%
\leq \frac{\overline{\Lambda}}{\alpha-1}\gamma_{2}^{1-\alpha}.
\]
Since $\varphi_{n}\downarrow0$ as $n\rightarrow \infty$, from Dini's theorem we
know $\sup \limits_{\gamma_{1}\leq|x|\leq \gamma_{2}}\varphi_{n}(x)\downarrow0$,
as $n\rightarrow \infty$, which implies that
\begin{align*}
\sup \limits_{F_{\mu}\in \mathcal{L}}\int_{\{ \gamma_{1}\leq|z|\leq \gamma_{2}%
\}}\varphi_{n}(z)F_{\mu}(dz)  &  \leq \sup_{\gamma_{1}\leq|x|\leq \gamma_{2}%
}\varphi_{n}(x)\sup \limits_{F_{\mu}\in \mathcal{L}}\int_{\mathbb{R}^{d}%
}\mathbbm{1}_{\{ \gamma_{1}\leq|z|\leq \gamma_{2}\}}F_{\mu}(dz)\\
&  \leq \frac{\overline{\Lambda}}{\alpha}(\gamma_{1}^{-\alpha}-\gamma
_{2}^{-\alpha})\sup_{\gamma_{1}\leq|x|\leq \gamma_{2}}\varphi_{n}%
(x)\rightarrow0,\text{ \ as }n\rightarrow \infty.
\end{align*}
Therefore, we have
\[
\lim_{n\rightarrow \infty}\mathbb{K}[\varphi_{n}]\leq \ C\sup \limits_{F_{\mu}%
\in \mathcal{L}}\int_{\{|z|\leq \gamma_{1}\}}|z|^{2}F_{\mu}(dz)+C\frac
{\overline{\Lambda}}{\alpha-1}\gamma_{2}^{1-\alpha}.
\]
In view of (\ref{F_mu condition 2}), the claim follows by letting $\gamma
_{1}\rightarrow0$ and $\gamma_{2}\rightarrow \infty$.

Step 2. Set
\[
\mathfrak{F}_{1}\mathfrak{=}\left \{  \varphi \in \mathfrak{F}_{0}:D\varphi
(0)=0\right \}  .
\]
Note that, $\mathfrak{F}_{1}\subset \mathfrak{F}_{0}$ and $\mathfrak{F}
_{1}\subset \mathfrak{R}$. Figure 1 illustrates the three sets $\mathfrak{R}%
,\mathfrak{F}_{0}$ and $\mathfrak{F}_{1}$.

\begin{figure}[h]
\caption{The three sets $\mathfrak{R},\mathfrak{F}_{0}$ and $\mathfrak{F}_{1}%
$.}%
\centering
%Requires \usepackage{graphicx}
\includegraphics[width=.4\textwidth]{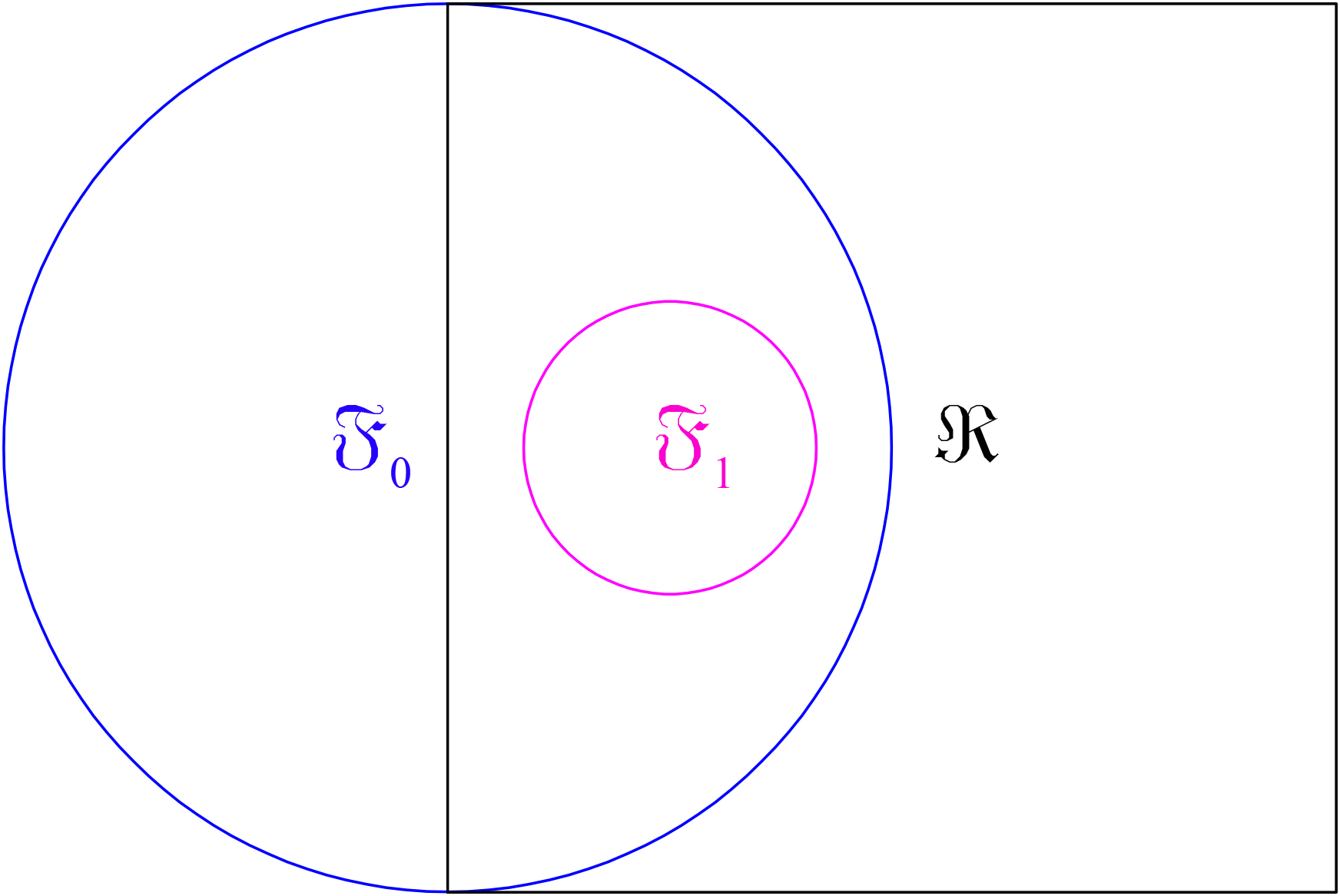}\end{figure}

For any $\varphi \in \mathfrak{F}_{1}$, using (\ref{2.0}) and $\delta_{\lambda
}\varphi(0)=\varphi(\lambda)$, we have
\[
\tilde{F}_{\theta}[\varphi]\leq \mathbb{K}[\varphi],\text{ for }\varphi
\in \mathfrak{F}_{1}.
\]
By Hahn-Banach theorem, we extend the linear functional $\tilde{F}_{\theta
}:\mathfrak{F}_{1}\rightarrow \mathbb{R}$ to $\tilde{F}_{\theta}:
\mathfrak{R}\rightarrow \mathbb{R}$ such that $\tilde{F}_{\theta}[\varphi]
\leq \mathbb{K}[\varphi]$ for $\varphi \in \mathfrak{R}$. Here we still use
$\tilde{F}_{\theta}$ for notation simplicity. Since $\mathbb{K}[\cdot]$ is
regular, it follows that
%$\tilde{F}_{\theta}$ is a Daniell integral on the vector lattice $\mathfrak{R}$, that is,
$\tilde{F}_{\theta}$ is a positive linear functional satisfying $\tilde
{F}_{\theta}[\varphi_{n}]\downarrow0$ for each $\varphi_{n}\in \mathfrak{R}$
such that $\varphi_{n}\downarrow0$. In turn, Daniell-Stone theorem implies
that there exists a unique measure $v$ on $\big(\mathbb{R}^{d}\backslash
\{0\},\mathcal{B}(\mathbb{R}^{d}\backslash \{0\})\big)$ such that
\[
\tilde{F}_{\theta}[\varphi]=\int_{\mathbb{R}^{d}\backslash \{0\}}%
\varphi(z)v(dz)\text{, for }\varphi \in \mathfrak{R}.
\]
%In fact, it is easy to check that $\sigma(\varphi|\varphi
%\in \mathfrak{R}|_{\mathbb{R}^{d}\backslash \{0\}})=B(\mathbb{R}^{d}%
%\backslash \{0\})$, where $\sigma(\varphi|\varphi \in \mathfrak{R})$ is the smallest $\sigma$-field
%with respect to which all the functional $\varphi \in \mathfrak{R}$ is
%measurable.
We claim that there exists some $F_{\mu}\in \mathcal{L}$\ such that
\begin{equation}
\tilde{F}_{\theta}[\varphi]=\int_{\mathbb{R}^{d}}\varphi(z)F_{\mu}(dz)\text{ ,
for }\varphi \in \mathfrak{R}. \label{2.3}%
\end{equation}
Suppose not. For any $F_{\mu}\in \mathcal{L}$, there exists some $\varphi
_{0}\in \mathfrak{R}$\ such that
\[
h(\varphi_{0},F_{\mu}):=\tilde{F}_{\theta}[\varphi_{0}]-\int_{\mathbb{R} ^{d}%
}\varphi_{0}(z)F_{\mu}(dz)\neq0.
\]
Without loss of generality we may assume $h(\varphi_{0},F_{\mu})>0$. Since
$k\varphi_{0}\in \mathfrak{R}$ for $k>0$, we show that for any $F_{\mu}%
\in \mathcal{L}$, $h(k\varphi_{0},F_{\mu})\rightarrow \infty$ as $k\rightarrow
\infty$. Note that $\mathcal{L}$ is a compact convex set, it follows
immediately from minimax theorem (cf. \cite{Fan1953,Sion1958}) that
\[
\sup_{\varphi \in \mathfrak{R}}\inf_{F_{\mu}\in \mathcal{L}}h(\varphi,F_{\mu
})=\inf_{F_{\mu}\in \mathcal{L}}\sup_{\varphi \in \mathfrak{R}}h(\varphi,F_{\mu
})=\infty.
\]
However, seeing that, $\tilde{F}_{\theta}[\varphi]-\mathbb{K}[\varphi]\leq0$
for any $\varphi \in \mathfrak{R}$, then we have
\[
\sup_{\varphi \in \mathfrak{R}}\inf_{F_{\mu}\in \mathcal{L}}h(\varphi,F_{\mu
})\leq0,
\]
which induces a contradiction.

Step 3. For each $k>1$ and $i=1,\ldots,d$, We define $f_{k}^{i}(z)=(z_{i}%
\wedge k)\vee(-k)$ for $z=(z_{1},\ldots,z_{d})\in \mathbb{R}^{d}$. Let
$f_{k}^{i,\varepsilon}(z):=\rho_{\varepsilon}\ast f_{k}^{i}(z)$,
$i=1,\ldots,d$, be the smooth function with the mollifier $\rho_{\varepsilon}$
(see Appendix C.5 in \cite{E2010}) and $f_{k}^{\varepsilon}:=(f_{k}%
^{1,\varepsilon},\ldots,f_{k}^{d,\varepsilon})$. Clearly, $f_{k}%
^{i,\varepsilon}\in \mathfrak{F}_{0}$, $Df_{k}^{i,\varepsilon}(0)=e_{i}$, where
$e_{i}$ is the unit vector with the $i$th component 1, and
\begin{equation}
|f_{k}^{\varepsilon}(z)-z|\leq C|z|\mathbbm{1}_{\{|z|\geq k\}}\text{ \ with
some }C>0. \label{2.4}%
\end{equation}
Note that, for each $\varphi \in \mathfrak{F}_{0}$,
\begin{equation}
\tilde{F}_{\theta}[\varphi]=\tilde{F}_{\theta}[\varphi-\langle D\varphi
(0),f_{k}^{\varepsilon}\rangle]+\tilde{F}_{\theta}[\langle D\varphi
(0),f_{k}^{\varepsilon}\rangle]. \label{2.5}%
\end{equation}
Since $\varphi-\langle D\varphi(0),f_{k}^{\varepsilon}\rangle \in
\mathfrak{F}_{1}$, from (\ref{2.3}), we know that there exists some $F_{\mu
}\in \mathcal{L}$ such that
\[
\tilde{F}_{\theta}[\varphi-\langle D\varphi(0),f_{k}^{\varepsilon}%
\rangle]=\int_{\mathbb{R}^{d}}(\varphi(z)-\langle D\varphi(0),f_{k}%
^{\varepsilon}(z)\rangle)F_{\mu}(dz).
\]
Besides, using (\ref{2.0}) and (\ref{2.4}), we derive that%
\[
\tilde{F}_{\theta}[\langle D\varphi(0),f_{k}^{\varepsilon}\rangle]\leq
\sup \limits_{F_{\mu}\in \mathcal{L}}\int_{\mathbb{R}^{d}}\langle D\varphi
(0),f_{k}^{\varepsilon}(z)-z\rangle F_{\mu}(dz)\leq C\sup \limits_{F_{\mu}%
\in \mathcal{L}}\int_{\{|z|\geq k\}}|z|F_{\mu}(dz)\rightarrow0
\]
as $k\rightarrow \infty$, and similarly, $\tilde{F}_{\theta}[\langle
D\varphi(0),-f_{k}^{\varepsilon}\rangle]\rightarrow0$ as $k\rightarrow \infty$.
This implies that, $|\tilde{F}_{\theta}[\langle D\varphi(0),f_{k}%
^{\varepsilon}\rangle]|\rightarrow0$ as $k\rightarrow \infty$. In addition,
$\int_{\mathbb{R}^{d}}\langle D\varphi(0),z-f_{k}^{\varepsilon}\rangle F_{\mu
}(dz)\rightarrow0$ as $k\rightarrow \infty$. Therefore, by letting
$k\rightarrow \infty$ in (\ref{2.5}), we obtain that
\begin{equation}
\tilde{F}_{\theta}[\varphi]=\int_{\mathbb{R}^{d}}(\varphi(z)-\langle
D\varphi(0),z\rangle)F_{\mu}(dz)\text{, for }\varphi \in \mathfrak{F}_{0}.
\label{2.8}%
\end{equation}
Together with (\ref{2.6})-(\ref{2.7}), the conclusion follows.
\end{proof}

\subsection{Connection to PIDE}

\label{Section_connection PIDE}

In this section, we relate the nonlinear L\'{e}vy process $(\tilde{L}%
_{t})_{t\in \lbrack0,1]}$ to the fully nonlinear PIDE (\ref{PIDE}). Let
$C_{b}^{2,3}([0,1]\times \mathbb{R}^{3d})$ denote the set of functions on
$[0,1]\times \mathbb{R}^{3d}$ having bounded continuous partial derivatives up
to the second order in $t$ and third order in $x,y,z$, respectively. Now we
give the definition of viscosity solution for PIDE (\ref{PIDE}).

\begin{definition}
A bounded upper semicontinuous (resp. lower semicontinuous) function $u$ on
$[0,1]\times \mathbb{R}^{3d}$ is called a viscosity subsolution (resp.
viscosity supersolution) of (\ref{PIDE}) if $u(0,\cdot,\cdot,\cdot)\leq
\phi(\cdot,\cdot,\cdot)$ $($resp. $\geq \phi(\cdot,\cdot,\cdot))$ and for each
$(t,x,y,z)\in(0,1]\times \mathbb{R}^{3d}$,
\[%
\begin{array}
[c]{l}%
\displaystyle \partial_{t}\psi(t,x,y,z)-\sup \limits_{(F_{\mu},q,Q)\in \Theta
}\left \{  \int_{\mathbb{R}^{d}}\delta_{\lambda}\psi(t,x,y,z)F_{\mu}%
(d\lambda)\right. \\
\displaystyle \text{\  \  \  \  \  \  \  \  \  \  \  \  \  \  \ }\left.  +\langle D_{y}%
\psi(t,x,y,z),q\rangle+\frac{1}{2}tr[D_{x}^{2}\psi(t,x,y,z)Q]\text{ }\right \}
\leq0\text{ }(\text{resp. }\geq0)
\end{array}
\]
whenever $\psi \in C_{b}^{2,3}((0,1]\times \mathbb{R}^{3d})$ is such that
$\psi \geq u$ (resp. $\psi \leq u$) and $\psi(t,x,y,z)=u(t,x,y,z)$. A bounded
continuous function $u$ is a viscosity solution of (\ref{PIDE}) if it is both
a viscosity subsolution and supersolution.
\end{definition}

For each $\phi \in C_{b,Lip}(\mathbb{R}^{3d})$, define
\begin{equation}
u(t,x,y,z)=\mathbb{\tilde{E}}[\phi(x+\tilde{\xi}_{t},y+\tilde{\eta}
_{t},z+\tilde{\zeta}_{t})],\text{ }(t,x,y,z)\in \lbrack0,1]\times
\mathbb{R}^{3d}. \label{u}%
\end{equation}

\begin{theorem}
\label{unique viscosity theorem}Suppose that assumptions (A1)-(A3) hold. Then,
the value function $u$ of (\ref{u}) is the unique viscosity solution of the
fully nonlinear PIDE (\ref{PIDE}), i.e.,
\begin{equation}
\left \{
\begin{array}
[c]{l}%
\displaystyle \partial_{t}u(t,x,y,z)-\sup \limits_{(F_{\mu},q,Q)\in \Theta
}\left \{  \int_{\mathbb{R}^{d}}\delta_{\lambda}u(t,x,y,z)F_{\mu}%
(d\lambda)\right. \\
\displaystyle \text{\  \  \  \  \  \  \  \  \  \  \  \  \  \  \ }\left.  +\langle
D_{y}u(t,x,y,z),q\rangle+\frac{1}{2}tr[D_{x}^{2}u(t,x,y,z)Q]\right \}  =0,\\
\displaystyle u(0,x,y,z)=\phi(x,y,z),\text{ \ }\forall(t,x,y,z)\in
\lbrack0,1]\times \mathbb{R}^{3d},
\end{array}
\right.  \label{1.0}%
\end{equation}
where $\delta_{\lambda}u(t,x,y,z)=u(t,x,y,z+\lambda)-u(t,x,y,z)-\langle
D_{z}u(t,x,y,z),\lambda \rangle$.
\end{theorem}

\begin{proof}
We first show that $u$ is continuous. It is clear that $u(t,\cdot,\cdot
,\cdot)$ is uniformly Lipschitz continuous with the same Lipschitz constant as
for $\phi$. For each $t,s\in \lbrack0,1]$ such that $t+s\leq1$, we obtain
\begin{equation}
u(t+s,x,y,z)=\mathbb{\tilde{E}}[u(t,x+\tilde{\xi}_{s},y+\tilde{\eta}%
_{s},z+\tilde{\zeta}_{s})],\text{ \ }(x,y,z)\in \mathbb{R}^{3d}, \label{DPP}%
\end{equation}
which implies the continuity of $u(\cdot,x,y,z)$:
\[
|u(t+s,x,y,z)-u(t,x,y,z)|\leq C\mathbb{\tilde{E}}[|\tilde{\xi}_{s}%
|+|\tilde{\eta}_{s}|+|\tilde{\zeta}_{s}|]\leq C\left(  \sqrt{s}+s+M_{z}%
\sqrt[\alpha]{s}\right)  .
\]

Next, we will prove that $u$ is the unique viscosity solution of (\ref{1.0}).
The uniqueness of viscosity solution can be found in Corollary 55 in
\cite{HP2021}. It suffices to prove that $u$ is a viscosity subsolution, and
the other case can be proved in a similar way. Assume that $\psi$ is a smooth
test function on $(0,1]\times \mathbb{R}^{3d}$ satisfying $\psi \geq u$ and
$\psi(\bar{t},\bar{x},\bar{y},\bar{z})=u(\bar{t},\bar{x},\bar{y},\bar{z})$ for
some point $(\bar{t},\bar{x},\bar{y},\bar{z})\in(0,1]\times \mathbb{R}^{3d}$.
For each $s\in(0,\bar{t})$, the dynamic programming principle (\ref{DPP})
shows that
\begin{align}
0  &  =\mathbb{\tilde{E}}[u(\bar{t}-s,\bar{x}+\tilde{\xi}_{s},\bar{y}%
+\tilde{\eta}_{s},\bar{z}+\tilde{\zeta}_{s})-u(\bar{t},\bar{x},\bar{y},\bar
{z})]\nonumber \\
&  \leq \mathbb{\tilde{E}}[\psi(\bar{t}-s,\bar{x}+\tilde{\xi}_{s},\bar
{y}+\tilde{\eta}_{s},\bar{z}+\tilde{\zeta}_{s})-\psi(\bar{t},\bar{x},\bar
{y},\bar{z})].\nonumber
\end{align}
We claim that
\begin{align}
\mathbb{\tilde{E}}  &  [\psi(\bar{t}-s,\bar{x}+\tilde{\xi}_{s},\bar{y}%
+\tilde{\eta}_{s},\bar{z}+\tilde{\zeta}_{s})-\psi(\bar{t},\bar{x},\bar{y}%
,\bar{z})]\label{1.1}\\
&  =-\partial_{t}\psi(\bar{t},\bar{x},\bar{y},\bar{z})s+\mathbb{\tilde{E}%
}[\psi(\bar{t},\bar{x},\bar{y},\bar{z}+\tilde{\zeta}_{s})-\psi(\bar{t},\bar
{x},\bar{y},\bar{z})\nonumber \\
&  \text{ \  \ }+\langle D_{y}\psi(\bar{t},\bar{x},\bar{y},\bar{z}),\tilde
{\eta}_{s}\rangle+\frac{1}{2}\langle D_{x}^{2}\psi(\bar{t},\bar{x},\bar
{y},\bar{z})\tilde{\xi}_{s},\tilde{\xi}_{s}\rangle]+o(s),\nonumber
\end{align}
whose proof will be given at the end of the proof. This implies that
\begin{align}
0  &  \leq-\partial_{t}\psi(\bar{t},\bar{x},\bar{y},\bar{z})+\lim
_{s\rightarrow0}\mathbb{\tilde{E}}[\psi(\bar{t},\bar{x},\bar{y},\bar{z}%
+\tilde{\zeta}_{s})-\psi(\bar{t},\bar{x},\bar{y},\bar{z})\label{1.6}\\
&  \text{ \  \ }+\langle D_{y}\psi(\bar{t},\bar{x},\bar{y},\bar{z}),\tilde
{\eta}_{s}\rangle+\frac{1}{2}\langle D_{x}^{2}\psi(\bar{t},\bar{x},\bar
{y},\bar{z})\tilde{\xi}_{s},\tilde{\xi}_{s}\rangle]s^{-1}.\nonumber
\end{align}

%\textcolor[rgb]{1.00,0.00,0.00}{Seeing that there exist a sequence $\{s_{k}\}_{k=1}\subset \mathcal{D}_{\infty }[0,1]$
%satisfying $s_{k}\downarrow s$ as $k\rightarrow \infty$ and a convergent sequence
%$\big \{(s_{k}/n_{i}^{\ast})^{\frac{1}{\alpha}}S_{n_{i}^{\ast}}^{3}\big \}_{i=1}^{\infty}$
%for each $s_{k}\in \mathcal{D}_{\infty }[0,1]$, such that
%\begin{equation*}
%\mathbb{\tilde{E}}[\psi(\bar{t},\bar{x},\bar{y},\bar{z}+\tilde{\zeta}
%_{s})]=\lim_{k\rightarrow \infty}\lim_{i\rightarrow \infty}\mathbb{\hat{E}}
%\big[\psi \big(\bar{t},\bar{x},\bar{y},\bar{z}+(s_{k}/n_{i}^{\ast})^{\frac{1}{\alpha}}S_{n_{i}^{\ast}}^{3}\big)\big].
%\end{equation*}
%In view of Theorem \ref{recursive}, we get
%\begin{align*}
%& \bigg \vert \mathbb{\tilde{E}}[\psi(\bar{t},\bar{x},\bar{y},\bar{z}+\tilde{
%\zeta}_{s})]-\psi(\bar{t},\bar{x},\bar{y},\bar{z})-s\sup \limits_{F_{\mu}\in
%\mathcal{L}}\int_{\mathbb{R}^{d}}\delta_{\lambda}\psi (\bar{t},\bar{x},\bar{
%y },\bar{z})F_{\mu}(d\lambda)\bigg \vert \\
%& \leq \lim_{k\rightarrow \infty}\bigg \vert \lim_{i\rightarrow \infty }
%\mathbb{\hat{E}}\big[\psi \big(\bar{t},\bar{x},\bar{y},\bar{z}+(s_{k}/n_{i}^{\ast})^{\frac{1}{\alpha}}
%S_{n_{i}^{\ast}}^{3}\big)\big]-\psi(\bar{t},\bar {x},\bar{y}, \bar{z}) \\
%&\text{\  \  \  \ } -s_{k}\sup \limits_{F_{\mu}\in \mathcal{L}}\int _{\mathbb{R}
%^{d}}\delta_{\lambda}\psi(\bar{t},\bar{x},\bar{y},\bar{z})F_{\mu }(d\lambda)
%\bigg \vert =o(s).
%\end{align*}}
In turn, Theorem \ref{represent theorem} yields that there exists an
uncertainty set $\Theta \subset \mathcal{L}\times \mathbb{R}^{d}\times
\mathbb{S}_{+}(d)$ satisfying
\[
\sup_{(F_{\mu},q,Q)\in \Theta}\left \{  \int_{\mathbb{R}^{d}}|z|\wedge
|z|^{2}F_{\mu}(dz)+|q|+|Q|\right \}  <\infty
\]
such that
\begin{align}
&  \lim_{s\rightarrow0}\mathbb{\tilde{E}}\big[\psi(\bar{t},\bar{x},\bar
{y},\bar{z} +\tilde{\zeta}_{s})-\psi(\bar{t},\bar{x},\bar{y},\bar{z})+\langle
D_{y}\psi( \bar{t},\bar{x},\bar{y},\bar{z}),\tilde{\eta}_{s}\rangle+\frac{1}{
2}\langle D_{x}^{2}\psi(\bar{t},\bar{x},\bar{y},\bar{z})\tilde{\xi}_{s},
\tilde{\xi} _{s}\rangle \big]s^{-1}\label{1.5}\\
&  =\sup_{(F_{\mu},q,Q)\in \Theta}\left \{  \int_{\mathbb{R}^{d}}\delta
_{\lambda}\psi(\bar{t},\bar{x},\bar{y},\bar{z})F_{\mu}(d\lambda)+\langle
D_{y}\psi(\bar{t},\bar{x},\bar{y},\bar{z}),q\rangle+\frac{1}{2} tr[D_{x}%
^{2}\psi(\bar{t},\bar{x},\bar{y},\bar{z})Q]\right \}  .\nonumber
\end{align}
Combining (\ref{1.6}) with (\ref{1.5}), it follows that
\begin{align*}
\displaystyle \partial_{t}\psi(\bar{t},\bar{x},\bar{y},\bar{z} )  &
-\sup_{(F_{\mu},q,Q)\in \Theta}\left \{  \int_{\mathbb{R} ^{d}}\delta_{\lambda
}\psi(\bar{t},\bar{x},\bar{y},\bar{z})F_{\mu}(d\lambda) \right. \\
\displaystyle  &  \left.  +\langle D_{y}\psi(\bar{t},\bar{x},\bar{y},\bar{z}
),q\rangle+\frac{1}{2}tr[D_{x}^{2}\psi(\bar{t},\bar{x},\bar{y},\bar{z}
)Q]\right \}  \leq0,
\end{align*}
which is the desired result.

We conclude the proof by showing (\ref{1.1}). Note that
\begin{align}
&  \psi(\bar{t}-s,\bar{x}+\tilde{\xi}_{s},\bar{y}+\tilde{\eta}_{s},\bar{z}+
\tilde{\zeta}_{s})-\psi(\bar{t},\bar{x},\bar{y},\bar{z})\nonumber \\
&  =\psi(\bar{t}-s,\bar{x}+\tilde{\xi}_{s},\bar{y}+\tilde{\eta}_{s},\bar{z}+
\tilde{\zeta}_{s})-\psi(\bar{t},\bar{x}+\tilde{\xi}_{s},\bar{y}+\tilde{\eta}
_{s},\bar{z}+\tilde{\zeta}_{s})\label{1.2}\\
&  \text{ \  \ }+\psi(\bar{t},\bar{x}+\tilde{\xi}_{s},\bar{y}+\tilde{\eta}
_{s}, \bar{z}+\tilde{\zeta}_{s})-\psi(\bar{t},\bar{x},\bar{y},\bar
{z}).\nonumber
\end{align}
Taylor's expansion yields that
\begin{equation}
\psi(\bar{t}-s,\bar{x}+\tilde{\xi}_{s},\bar{y}+\tilde{\eta}_{s},\bar{z}+
\tilde{\zeta}_{s})-\psi(\bar{t},\bar{x}+\tilde{\xi}_{s},\bar{y}+\tilde{\eta
}_{s},\bar{z}+ \tilde{\zeta}_{s})=-\partial_{t}\psi(\bar{t},\bar{x},\bar
{y},\bar{z})s+\epsilon_{1}, \label{1.3}%
\end{equation}
where%
\[
\epsilon_{1}=s\int_{0}^{1}\big[-\partial_{t}\psi(\bar{t}-\theta s,\bar
{x}+\tilde{ \xi}_{s},\bar{y}+\tilde{\eta}_{s},\bar{z}+\tilde{\zeta}%
_{s})+\partial_{t}\psi(\bar{t},\bar{x},\bar{y},\bar{z})\big]d\theta.
\]
Since $\psi$ is the smooth function and $(\tilde{\xi}_{s},\tilde{\eta}_{s})$
$\overset{d}{=}(\sqrt{s}\tilde{\xi}_{1},s\tilde{\eta}_{1})$, we obtain that
\[
\mathbb{\tilde{E}}[|\epsilon_{1}|]\leq C_{\psi}(s^{2}+\mathbb{\tilde{E}}[|
\tilde{\xi}_{s}+\tilde{\eta}_{s}|]s+\mathbb{\tilde{E}}[|\tilde{\zeta}
_{s}|]s)\leq Cs^{\frac{3}{2}}.
\]
On the other hand, using Taylor's expansion, we derive that
\begin{align}
&  \psi(\bar{t},\bar{x}+\tilde{\xi}_{s},\bar{y}+\tilde{\eta}_{s},\bar{z}+
\tilde{\zeta}_{s})-\psi(\bar{t},\bar{x},\bar{y},\bar{z})\nonumber \\
&  =\psi(\bar{t},\bar{x},\bar{y},\bar{z}+\tilde{\zeta}_{s})-\psi(\bar{t},\bar{
x},\bar{y},\bar{z})+\langle D_{x}\psi(\bar{t},\bar{x},\bar{y},\bar{z}),
\tilde{\xi}_{s}\rangle \label{1.4}\\
&  \text{ \  \ }+\langle D_{y}\psi(\bar{t},\bar{x},\bar{y},\bar{z}),\tilde{
\eta}_{s}\rangle+\frac{1}{2}\langle D_{x}^{2}\psi(\bar{t},\bar{x},\bar{y},
\bar{z})\tilde{\xi}_{s},\tilde{\xi}_{s}\rangle+\epsilon_{2},\nonumber
\end{align}
where $\epsilon_{2}:=\epsilon_{2,1}+\epsilon_{2,2}$,
\begin{align*}
\epsilon_{2,1}  &  =\psi(\bar{t},\bar{x}+\tilde{\xi}_{s},\bar{y}+\tilde{\eta}
_{s},\bar{z}+\tilde{\zeta}_{s})-\psi(\bar{t},\bar{x},\bar{y},\bar{z}+\tilde{
\zeta}_{s})\\
&  \text{ \  \ }-(\psi(\bar{t},\bar{x}+\tilde{\xi}_{s},\bar{y}+\tilde{\eta}
_{s},\bar{z})-\psi(\bar{t},\bar{x},\bar{y},\bar{z}))\\
&  =\int_{0}^{1}\big \langle D_{x}\psi(\bar{t},\bar{x}+\theta \tilde{\xi}%
_{s},\bar{ y}+\theta \tilde{\eta}_{s},\bar{z}+\tilde{\zeta}_{s})-D_{x}\psi
(\bar{t}, \bar{x}+\theta \tilde{\xi}_{s},\bar{y}+\theta \tilde{\eta}_{s},\bar
{z}), \tilde{\xi}_{s}\big \rangle d\theta \\
&  \text{\  \  \ }+\int_{0}^{1}\big \langle D_{y}\psi(\bar{t},\bar{x}%
+\theta \tilde{ \xi}_{s},\bar{y}+\theta \tilde{\eta}_{s},\bar{z}+\tilde{\zeta}
_{s})-D_{y}\psi(\bar{t},\bar{x}+\theta \tilde{\xi}_{s},\bar{y}+\theta \tilde{
\eta}_{s},\bar{z}),\tilde{\eta}_{s}\big \rangle d\theta \\
&  :=\int_{0}^{1}\big \langle \delta_{x}\psi(\tilde{\zeta}_{s};\bar{t},\bar
{x}+\theta \tilde{ \xi}_{s},\bar{y}+\theta \tilde{\eta}_{s}), \tilde{\xi}%
_{s}\big \rangle d\theta+ \int_{0}^{1}\big \langle \delta_{y}\psi(\tilde
{\zeta}_{s};\bar{t},\bar{x}+\theta \tilde{ \xi}_{s},\bar{y}+\theta \tilde{\eta
}_{s}), \tilde{\eta}_{s}\big \rangle d\theta
\end{align*}
and
\begin{align*}
\epsilon_{2,2}  &  =\psi(\bar{t},\bar{x}+\tilde{\xi}_{s},\bar{y}+\tilde{\eta}
_{s},\bar{z})-\psi(\bar{t},\bar{x},\bar{y},\bar{z})-\langle D_{x}\psi(\bar{ t
},\bar{x},\bar{y},\bar{z}),\tilde{\xi}_{s}\rangle \\
&  \text{\  \  \ }-\langle D_{y}\psi(\bar{t},\bar{x},\bar{y},\bar{z}),\tilde{
\eta}_{s}\rangle-\frac{1}{2}\langle D_{x}^{2}\psi(\bar{t},\bar{x},\bar{y},
\bar{z})\tilde{\xi}_{s},\tilde{\xi}_{s}\rangle \\
&  =\int_{0}^{1}\int_{0}^{1}\big \langle(D_{x}^{2}\psi(\bar{t},\bar{x}%
+\tau \theta \tilde{\xi}_{s},\bar{y}+\tau \theta \tilde{\eta}_{s},\bar{z}%
)-D_{x}^{2}\psi( \bar{t},\bar{x},\bar{y},\bar{z}))\tilde{\xi}_{s},\tilde{\xi
}_{s}\big \rangle \theta d\theta d\tau \\
&  \text{\  \  \ }+\int_{0}^{1}\int_{0}^{1}\big \langle D_{y}^{2}\psi(\bar
{t},\bar{ x }+\tau \theta \tilde{\xi}_{s},\bar{y}+\tau \theta \tilde{\eta}%
_{s},\bar{z} ) \tilde{\eta}_{s},\tilde{\eta}_{s}\big \rangle \theta d\theta
d\tau \\
&  \text{\  \  \ }+2\int_{0}^{1}\int_{0}^{1}\big \langle D_{xy}^{2}\psi(\bar
{t},\bar{ x}+\tau \theta \tilde{\xi}_{s},\bar{y}+\tau \theta \tilde{\eta}_{s}%
,\bar{z} ) \tilde{\eta}_{s},\tilde{\xi}_{s}\big \rangle \theta d\theta d\tau.
\end{align*}
Noting that $D\psi$ is bounded and choosing $p=\frac{2+\alpha}{2-\alpha}$ and
$q=\frac{2+\alpha}{2\alpha}$, it follows that
%\begin{align*}
%\mathbb{\tilde{E}}[|\epsilon_{2,1}|]  &  \leq \big(\mathbb{\tilde{E}}%
%[|\tilde{\xi}_{s}+\tilde{\eta}_{s}|^{p}]\big)^{\frac{1}{p}}%
%\bigg(\mathbb{\tilde{E}}\bigg[\int_{0}^{1}|\delta \psi(\tilde{\zeta}_{s}%
%;\bar{t},\bar{x}+\theta \tilde{\xi}_{s},\bar{y}+\theta \tilde{\eta}_{s}%
%)|^{q}d\theta \bigg]\bigg)^{\frac{1}{q}}\\
%&  \leq C\big(\mathbb{\tilde{E}}[|\tilde{\xi}_{s}+\tilde{\eta}_{s}%
%|^{p}]\big)^{\frac{1}{p}}\bigg(\mathbb{\tilde{E}}\bigg[\int_{0}^{1}|\delta
%\psi(\tilde{\zeta}_{s};\bar{t},\bar{x}+\theta \tilde{\xi}_{s},\bar{y}%
%+\theta \tilde{\eta}_{s})|d\theta \bigg]\bigg)^{\frac{1}{q}}\\
%&  \leq C_{1}\big(\mathbb{\tilde{E}}[|\tilde{\xi}_{s}+\tilde{\eta}_{s}%
%|^{p}]\big)^{\frac{1}{p}}\big(\mathbb{\tilde{E}}[|\tilde{\zeta}_{s}%
%|]\big)^{\frac{1}{q}}\\
%&  \leq C_{2}s^{1+\frac{2-\alpha}{4+2\alpha}}=o(s),
%\end{align*}
%where $$\delta
%\psi(z-z^{\prime};t,x,y):=D_{x}\psi(t,x,y,z)-D_{x}\psi
%(t,x,y,z^{\prime})+D_{y}\psi(t,x,y,z)-D_{y}\psi(t,x,y,z^{\prime}).$$%
\begin{align*}
\mathbb{\tilde{E}}[|\epsilon_{2,1}|]  &  \leq \big(\mathbb{\tilde{E}}[|\tilde{
\xi}_{s}|^{p}]\big)^{\frac{1}{p}}\bigg(\mathbb{\tilde{E}}\bigg[ \int_{0}%
^{1}|\delta_{x}\psi(\tilde{\zeta}_{s};\bar{t},\bar{x}+\theta \tilde{ \xi}%
_{s},\bar{y}+\theta \tilde{\eta}_{s})|^{q}d\theta \bigg]\bigg)^{\frac{1}{q}}\\
&  \text{ \  \ }+\big(\mathbb{\tilde{E}}[|\tilde{\eta}_{s}|^{p}]\big)^{\frac
{1}{p}}\bigg(\mathbb{\tilde{E}}\bigg[\int_{0}^{1}|\delta_{y}\psi(\tilde{\zeta}
_{s};\bar{t},\bar{x}+\theta \tilde{\xi}_{s},\bar{y}+\theta \tilde{\eta}
_{s})|^{q}d\theta \bigg]\bigg)^{\frac{1}{q}}\\
&  \leq C\big(\mathbb{\tilde{E}}[|\tilde{\xi}_{s}|^{p}]\big)^{\frac{1}{p}}
\bigg(\mathbb{\tilde{E}}\bigg[\int_{0}^{1}|\delta_{x}\psi(\tilde{\zeta}_{s};
\bar{t},\bar{x}+\theta \tilde{\xi}_{s},\bar{y}+\theta \tilde{\eta} _{s}%
)|d\theta \bigg]\bigg)^{\frac{1}{q}}\\
&  \text{ \  \ }+C\big(\mathbb{\tilde{E}}[|\tilde{\eta}_{s}|^{p}]\big)^{\frac{
1}{p}}\bigg(\mathbb{\tilde{E}}\bigg[\int_{0}^{1}|\delta_{y}\psi(\tilde{\zeta
}_{s};\bar{t},\bar{x}+\theta \tilde{\xi}_{s},\bar{y}+\theta \tilde{\eta}
_{s})|d\theta \bigg]\bigg)^{\frac{1}{q}}\\
&  \leq C_{1}\Big[ \big(\mathbb{\tilde{E}}[|\tilde{\xi}_{s}|^{p} ]\big)^{
\frac{1}{p}}+\big(\mathbb{\tilde{E}}[|\tilde{\eta}_{s}|^{p} ]\big)^{\frac{1}{
p}}\Big] \big(\mathbb{\tilde{E}}[|\tilde{\zeta}_{s}|]\big)^{\frac{1}{q}}
=o(s).
\end{align*}
Likewise, we have
\[
\mathbb{\tilde{E}}[|\epsilon_{2,2}|]\leq C\mathbb{\tilde{E}}\big[|\tilde{\xi}
_{s}|^{3}+|\tilde{\eta}_{s}||\tilde{\xi}_{s}|^{2}+|\tilde{\eta}_{s}|^{2}+2|
\tilde{\eta}_{s}||\tilde{\xi}_{s}|\big]=o(s).
\]
Consequently, together with (\ref{1.2})-(\ref{1.4}) and $\mathbb{\tilde{E}%
}[\tilde{\xi}_{s}]=\mathbb{\tilde{E}}[-\tilde{\xi}_{s}]=0$, we prove
(\ref{1.1}). The proof is completed.
\end{proof}

\section{An example}

\label{Section_an example}

\begin{example}
This example illustrates the rationality of the conditions (A2)-(A3). For
simplicity, we consider the case $d=2$. Given $\underline{\Lambda}%
,\overline{\Lambda}>0$. Let $F_{\mu}$ be the L\'{e}vy measure given in Remark
\ref{Remark} with $\mu$ concentrated on the points $S_{0}%
=\{(1,0),(-1,0),(0,1),(0,-1)\}$,
\[
\mathcal{L}_{0}=\left \{  F_{\mu}\  \text{measure on }\mathbb{R}^{2}:\mu
(\beta)\in(\underline{\Lambda},\overline{\Lambda}),\  \text{for }\beta \in
S_{0}\right \}  ,
\]
and $\mathcal{L\subset L}_{0}$ be a nonempty compact convex set. Denote
$K=(\underline{\Lambda},\overline{\Lambda})^{4}$ and
\[%
\begin{array}
[c]{llll}%
k_{1}^{1}=\mu((-1,0)), & k_{2}^{1}=\mu((1,0)), & k_{1}^{2}=\mu((0,-1)), &
k_{2}^{2}=\mu((0,1)).
\end{array}
\]
%Denote
%\begin{align*}
%K_{\pm}=&\left \{ (k_{+}^{1},k_{-}^{1},k_{+}^{2},k_{-}^{2}):F_{\mu}\in
%\mathcal{L}\text{, s.t. }k_{+}^{1}=\mu((1,0)),k_{-}^{1}=\mu((-1,0)),\right.
%\\
%&\left.k_{+} ^{2}=\mu((0,1)),k_{-}^{2}=\mu((0,-1))\right \}.
%\end{align*}

For each $(k_{1}^{1},k_{2}^{1},k_{1}^{2},k_{2}^{2})\in K$, let $W_{k^{i}}$,
$i=1,2$, be two classical random variables such that for $i=1,2$,

\begin{description}
\item[(i)] $W_{k^{i}}$ has mean zero;

\item[(ii)] $W_{k^{i}}$ has a cumulative distribution function
\[
F_{W_{k^{i}}}(x)=\left \{
\begin{array}
[c]{ll}%
\displaystyle \left[  k_{1}^{i}/\alpha+\beta_{1}^{i}(x)\right]  \frac
{1}{|x|^{\alpha}}, & x<0,\\
\displaystyle1-\left[  k_{2}^{i}/\alpha+\beta_{2}^{i}(x)\right]  \frac
{1}{x^{\alpha}}, & x>0,
\end{array}
\right.
\]
with some continuously differentiable functions $\beta_{1}^{i}:$ $(-\infty,0]$
$\rightarrow \mathbb{R}$ and $\beta_{2}^{i}:[0,\infty)\rightarrow \mathbb{R}$
satisfying
\[
\lim_{x\rightarrow-\infty}\beta_{1}^{i}(x)=\lim_{x\rightarrow \infty}\beta
_{2}^{i}(x)=0.
\]

\item[(iii)] There exists a non-negative function $f$ on $\mathbb{N}$ tending
to 0 as $n\rightarrow \infty$, such that the following quantities are less than
$f(n)$\ for all $n$:
\[%
\begin{array}
[c]{lll}%
\displaystyle|\beta_{1}^{i}(-n^{1/\alpha})|,\text{ \  \ } & \displaystyle \int
_{-\infty}^{-1}\frac{|\beta_{1}^{i}(n^{1/\alpha}x)|}{|x|^{\alpha}}dx,\text{
\ } & \displaystyle \int_{-1}^{0}\frac{|\beta_{1}^{i}(n^{1/\alpha}%
x)|}{|x|^{\alpha-1}}dx,\\
&  & \\
\displaystyle|\beta_{2}^{i}(n^{1/\alpha})|,\text{ \ } & \displaystyle \int
_{1}^{\infty}\frac{|\beta_{2}^{i}(n^{1/\alpha}x)|}{x^{\alpha}}dx,\text{ \ } &
\displaystyle \int_{0}^{1}\frac{|\beta_{2}^{i}(n^{1/\alpha}x)|}{x^{\alpha-1}%
}dx.
\end{array}
\]

\end{description}

Denote $\tilde{\Omega}=\mathbb{R}^{2}$, $\mathcal{\tilde{H}}=C_{Lip}%
(\mathbb{R}^{2})$. For each $\chi=\varphi(x_{1},x_{2})\in \mathcal{\tilde{H}}$,
define the sublinear expectation
\[
\mathbb{\hat{E}}_{1}[\chi]=\sup_{(k_{1}^{1},k_{2}^{1},k_{1}^{2},k_{2}^{2})\in
K}\int_{\mathbb{R}^{2}}\varphi(x_{1},x_{2})dF_{W_{k^{1}}}(x_{1})dF_{W_{k^{2}}%
}(x_{2}).
\]
We consider an $\mathbb{R}^{2}$-valued random variable%
\[
Z(\tilde{\omega})=(Z^{1},Z^{2})(\tilde{\omega})=\tilde{\omega},\text{\  \ }%
\tilde{\omega}=(x_{1},x_{2})\in \tilde{\Omega}.
\]
Clearly, $\mathbb{\hat{E}}_{1}[Z^{i}]=\mathbb{\hat{E}}_{1}[-Z^{i}]=0$ and
$\mathbb{\hat{E}}_{1}[|Z^{i}|]<\infty$, for $i=1,2$. Construct a product space
(cf. \cite{P2010})
\[
\big(\Omega,\mathcal{H},\mathbb{\hat{E}}\big):=\big(\tilde{\Omega}%
^{\mathbb{N}},\tilde{\mathcal{H}}^{\otimes^{\mathbb{N}}},\mathbb{\hat{E}}%
_{1}^{\otimes^{\mathbb{N}}}\big)
\]
and introduce $Z_{i}(\omega):=Z(\tilde{\omega}_{i})$, for $\omega
=(\tilde{\omega}_{1},\tilde{\omega}_{2,}\cdots)\in \Omega$, $i=1,2,\cdots$.
Then $\{Z_{i}\}_{i=1}^{\infty}$ be a sequence of i.i.d. $\mathbb{R}^{2}%
$-valued random variables on $\big(\Omega,\mathcal{H},\mathbb{\hat{E}}\big)$
in the sense that $Z_{i+1}\overset{d}{=}Z_{i}$, and $Z_{i+1}\perp \! \! \!
\perp(Z_{1},Z_{2},\ldots,Z_{i})$ for each $i\in \mathbb{N}$.

Next, we shall adapt the method developed in \cite{HJL2021} to prove
$M_{z}<\infty$. For given $n$, define an approximation scheme $u_{n}%
:[0,1]\times \mathbb{R}^{2}\rightarrow \mathbb{R}$ recursively by
\[%
\begin{array}
[c]{l}%
u_{n}(t,z)=|z|,\text{ \ if }t\in \lbrack0,1/n),\\
u_{n}(t,z)=\mathbb{\hat{E}}[u_{n}(t-1/n,z+n^{-1/\alpha}Z)]\text{, \ if }%
t\in \lbrack1/n,1].
\end{array}
\]
Then, $u_{n}(1,0)=n^{-\frac{1}{\alpha}}\mathbb{\hat{E}}[|S_{n}^{3}|]$. By
means of Theorem 4.1 in \cite{HJL2021}, we get
\[
u_{n}(1,0)=u_{n}(1,0)-u_{n}(0,0)\leq C(1+\sqrt{f(n)})
\]
approaches $C<\infty$ as $n\rightarrow \infty$, which implies (A2) holds.

To verify (A3), we further impose the condition

\begin{description}
\item[(iv)] There exists a constant $M>0$ such that for any $(k_{1}^{1}%
,k_{2}^{1},k_{1}^{2},k_{2}^{2})\in K$, the following quantities are less than
$M$:
\[%
\begin{array}
[c]{lll}%
\displaystyle \left \vert \int_{-\infty}^{-1}\frac{\beta_{1}^{i}(x)}%
{|x|^{\alpha}}dx\right \vert ,\text{\ } &  & \displaystyle \left \vert \int
_{1}^{\infty}\frac{\beta_{2}^{i}(x)}{x^{\alpha}}dx\right \vert .
\end{array}
\]

\end{description}

Following along similar arguments as in Remark \ref{Remark}, we derive that
for $\varphi \in C_{b}^{3}(\mathbb{R}^{2})$
\begin{align*}
&  \frac{1}{s}\bigg \vert \mathbb{\hat{E}}\big[\varphi(z+s^{\frac{1}{\alpha}%
}Z_{1})-\varphi(z)\big]-s\sup_{F_{\mu}\in \mathcal{L}}\int_{\mathbb{R}^{2}%
}\delta_{\lambda}\varphi(z)F_{\mu}(d\lambda)\bigg \vert \\
&  \leq \frac{1}{s}\mathbb{\hat{E}}\bigg[\bigg \vert \mathbb{\hat{E}%
}\big[\delta_{s^{1/\alpha}Z_{1}^{1}}^{1}\varphi(z_{1},z_{2}+x_{2}%
)\big]-s\sup_{F_{\mu}\in \mathcal{L}}\int_{\mathbb{R}}\delta_{\lambda_{1}}%
^{1}\varphi(z_{1},z_{2}+x_{2})F_{\mu}^{1}(d\lambda_{1})\bigg \vert_{x_{2}%
=s^{1/\alpha}Z_{1}^{2}}\bigg]\\
&  \text{ \  \ }+Cs^{\frac{1}{\alpha}}\mathbb{\hat{E}}[|Z_{1}^{2}|]+\frac{1}%
{s}\bigg \vert \mathbb{\hat{E}}\big[\delta_{s^{1/\alpha}Z_{1}^{2}}^{2}%
\varphi(z_{1},z_{2})\big]-s\sup_{F_{\mu}\in \mathcal{L}}\int_{\mathbb{R}}%
\delta_{\lambda_{2}}^{2}\varphi(z_{1},z_{2})F_{\mu}^{2}(d\lambda
_{2})\bigg \vert \rightarrow0
\end{align*}
uniformly on $z\in \mathbb{R}^{2}$ as $s\rightarrow0$.
\end{example}

\section{Appendix}

\subsection{Proof of Theorem \ref{tight theorem}}

Since sublinear expectations $\{ \mathbb{\hat{E}}_{\alpha}\}_{\alpha
\in \mathcal{A}}$ are tight on $(\mathbb{R}^{n},C_{b,Lip}(\mathbb{R}^{n}))$,
then there exists a tight sublinear expectation $\mathbb{\hat{E}}$ on
$(\mathbb{R}^{n},C_{b,Lip}(\mathbb{R}^{n}))$ such that
\[
\mathbb{\hat{E}}_{\alpha}[\varphi]-\mathbb{\hat{E}}_{\alpha}[\varphi^{\prime
}]\leq \mathbb{\hat{E}}[\varphi-\varphi^{\prime}],\text{ for each }%
\varphi,\varphi^{\prime}\in C_{b,Lip}(\mathbb{R}^{n}).
\]
Let $\{N_{i}\}_{i=1}^{\infty}$ be a sequence of strictly increasing positive
integers satisfying for each $i\in \mathbb{N}$ there exists a $\phi_{i}\in
C_{b,Lip}(\mathbb{R}^{n})$ with $\mathbbm{1}_{\{|x|\geq N_{i}\}}\leq \phi_{i}$
such that $\mathbb{\hat{E}}[\phi_{i}]\leq1/i$. Denote $K_{i}:=\{|x|\leq
N_{i}\}$ for $i\in \mathbb{N}$. Let $\{ \varphi_{j}\}_{j=1}^{\infty}$
constitute a linear subspace of $C_{b,Lip}(\mathbb{R}^{n})$ such that for each
$K_{i}$, $\{ \varphi_{j}(x)\}_{j=1}^{\infty}|_{x\in K_{i}}$ is dense in
$C_{b,Lip}(K_{i})$.

We first claim that there exists a subsequence $\{ \mathbb{\hat{E}}%
_{\alpha_{n_{i}}}\}_{i=1}^{\infty}$ such that, for each $j\in \mathbb{N}$, $\{
\mathbb{\hat{E}}_{\alpha_{n_{i}}}[\varphi_{j}]\}_{i=1}^{\infty}$ is a Cauchy
sequence. Indeed, note that $\{ \mathbb{\hat{E}}_{\alpha_{n}}[\varphi
_{1}]\}_{n=1}^{\infty}$ is a bounded sequence, then there exists a subsequence
$\{ \mathbb{\hat{E}}_{\alpha_{n_{1}(i)}}\}_{i=1}^{\infty}\subset \{
\mathbb{\hat{E}}_{\alpha_{n}}\}_{n=1}^{\infty}$ such that $\{ \mathbb{\hat{E}%
}_{\alpha_{n_{1}(i)}}[\varphi_{1}]\}_{i=1}^{\infty}$ is a Cauchy sequence.
Similar procedure applies to $\{ \mathbb{\hat{E}}_{\alpha_{n_{1}(i)}}%
[\varphi_{2}]\}_{i=1}^{\infty}$, we find a subsequence $\{ \mathbb{\hat{E}%
}_{\alpha_{n_{2}(i)}}\}_{i=1}^{\infty}\subset \{ \mathbb{\hat{E}}%
_{\alpha_{n_{1}(i)}}\}_{i=1}^{\infty}$ such that $\{ \mathbb{\hat{E}}%
_{\alpha_{n_{2}(i)}}[\varphi_{2}]\}_{i=1}^{\infty}$ is a Cauchy sequence.
Repeating this process, we find that, for each $j\in \mathbb{N}$, a subsequence
$\{ \mathbb{\hat{E}}_{\alpha_{n_{j}(i)}}\}_{i=1}^{\infty}\subset \{
\mathbb{\hat{E}}_{\alpha_{n_{j-1}(i)}}\}_{i=1}^{\infty}$ such that $\{
\mathbb{\hat{E}}_{\alpha_{n_{j}(i)}}[\varphi_{j}]\}_{i=1}^{\infty}$ is a
Cauchy sequence. Taking the diagonal sequence $\{ \mathbb{\hat{E}}%
_{\alpha_{n_{i}(i)}}\}_{i=1}^{\infty}$, then for each $j\in \mathbb{N}$, $\{
\mathbb{\hat{E}}_{\alpha_{n_{i}(i)}}[\varphi_{j}]\}_{i=1}^{\infty}$ is a
Cauchy sequence. Thus, the claim follows by letting $n_{i}=n_{i}(i)$.

It remains prove that for each $\varphi \in C_{b,Lip}(\mathbb{R}^{n})$, $\{
\mathbb{\hat{E}}_{\alpha_{n_{i}}}[\varphi]\}_{i=1}^{\infty}$ is also a Cauchy
sequence. Denote $M:=\sup_{x\in \mathbb{R}^{n}}|\varphi(x)|$. For each
$\varepsilon>0$, we choose some $m>16M/\varepsilon$ such
that\ $\mathbbm{1}_{\{|x|\geq N_{m}\}}\leq \phi_{m}$ and $\mathbb{\hat{E}}%
[\phi_{m}]\leq \frac{\varepsilon}{16M}$. Let $\{ \varphi_{j_{l}}\}_{l=1}%
^{\infty}$ be a subsequence of $\{ \varphi_{j}\}_{j=1}^{\infty}$ such that
$\sup \limits_{l\in \mathbb{N}}\sup \limits_{x\in \mathbb{R}^{n}}|\varphi_{j_{l}%
}(x)|\leq M$ and
\[
\lim \limits_{l\rightarrow \infty}\sup \limits_{x\in K_{m}}|\varphi_{j_{l}%
}(x)-\varphi(x)|=0,
\]
Thus there exists some large integer $l_{0}$ such that $\sup \limits_{x\in
K_{m}}|\varphi_{j_{l_{0}}}(x)-\varphi(x)|\leq \varepsilon/8$, which implies
that for $x\in \mathbb{R}^{n}$
\begin{align*}
|\varphi_{j_{l_{0}}}(x)-\varphi(x)|  &  =|\varphi_{j_{l_{0}}}(x)-\varphi
(x)|\mathbbm{1}_{K_{m}}(x)+|\varphi_{j_{l_{0}}}(x)-\varphi
(x)|\mathbbm{1}_{K_{m}^{c}}(x)\\
&  \leq \sup \limits_{y\in K_{m}}|\varphi_{j_{l_{0}}}(y)-\varphi
(y)|\mathbbm{1}_{K_{m}}(x)+2M\mathbbm{1}_{\{|x|\geq N_{m}\}}\\
&  \leq \varepsilon/8+2M\phi_{m}(x).
\end{align*}
Then one easily gets $\mathbb{\hat{E}}\big[|\varphi_{j_{l_{0}}}-\varphi
|\big]\leq \varepsilon/4$. For this $l_{0}$, we know that there exist a large
integer $i_{0}$, such that $\big|\mathbb{\hat{E}}_{\alpha_{n_{i}}}%
[\varphi_{j_{l_{0}}}] -\mathbb{\hat{E}}_{\alpha_{n_{j}}}[\varphi_{j_{l_{0}}%
}]\big|\leq \varepsilon/2$ for any $i,j\geq i_{0}$. Then, it follows that
\[
\big|\mathbb{\hat{E}}_{\alpha_{n_{i}}}[\varphi]-\mathbb{\hat{E}}
_{\alpha_{n_{j}}}[\varphi]\big|\leq2\mathbb{\hat{E}}\big[|\varphi
-\varphi_{j_{l_{0}}}|\big]+\big|\mathbb{\hat{E}}_{\alpha_{n_{i}}}
[\varphi_{j_{l_{0}}}]-\mathbb{\hat{E}}_{\alpha_{n_{j}}}[\varphi_{j_{l_{0}}
}]\big|\leq \varepsilon.
\]
Therefore, we conclude that $\{ \mathbb{\hat{E}}_{\alpha_{n_{i}}}
[\varphi]\}_{i=1}^{\infty}$ is a Cauchy sequence. This completes the proof.

\subsection{The construction of nonlinear L\'evy process $(\tilde{L}%
_{t})_{t\in \lbrack0,1]}$}

In the following, we will construct a sublinear expectation $\mathbb{\tilde{
E}}:Lip(\tilde{\Omega})\rightarrow \mathbb{R}$ such that the canonical process
$(\tilde{L}_{t})_{t\in[0,1]}$ is a nonlinear L\'evy process.
%The main idea is based on Theorem 3.3 in \cite{GHJL2017}.
We divide it into the following three steps.

Step 1. For each $m\geq0$, let $\tau_{m}=2^{-m}$,%
\[%
\begin{array}
[c]{l}%
\displaystyle \mathcal{H}^{m}=\big \{ \varphi \big(\tilde{L}_{\tau_{m}}%
,\tilde{L}_{2\tau_{m}}-\tilde{L}_{\tau_{m}},\ldots,\tilde{L}_{2^{m}\tau_{m}%
}-\tilde{L}_{(2^{m}-1)\tau_{m}}\big):\forall \varphi \in C_{b,Lip}%
\big(\mathbb{R}^{2^{m}\times3d}\big)\big \} \text{, \ for }m\geq1,\\
\displaystyle \mathcal{H}^{0}=\big \{ \phi \big(\tilde{L}_{1}\big):\forall
\phi \in C_{b,Lip}\big(\mathbb{R}^{3d}\big)\big \}.
\end{array}
\]
Let $\{L_{\tau_{m}}^{n}\}_{n=1}^{\infty}$ be a sequence of i.i.d.
$\mathbb{R}^{3d}$-valued random variables defined on $(\Omega,\mathcal{H}%
,\mathbb{\hat{E}}_{1})$ in the sense that $L_{\tau_{m}}^{1}\overset{d}%
{=}L_{\tau_{m}}$, $L_{\tau_{m}}^{n+1}\overset{d}{=}L_{\tau_{m}}^{n}$ and
$L_{\tau_{m}}^{n+1}\perp \! \! \! \perp(L_{\tau_{m}}^{1},L_{\tau_{m}}%
^{2},\ldots,L_{\tau_{m}}^{n})$ for each $n\in \mathbb{N}$. For given $m\geq1$,
$\phi \big(\tilde{L}_{n\tau_{m}}-\tilde{L}_{(n-1)\tau_{m}}\big)$ with $1\leq
n\leq2^{m}$, and $\phi \in C_{b,Lip}(\mathbb{R}^{3d})$, define
\[
\mathbb{\tilde{E}}^{m}\big[\phi \big(\tilde{L}_{n\tau_{m}}-\tilde{L}%
_{(n-1)\tau_{m}}\big)\big]=\mathbb{\hat{E}}_{1}\big[\phi(L_{\tau_{m}}%
^{n})\big].
\]
For $\varphi \big(\tilde{L}_{\tau_{m}},\ldots,\tilde{L}_{2^{m}\tau_{m}}%
-\tilde{L}_{(2^{m}-1)\tau_{m}}\big)\in \mathcal{H}^{m}$, for some $\varphi \in
C_{b,Lip}\big(\mathbb{R}^{2^{m}\times3d}\big)$, define
\[
\mathbb{\tilde{E}}^{m}\big[\varphi \big(\tilde{L}_{\tau_{m}},\ldots,\tilde
{L}_{2^{m}\tau_{m}}-\tilde{L}_{(2^{m}-1)\tau_{m}}\big)\big]=\varphi_{0},
\]
where $\varphi_{0}$ is defined iteratively through
\[%
\begin{array}
[c]{c}%
\displaystyle \varphi_{2^{m}-1}(x_{1},x_{2},\ldots,x_{2^{m}-1})=\mathbb{\tilde
{E}}^{m}\big[\varphi \big(x_{1},x_{2},\ldots,x_{2^{m}-1},\tilde{L}_{2^{m}%
\tau_{m}}-\tilde{L}_{(2^{m}-1)\tau_{m}}\big)\big]\\
\displaystyle \varphi_{2^{m}-2}(x_{1},x_{2},\ldots,x_{2^{m}-2})=\mathbb{\tilde
{E}}^{m}\big[\varphi_{2^{m}-1}\big(x_{1},x_{2},\ldots,x_{2^{m}-2},\tilde
{L}_{(2^{m}-1)\tau_{m}}-\tilde{L}_{(2^{m}-2)\tau_{m}}\big)\big]\\
\displaystyle \vdots \\
\displaystyle \varphi_{1}(x_{1})=\mathbb{\tilde{E}}^{m}\big[\varphi
_{2}\big(x_{1},\tilde{L}_{2\tau_{m}}-\tilde{L}_{\tau_{m}}\big)\big]\\
\displaystyle \varphi_{0}=\mathbb{\tilde{E}}^{m}\big[\varphi_{1}\big(\tilde
{L}_{\tau_{m}}\big)\big].
\end{array}
\]
For $\phi \in C_{b,Lip}(\mathbb{R}^{3d})$, we also define
\[
\mathbb{\tilde{E}}^{0}\big[\phi \big(\tilde{L}_{1}\big)\big]=\mathbb{\hat{E}%
}_{1}\big[\phi(L_{1})\big].
\]
From the above definition we know that $(\tilde{\Omega},\mathcal{H}%
^{m},\mathbb{\tilde{E}}^{m})$ is a sublinear expectation space under which
$\tilde{L}_{t}-\tilde{L}_{s}\overset{d}{=}$ $\tilde{L}_{t-s}$ and $\tilde
{L}_{t}-\tilde{L}_{s}\perp \! \! \! \perp(\tilde{L}_{t_{1}},\ldots,\tilde
{L}_{t_{i}})$ for each $t_{i},s,t\in \mathcal{D}_{m}[0,1]:=\{l2^{-m}:0\leq
l\leq2^{m},l\in \mathbb{N}\}$ with $t_{i}\leq s\leq t$. Also, it can be checked
that $\mathbb{\tilde{E}}^{m}[\cdot]$ is consistent, i.e., for each $m\geq0$,
$\mathbb{\tilde{E}}^{m+1}[\cdot]=\mathbb{\tilde{E}}^{m}[\cdot]$ on
$\mathcal{H}^{m}$.

Step 2. Note that $\mathcal{H}^{m}\subset \mathcal{H}^{m+1}$,\ for each
$m\geq0$. Denote
\[
\mathcal{H}^{\infty}=\bigcup_{m\geq0}\mathcal{H}^{m}.
\]
Obviously, $\mathcal{H}^{\infty}\subset Lip(\tilde{\Omega})$ such that if
$\chi_{1},\cdots,\chi_{i}\in \mathcal{H}^{\infty}$, then $\varphi(\chi
_{1},\cdots,\chi_{i})\in \mathcal{H}^{\infty}$ for $\varphi \in C_{b,Lip}%
(\mathbb{R}^{i})$. For any $\chi \in \mathcal{H}^{\infty}$, there exists an
$m_{0}\in \mathbb{N}$ such that $\chi \in \mathcal{H}^{m_{0}}$, define
\[
\mathbb{\tilde{E}}[\chi]:=\mathbb{\tilde{E}}^{m_{0}}[\chi].
\]
Since $\mathbb{\tilde{E}}^{m}[\cdot]$\ is consistent, $\mathbb{\tilde{E}%
}:\mathcal{H}^{\infty}\rightarrow \mathbb{R}$ is a well-defined sublinear
expectation.

Step 3. We extend the sublinear expectation $\mathbb{\tilde{E}}:\mathcal{H}
^{\infty}\rightarrow \mathbb{R}$ to $\mathbb{\tilde{E}}:Lip(\tilde{\Omega
})\rightarrow \mathbb{R}$ and still use $\mathbb{\tilde{E}}$ for simplicity.
Denote $\mathcal{D}_{\infty}[0,1]:=\cup_{m\geq0}\mathcal{D}_{m}[0,1]$. For
each $\varphi(\tilde{L}_{t_{1}},\ldots,\tilde{L}_{t_{n}}-\tilde{L}_{t_{n-1}
})\in Lip(\tilde{\Omega})$ with $\varphi \in C_{b,Lip}(\mathbb{R}^{n\times3d}
)$, for each $t_{k}\in \lbrack0,1]$, $1\leq k\leq n$, we choose a sequence
$\{t_{k}^{i}\}_{i=1}^{\infty}\in \mathcal{D}_{\infty}[0,1]$ such that
$t_{k}^{i}<t_{k+1}^{i}$ and $t_{k}^{i}\downarrow t_{k}$ as $i\rightarrow
\infty$. Define
\[
\mathbb{\tilde{E}}\big[\varphi \big(\tilde{L}_{t_{1}},\ldots,\tilde{L}_{t_{n}%
}-\tilde{L}_{t_{n-1}}\big)\big]=\lim_{i\rightarrow \infty}\mathbb{\tilde{E}%
}\big[\varphi \big(\tilde{L}_{t_{1}^{i}},\ldots,\tilde{L}_{t_{n}^{i}}-\tilde
{L}_{t_{n-1}^{i}}\big)\big].
\]
It can be verified that the limit does not depend on the choice of
$\{t_{k}^{i}\}_{i=1}^{\infty}$. Indeed, for two descending sequences
$\{t_{k}^{i}\}_{i=1}^{\infty}$ and $\{t_{k}^{i^{\prime}}\}_{i^{\prime}%
=1}^{\infty}$ with the same limit $t_{k}$ as $i,i^{\prime}\rightarrow \infty$,
we assume that $t_{k}^{i}-t_{k}^{i^{\prime}}:=l_{k}\tau_{m_{k}}$ for some
$m_{k}\in \mathbb{N}$ and $1\leq l_{k}\leq2^{m_{k}}$. From the construction of
$\tilde{L}_{\cdot}$ and Remark \ref{ramark i.i.d.}, there exists a convergent
sequence $\{ \bar{S}_{n_{j}^{\ast}}^{\tau_{m_{k}}}\}_{j=1}^{\infty}$ such that
for $\varphi \in C_{b,Lip}(\mathbb{R}^{n\times3d})$
\[%
\begin{array}
[c]{l}%
\displaystyle \left \vert \mathbb{\tilde{E}}\big[\varphi \big(\tilde{L}
_{t_{1}^{i}},\ldots,\tilde{L}_{t_{n}^{i}}-\tilde{L}_{t_{n-1}^{i}
}\big)\big]-\mathbb{\tilde{E}}\big[\varphi \big(\tilde{L}_{t_{1}^{i^{\prime}}
},\ldots,\tilde{L}_{t_{n}^{i^{\prime}}}-\tilde{L}_{t_{n-1}^{i^{\prime}}
}\big)\big]\right \vert \\
\displaystyle \leq L_{\varphi}\mathbb{\tilde{E}}\big[\big(\sum_{k=1}
^{n}|\tilde{L}_{t_{k}^{i}}-\tilde{L}_{t_{k-1}^{i}}-\tilde{L}_{t_{k}
^{i^{\prime}}}+\tilde{L}_{t_{k-1}^{i^{\prime}}}|\big)\wedge2N_{\varphi}\big]\\
\displaystyle \leq2L_{\varphi}\sum_{k=1}^{n}\mathbb{\tilde{E}}\big[|\tilde
{L}_{t_{k}^{i}-t_{k}^{i^{\prime}}}|\wedge2N_{\varphi}\big]\\
\displaystyle=2L_{\varphi}\sum \limits_{k=1}^{n}\mathbb{\hat{E}}_{1}
\big[|L_{\tau_{m_{k}}}^{1}+\cdots+L_{\tau_{m_{k}}}^{l_{k}}|\wedge2N_{\varphi
}\big]\\
\displaystyle=2L_{\varphi}\sum_{k=1}^{n}\lim \limits_{j\rightarrow \infty
}\mathbb{\hat{E}}\big[|\bar{S}_{l_{k}n_{j}^{\ast}}^{t_{k}^{i}-t_{k}%
^{i^{\prime}}}|\wedge2N_{\varphi}\big]\\
\displaystyle \leq2L_{\varphi}\sum_{k=1}^{n}\lim \limits_{j\rightarrow \infty
}\mathbb{\hat{E}}\big[\big|(t_{k}^{i}-t_{k}^{i^{\prime}})^{1/2}(l_{k}
n_{j}^{\ast})^{-1/2}S_{l_{k}n_{j}^{\ast}}^{1}\big|\wedge2N_{\varphi}\big]\\
\displaystyle \text{ \  \ }+2L_{\varphi}\sum_{k=1}^{n}\lim \limits_{j\rightarrow
\infty}\mathbb{\hat{E}}\big[\big|(t_{k}^{i}-t_{k}^{i^{\prime}})(l_{k}%
n_{j}^{\ast})^{-1} S_{l_{k}n_{j}^{\ast}}^{2}\big|\wedge2N_{\varphi}\big]\\
\displaystyle \text{ \  \ }+2L_{\varphi}\sum_{k=1}^{n}\lim \limits_{j\rightarrow
\infty}\mathbb{\hat{E}}\big[\big|(t_{k}^{i}-t_{k}^{i^{\prime}})^{1/\alpha
}(l_{k}n_{j}^{\ast})^{-1/\alpha}S_{l_{k}n_{j}^{\ast}}^{3}\big|\wedge
2N_{\varphi}\big]\\
\displaystyle \leq2L_{\varphi}\sum_{k=1}^{n}\mathbb{\hat{E}}_{1}%
\big[\big(|t_{k}^{i}-t_{k}^{i^{\prime}}|^{1/2} |\xi_{1}|\big)\wedge
2N_{\varphi}\big]\\
\displaystyle \text{ \  \ }+2L_{\varphi}\sum_{k=1}^{n}\mathbb{\hat{E}}%
_{1}\big[\big(|t_{k}^{i}-t_{k}^{i^{\prime}}||\eta_{1}|\big)\wedge2N_{\varphi
}\big]+2L_{\varphi}M_{z}\sum_{k=1}^{n}|t_{k}^{i}-t_{k}^{i^{\prime}}%
|^{1/\alpha}\\
\displaystyle \rightarrow0\text{, \ as }i,i^{\prime}\rightarrow \infty,
\end{array}
\]
where $L_{\varphi}>0$ is the Lipschitz constant of $\varphi$, $K_{\varphi
}:=|\varphi|_{0}$ and $N_{\varphi}:=\frac{K_{\varphi}}{L_{\varphi}}$.
%\begin{align*}
%\bar{S}_{l_{k}n_{j}^{\ast}}^{t_{k}^{i}-t_{k}^{i^{^{\prime}}}} &
%=\Big((\tau_{m_{k}}/n_{j}^{\ast})^{1/2}(S_{l_{k}n_{j}^{\ast}}^{1}%
%-S_{(l_{k}-1)n_{j}^{\ast}}^{1}),(\tau_{m_{k}}/n_{j}^{\ast})(S_{l_{k}%
%n_{j}^{\ast}}^{2}-S_{(l_{k}-1)n_{j}^{\ast}}^{2}),(\tau_{m_{k}}/n_{j}^{\ast
%})^{1/\alpha}(S_{l_{k}n_{j}^{\ast}}^{3}-S_{(l_{k}-1)n_{j}^{\ast}}^{3})\Big)\\
%&  \text{ \  \ }+\cdots+\Big((\tau_{m_{k}}/n_{j}^{\ast})^{1/2}S_{n_{j}^{\ast}%
%}^{1},(\tau_{m_{k}}/n_{j}^{\ast})S_{n_{j}^{\ast}}^{2},(\tau_{m_{k}}%
%/n_{j}^{\ast})^{1/\alpha}S_{n_{j}^{\ast}}^{3}\Big).
%\end{align*}
%By a similar analysis as in \cite[Theorem 3.3]{GHJL2017},
Also, $\mathbb{\tilde{E}}:Lip(\tilde{\Omega})\rightarrow \mathbb{R}$ is a
well-defined sublinear expectation, that is, if $\varphi(\tilde{L}_{t_{1}%
},\ldots,$ $\tilde{L}_{t_{n}}-\tilde{L}_{t_{n-1}})=\varphi^{\prime}(\tilde
{L}_{t_{1}},\ldots,\tilde{L}_{t_{n}}-\tilde{L}_{t_{n-1}})$ with $\varphi
,\varphi^{\prime}\in C_{b,Lip}(\mathbb{R}^{n\times3d})$, then
\[
\mathbb{\tilde{E}}\big[\varphi \big(\tilde{L}_{t_{1}},\ldots,\tilde{L}_{t_{n}%
}-\tilde{L}_{t_{n-1}}\big)\big]=\mathbb{\tilde{E}}\big[\varphi^{\prime
}\big(\tilde{L}_{t_{1}},\ldots,\tilde{L}_{t_{n}}-\tilde{L}_{t_{n-1}%
}\big)\big].
\]
Moreover, for each $t_{i},s,t\in \lbrack0,1]\ $with $t_{i}\leq s\leq t$,
$\tilde{L}_{t}-\tilde{L}_{s}\overset{d}{=}$ $\tilde{L}_{t-s}$ and $\tilde
{L}_{t}-\tilde{L}_{s}\perp \! \! \! \perp(\tilde{L}_{t_{1}},\ldots,\tilde
{L}_{t_{i}})$ under $\mathbb{\tilde{E}}$. Thus, $(\tilde{\Omega}%
,Lip(\tilde{\Omega}),\mathbb{\tilde{E}})$ is a sublinear expectation space on
which the canonical process $(\tilde{L}_{t})_{t\in \lbrack0,1]}$ is a nonlinear
L\'{e}vy process.

%%%%%%%%%%%%%%%%%%%%%%%%%%%%%%%%%%%%%%%%%%%%%%%%%%%%%%%%%%%%%%%%%%%%%%%%

\end{document}